\documentclass[12pt, letterpaper, reqno]{amsart}

\usepackage[toc,page]{appendix}
\usepackage[english]{babel}
\usepackage{geometry} 
\usepackage{etex}
\usepackage{enumerate}
\usepackage[cmtip,all]{xy}                   
\usepackage{setspace}  
\usepackage{amscd}
\usepackage{amsmath}
\usepackage{amssymb}
\usepackage{amsthm}
\usepackage{mathrsfs}
\usepackage{epsf}
\usepackage{tikz}
\usetikzlibrary{shapes,arrows,shadows}
\usepackage{graphicx}
\usepackage{caption}
\usepackage{subcaption}
\usepackage[all]{xy}
\usepackage{float}
\usepackage[utf8]{inputenc}
\usepackage[T1]{fontenc}
\usepackage{enumerate}
\usepackage{mathabx}
\usepackage{hyperref}
\hypersetup{
    pdfborder = {0 0 0}
}

\theoremstyle{plain}

\input epsf.tex
\theoremstyle{plain}

\theoremstyle{definition} 
\newtheorem{definition}{Definition}
\newtheorem{theorem}{Theorem}
\newtheorem{lemma}[theorem]{Lemma}
\newtheorem{corollary}[theorem]{Corollary}
\newtheorem{proposition}[theorem]{Proposition}
\newtheorem{conjecture}[theorem]{Conjecture}

\theoremstyle{remark}
\newtheorem{remark}[theorem]{Remark}

\newcommand{\R}{\mathbf{R}}

\newcommand{\Z}{\mathbf{Z}}

\newcommand{\de}{\delta}

\newcommand{\longsquiggly}{\xymatrix{{}\ar@{~>}[r]&{}}}
\newcommand{\innerproduct}[1]{\ensuremath{\langle #1 \rangle}}
\newcommand{\norm}[1]{\left\lVert#1\right\rVert}
\newlength{\arrow}
\numberwithin{equation}{section}
\numberwithin{theorem}{section}
\numberwithin{definition}{section}

\newcommand{\burl}[1]{\textcolor{blue}{\url{#1}}}

\begin{document}

\begin{center}

  \textsc{}
\vspace{.65in}

{\Large{\textsc{{Fredholm Theory and Optimal Test Functions for Detecting Central Point Vanishing Over Families of $L$-functions}}}}

\vspace{.9in}

\large{\textsc{Jesse Freeman}}

\vspace{2.4in}

%
%
%
%
%
%
%
\thispagestyle{empty}

\end{center}



\newpage

\pagestyle{plain}
\centerline{
\textsc{{Abstract}}
}

\bigskip 

 The Riemann Zeta-Function is the most studied $L$-function -- its zeros give information about the prime numbers. We can associate $L$-functions to a wide array of objects. In general, the zeros of these $L$-functions give information about those objects. For arbitrary $L$-functions, the order of vanishing at the central point is of particular importance. For example, the Birch and Swinnerton-Dyer conjecture states that the order vanishing at the central point of an elliptic curve $L$-function is the rank of the Mordell-Weil group of that elliptic curve. 
  
  The Katz-Sarnak Density Conjecture states that this order vanishing (and other behavior) are well-modeled by random matrices drawn from the classical compact groups. In particular, the conjecture states that an average order vanishing (over a ``family'' of $L$-functions) can bounded using only a given weight function and a chosen test function $\phi$. The conjecture is known for many families when the test functions are suitably restricted.  

  It is natural to ask which test function is best for each family and for each set of natural restrictions on $\phi$. Our main result is a reduction of an otherwise infinite-dimensional optimization to a finite-dimensional optimization problem for all families and all sets of restrictions. We explicitly solve many of these optimization problems and compute the improved bound we obtain on average rank. While we do not verify the density conjecture for these new, looser restrictions, with this project, we are able to precisely quantify the benefits of such efforts with respect to average rank. Finally, we are able to show that this bound stictly improves as we increase support.  

\newpage

\section{Introduction}
\subsection{Background: $L$-functions and random matrices}
Our interest in random marices begins with the connections observed by Montgomery and Dyson \cite{Mon} in the 1970s. The two discovered that pair correlation of zeros of the Riemann Zeta-Function was identical to random matrix models that had been extensively studied in physics. More generally, the eigenvalues of random matrices drawn from the Haar measure on classical compact groups. We concentrate on low-lying zeros, i.e. zeros near the central point, over families of $L$-functions. However, other statistics, including $n$-level correlations \cite{Hej,Mon,RS}, spacings \cite{Od1,Od2}, and moments \cite{CFKRS}. (See \cite{FM,Ha} for a brief history of the subject and \cite{Con,For,KaSa1,KaSa2,KeSn1,KeSn2,KeSn3,Meh,MT-B,T}) for some articles and textbooks on the connections. 

In earlier work studying zeros of $L$-functions, most of the statistics used were insensitive to the behavior of finitely many zeros. But, the order of single zeros, especially the zero at the central point, is sometimes tantamount. The most natural example of this phenomenon is the Birch and Swinnerton-Dyer conjceture, which states that the order vanishing of an elliptic curve $L$-function at the central point equals the rank of the Mordell-Weil group of that curve. So, on the opposite end of the spectrum is the $n$-level density, which, for suitably chosen test functions, essentially reflects only the behavior of the low-lying zeros. Indeed, the main application of our results is to improved estimates on average order vanishing across families of $L$-functions. But, the pursuit of optimal test functions in this domain has other applications a well (for example, in \cite{IS} good estimates here are connected to the Landau-Siegel zero question). 

In our analysis, we concentrate on limiting behavior (as the conductor approaches infinity). This is the setting in which lower-order terms can be conclusively dealt with and for which the density conjecture has been verified in some cases (see \cite{ILS}). However, the rate of convergence to this behavior is quite slow. (See \cite{BMSW} for a nice summary of data and conjectures). We hope that with the new results on lower order terms in families (such as \cite{HKS,MMRW,Mil2,Yo1}), the results of this thesis can be extended to include these to refine estimates for finite conductors. We invite the reader to examine the introduction of \cite{FrM}, on which this introduction is based, for a more detailed discussion of the literature. 
\subsection{One Level Density}
Away from the central central, the zeros of $L$-functions seem to exhibit universal behavior, in an average sense. Near the central point, there are few zeros and thus there is no hope of averaging when examining a single $L$-function. So, we study families of $L$-functions, indexed by conductor and symmetry group. Broadly, the Katz-Sarnak philosophy \cite{KaSa1,KaSa2} posits that the behavior of a family of $L$-functions should be well-modeled by a corresponding classical compact group, with the conductor of the family tending to infinity (as the matrix size grows in the physics analogue). 

Throughout, we assume a Generalized Riemann Hypothesis, so that for an $L(s,f)$, all zeros are of the form $1/2 + i\gamma_{j;f}$, with $\gamma_{j;f}$ real. While the $n$-level density makes sense without this hypothesis, assuming GRH allows us to extend the support calculation for many of the number theory computations. We will introduce only the one-level density, as do not engage with higher level densities in this work. The one level density for $f$ with a test function $\phi$ is 
\begin{equation}
  D_{1}(f,\phi) = \sum_{j_{i}}^{} \phi(L_{f \gamma_{j_{i}; f}})
  \label{OneLevelDensityDef}
\end{equation}
where $L_{f}$ is a scaling parameter frequently related to the conductor. Given a family $\mathscr{F}$, we may consider 
\begin{equation}
  \mathscr{F}(Q) = \left\{ f \in \mathscr{F}; c_{f} \le Q \right\}
  \label{ConductorSubsetDef}
\end{equation}
where $c_{f}$ is the conductor of $f$. We assume that $\mathscr{F}$ has many independent forms relative to conductors so that $\left \vert \mathscr{F}(Q) \right \vert \to \infty$ as $Q \to \infty$. Then, the density conjecture (considered only in the one-level case) states that 
\begin{equation}
  \lim_{Q \to \infty} \frac{1}{\left \vert \mathscr{F}(Q) \right \vert} \sum_{f \in \mathscr{F}(Q)}^{} D_{1}(f,\phi) \ = \ \int_{-\infty}^{\infty} \phi(x) W(\mathscr{F})(x) \ dx
  \label{DensityConjecture}
\end{equation}
where $W(\mathscr{F})$ is a distribution depending on $\mathscr{F}$. \\

The seminal paper \cite{ILS} verifies this conjecture for holomoprhic cusp forms of weight $k$ which are newforms for level $N$, where $N$ is square-free and the test function $\phi$ is restricted by $\textrm{support}(\hat{\phi}) \subset (-2,2)$. 

\subsection{Bounding Average Rank} We now briefly describe the technical details of our main application of the Density Theorems, bounding the average order vanishing. The exposition here follows closely that of Remark $E$ in \cite{ILS}. Let
\begin{equation}
  p_{m}(Q) = \frac{1}{\left \vert \mathscr{F}(Q) \right \vert}\left\{ f \in \mathscr{F}(Q); \quad \underset{s = 1/2}{\textrm{ord}}L(s,f)  = m \right\}
  \label{PercentVanishingToM}
\end{equation}
so that 
\begin{equation}
  \sum_{m=0}^{\infty} p_{m}(Q) \ = \ 1. 
  \label{FactAboutPercentsVanishingToM}
\end{equation}
Recall that we choose $\phi \ge 0$ with $\phi(0) = 1$ and support $\hat{\phi}$ compact. So, by \eqref{DensityConjecture} and the Plancherel Theorem, which states that
\begin{equation}
  \int_{-\infty}^{\infty} \phi(x)W(\mathscr{F})(x) \ dx \ = \ \int_{-\infty}^{\infty} \widehat{\phi}(y)\widehat{W}(\mathscr{F})(y) \ dy,
  \label{Plancherel}
\end{equation}
one can derive,
\begin{equation}
  \sum_{m=1}^{\infty} mp_{m}(Q) < g + \varepsilon
  \label{AvgRankBound}
\end{equation}
for any $\varepsilon > 0$, provided $Q$ is large, where 
\begin{equation}
  g := \int_{-\infty}^{\infty} \widehat{\phi}(y)\widehat{W}(\mathscr{F})(y) \ dy
  \label{BoundAvgRank}
\end{equation}
which implies the upper bound
\begin{equation}
  p_{m}(Q) < m^{-1}(g + \varepsilon)
  \label{NaiveUpperBoundPmQ}
\end{equation}
for any $m \ge 1$. Note also that subtracting \eqref{AvgRankBound} from \eqref{FactAboutPercentsVanishingToM} gives the lower bound
\begin{equation}
  p_{0}(Q) > 1 - g - \varepsilon.
  \label{LowerBoundOnP0Q}
\end{equation}
By breaking up families of $L$-functions with respect to the parity of the functional equaiton, one can obtain better estimates. These details may be found in \cite{ILS}, also in their Remark $E$. 
\subsection{Setup of Our Problem}
\label{BackgroundSection}

Throughout this paper, as in \cite{ILS}, $\phi$ will be a Schwartz class function whose Fourier transform
\begin{equation}
  \hat{\phi}(\xi) = \int_{-\infty}^{\infty} \phi(x)e^{-2\pi i x \xi} dx
  \label{DefFourierTransform}
\end{equation}
has compact support.  \\

Let $W(x) \ge 0$ be a function on $\R$ whose Fourier transform
$\widehat{W}(\xi)$ is known in $[-2\sigma,2\sigma]$ for $\sigma > 0 $. We want to determine
\begin{equation}
\inf_{\phi}\frac{\int_{-\infty}^\infty \phi(x)W(x)dx}{\phi(0)}
  \label{Minimize}
\end{equation}
such that $\phi \ge 0$, $\phi \in L^1(\R)$, $\phi(0) > 0$, and support$(\hat{\phi}) \subset [-2\sigma,2\sigma]$. When $\phi$ satisfies these conditions, we say $\phi$ is \textit{admissible}. \\

The weight functions associated to the classical compact groups are
\begin{align}
  W(\textrm{O}) \ &= \ 1 + \frac{1}{2} \delta_{0}(x), \notag \\
  W(\textrm{SO(Even)})(x) \ &= \ 1 + \frac{\sin 2 \pi x}{2 \pi x}, \notag \\
  W(\textrm{SO(Odd)})(x) \ &= \ 1 - \frac{\sin 2 \pi x}{2 \pi x} + \delta_{0}(x), \notag \\
  W(Sp)(x) \ &= \ 1 - \frac{\sin 2 \pi x}{2 \pi x}.  
  \label{GroupWeightFunctions}
\end{align}
Above, $\delta_{0}(x)$ is the Dirac distribution at $x = 0$. 

We will examine the Fourier transforms of the density functions of the weight functions in \eqref{GroupWeightFunctions}. These are given by 
\begin{equation}
  \label{eq:1}
  \widehat{W}(\xi) = \de_0 + m(\xi)
\end{equation}

where we have 
\begin{align}
  m(\mathrm{SO(even))}(\xi) &= \frac{1}{2}\mathrm{I}_{[-1,1]}(\xi) \notag \\
  m(\mathrm{SO(odd))}(\xi) &= 1 - \frac{1}{2}\mathrm{I}_{[-1,1]}(\xi) \notag \\
  m(Sp)(\xi) &= -\frac{1}{2}\mathrm{I}_{[-1,1]}(\xi) \notag \\
  m(O)(\xi) &= \frac{1}{2} \label{mfunctions}
\end{align}
and $\mathrm{I}$ is the indicator function. \\

In Section \ref{GallagherSection}, we show
\begin{itemize}
\item $\phi(z) = |h(z)|^2$, where $h(z)$ is entire and exponential of
  type $2 \sigma$. 
\item $\hat{\phi} = (g * \check{g})(\xi)$, where $\check{g} =
  \overline{g(-\xi)}$ and support$(g) \subset [-\sigma,\sigma]$.  Here, $\ast$ denotes convolution. 
\end{itemize}

Section \ref{arbsupportsec} shows there exists a unique optimal test function for all $\sigma > 0$. Section \ref{OptCritSect} provides an optimality criterion used to find this function. This natural criterion is a condition on the function $g$ such that $g \ast \check{g} = \hat{\phi}$. In general, we do not find the optimal test functions directly; they are unwieldy and not necessary for computing the bounds on average rank. Instead, we find the optimal $g$. This is why we cut against notational convention and write $\textrm{support}(\hat{\phi}) \subset [-2 \sigma, 2 \sigma]$ -- $g$ will be supported in $[-\sigma,\sigma]$. \\

Section \ref{orthogsec} is short and finds the optimal functions for the orthogonal group for all levels of support. This problem is trivial relative to the general problem and requires none of the methods we develop in later sections. In Section \ref{smoothnesssection}, we uncover smoothness facts about the optimal $g$, crucial to our approach. Sections \ref{smallsupportsection} and \ref{-33section} find optimal test functions for all groups and all $0 \le 2 \sigma \le 3$, recovering the results of VanderKam in Appendix A of \cite{ILS} and laying out several examples of our general method. We present our method in full generality in Section \ref{findimreductionsection}, reducing the problem of finding the optimal test function for all groups and all $\sigma$ to a finite-dimensional problem which scales piecewise-linearly with $\sigma$. Section \ref{-44section} uses the general results of \ref{findimreductionsection} to compute a further family of examples; $3 \le 2 \sigma \le 4$. \\

Following the arguments of \cite{ILS}, $g$ is equal to the solution to the
integral equation 
\begin{equation}
  \label{eq:0}
  f_0(x) + \int_{-\sigma}^{\sigma}m(x-y)f_0(y) = 1.  
\end{equation}
Before finding the optimal test functions for extended support, we
prove a simple consequence of Gallagher's argument \cite{ILS}. 
\begin{proposition}
\label{invertg}
$\mathcal{F}^{-1}(g) = \overline{\mathcal{F}^{-1}(\check{g})}$.  
\end{proposition}
\begin{proof}
   We have 
  \begin{align*}
    \overline{\mathcal{F}^{-1}(\check{g})} \ &= \  \overline{\int_{-\infty}^{\infty} e^{2 \pi i \xi x} \overline{g(-\xi)}} \ d \xi \notag \\
    \ &= \ \int_{-\infty}^{\infty} e^{-2 \pi i \xi x} g(-\xi) d\xi \notag \\
    \ &= \ \int_{-\infty}^{\infty} e^{2 \pi i u x} g(u) du \tag{$u = - \xi$} \\ 
    \ &= \ \mathcal{F}^{-1}(g).
  \end{align*}

\end{proof}
\begin{corollary}\label{splitcor}
  $\phi(z) = \left \vert h(z) \right \vert^{2}$, where $h(z)$ is entire and exponential of type $\sigma$
\end{corollary}
\begin{proof}
  As $\hat{\phi} = g \ast \check{g}$, $\phi = \mathcal{F}^{-1}(g) \cdot \mathcal{F}^{-1}(\check{g})$. Let $h(z):= \mathcal{F}^{-1}(g)$. \\

  By the Paley-Wiener Theorem (Theorem \ref{PaleyWiener}), $h$ is exponential of type $\sigma$. \\

  By proposition \ref{invertg}, $\phi = \left \vert h(z) \right \vert^{2}$.  
\end{proof}

\begin{lemma}\label{evenlemma}
  Suppose that a unique solution to (\ref{eq:0}) exists. Then, it is even. 
\end{lemma}
\begin{proof}
  The key is that $m$ is even. Suppose $g(x)$ is a solution to (\ref{eq:0}). Let $r(x) = g(-x)$. Then, 
  \begin{equation}
   r(x) + \int_{-\sigma}^{\sigma} m(-x-y)g(y) dy = 1.  
    \label{hconsequence}
  \end{equation}
Rearranging the above expression and making the substitution $u = -y$, we obtain
\begin{align*}
  r(x) &= 1 - \int_{-2\sigma}^{2\sigma} m(-x-y)g(y) dy \\
  &= 1 + \int_{2\sigma}^{-2\sigma} m(-x + u) g(-u) du \\
  &= 1 - \int_{-2\sigma}^{2\sigma} m(x - u) r(u) du. 
\end{align*}
The first line is a rearrangement of (\ref{hconsequence}). The second is a result of our substitution. The final line comes from the fact that $m$ is even and $r(u) = g(-u)$. \\  
However, we have just shown that $r$ satisfies (\ref{eq:0}). By our assumption of uniqueness, $g=r$. \\

These techniques also show that if $f$ is even and $g$ is defined by 
\begin{equation}
  g(x) = r(x) - \int_{-\sigma}^{\sigma}m(x-y)f(y) dy 
  \label{BanachSmallSupportContext}
\end{equation}
where $r,m$ are even and $\sigma$ is real, then $g$ is even. 
\end{proof}

%

\newpage

\section{General Form of $\hat{\phi}$ for the $1$-level via an Argument of Gallagher} 
\label{GallagherSection}
We now prove a theorem on the general form of the Fourier transform of $\phi$. We need three results from complex analysis. We will prove that $\hat{\phi}$ is the convolution of two functions of a certain exponential type. First, we define exponential type. 

\begin{definition}
  A function $F(z)$ is said to be of exponential type $\sigma > 0$ if for every $\varepsilon > 0$, there exists a constant $A_{\varepsilon}$ such that 
  \[\left \vert F(z) \right \vert \le A_{\varepsilon}e^{(\sigma + \varepsilon)|z|}. \] 
\end{definition}
\begin{theorem}[Paley-Wiener]
  Suppose $f(x) \in L^2(\R)$. Then, $f$ is the restriction of an entire function of exponential type $A$ if and only if $\textrm{support}(\hat{f}) \subseteq (-A,A)$ and $\hat{f} \in L^{2}(-A,A)$. 
  \label{PaleyWiener}
\end{theorem}
\begin{proof}
  See \cite{R}, Theorem 19.3. 
%
\end{proof}
Now that we know $\phi$ is of exponential type, we may invoke a Theorem of Ahiezer:

\begin{theorem}[Ahiezer]
  Let $F(z)$ be entire of and of exponential type, let $F(x) \ge 0$ on $\R$, and suppose that 
  \begin{equation}
   \int_{-\infty}^{\infty} \frac{\log^{+}F(x)}{1 + x^2} \ dx < \infty 
    \label{ReasonableDecay}
  \end{equation}
  where, 
  \[\log^{+}(a) = 
    \begin{cases}
      \log(a) \ &\textrm{if} \ a \ge 1 \\
      0 \ &\textrm{if} \ a < 1.
    \end{cases}
  \] 
  Then, there is an entire function $f(z)$ of exponential type without zeros in $\textrm{Im}(z) > 0$ such that $F(z) = f(z) \overline{f(\overline{z})}$. In particular, $F(x) = \left \vert f(x) \right \vert^{2}, x \in \R$. \cite{K}
  \label{RieszFejer}
\end{theorem}
\begin{proof}
  See \cite{K}, page 55. 
\end{proof}
%
%

\begin{theorem}
  Let $\phi$ be an admissible test function for (\ref{Minimize}) with $\textrm{support}(\hat{\phi}) \subseteq (-2\sigma,2\sigma)$. Then $\hat{\phi}(\xi) = (g \ast \check{g})(\xi)$, where $\textrm{support}(g) \subseteq (-\sigma,\sigma), \ g \in L^{2}(-\sigma,\sigma)$ and
  \[ \check{g}(\xi) = \overline{g(-\xi)}. \]     
  \label{FormOfPhiHat}
\end{theorem}
\begin{proof}
  We know from Theorem \ref{PaleyWiener} that $\phi$ is the restriction, to the real line, of an entire function of type $2\sigma$. As we require $\phi \ge 0$ for $x \in \R$ and $\phi \in L^{1}(\R)$, $\phi$ must satisfy \eqref{ReasonableDecay}. Since,
  \begin{align*}
    \int_{-\infty}^{\infty} \frac{\log^{+}\phi(x)}{1 + x^2} \ dx &\le \int_{-\infty}^{\infty} \log^{+}\phi(x) \ dx \\
    &\le \int_{-\infty}^{\infty} \phi(x) \\
    &< \infty. 
  \end{align*}
  It follows from Theorem \ref{RieszFejer} that $\phi(x) = h(z)\overline{h(\overline{z})}$ for some function $h$ of exponential type. Since $\phi$ is of exponential type $2\sigma$, $h$ is of exponential type $\sigma$. By the reverse direction of Theorem \ref{PaleyWiener}, $g:= \mathcal{F}(h)$ is supported in $(-\sigma, \sigma)$ and is in $L^{2}(-\sigma,\sigma)$. Note that $\mathcal{F}(f_1f_2) = \mathcal{F}(f_1) \ast \mathcal{F}(f_2)$. By Proposition \ref{invertg}, we know  $\mathcal{F}(\overline{h(\overline{x})}) = \check{g}(\xi)$. Hence,  
  \begin{align*}
	  \mathcal{F}(h(x) \overline{h(x)}) \ &= \ \mathcal{F}(h(x) \overline{h(\overline{x})}) \\
	 \ &= \ \mathcal{F}(h(x)) \ast \mathcal{F}(\overline{h}(\overline{x})) \\
	 \ &= \ g(\xi) \ast \check{g}(\xi)
  \end{align*}
  which completes the proof.
\end{proof}

\newpage

\section{Existence and Uniqueness For Arbitrary Support}\label{arbsupportsec}
%
%

The key to this section is the Fredholm alternative - a more powerful infinite dimensional analogue of the Fundamental Theorem of Linear Algebra. 

\begin{definition}
  A bounded linear operator $A:X \to X$ on a normed space $X$ is said to satisfy the \textit{Fredholm altenative} if $A$ is such that 
  \label{DefFredAlt}
     The nonhomogenous equations
      \[Ax = y, \quad \quad A^{\times}f = g \] 
      (where $A^{\times}$ is the adjoint operator of $A$) have solutions $x$ and $f$, respectively, for every given $y \in X$ and $g \in X'$, the solutions being unique. Equivalently, the corresponding homogenous equations
      \[Ax = 0 \quad A^{\times}f = 0 \] 
      have only the trivial solutions $x = 0$ and $f = 0$, respectively. 
\end{definition}
\begin{theorem}[Fredholm Alternative]
  Let $T:X \to X$ be a compact linear operator on a normed space $X$ and let $\lambda \not = 0$. Then $T_{\lambda} = T - \lambda I$ satisfies the Fredholm alternative \cite{RS}. 
  \label{FredholmAlternative}
\end{theorem}

Given this result, we first translate the optimization problem into one involving bounded linear operators. Then we show that those operators are compact. Finally, we show that the operators are strictly positive definite. 

\begin{proposition}
  The minimization of (\ref{Minimize}) is equivalent to the minimization of 
  \begin{equation}
    R(g) = \frac{\innerproduct{(I+K_{\mathcal{G},\sigma})g,g}}{\left \vert \innerproduct{g,1} \right \vert^{2}} 
    \label{OperatorForm}
  \end{equation}
  over all $g$ in $L^{2}[-\sigma,\sigma]$ (such that the corresponding $\phi$ is admissible), where $K_{\mathcal{G},\sigma}: L^{2}[-\sigma,\sigma] \rightarrow L^{2}[-\sigma,\sigma]$ is defined by 
  \begin{equation}
    K_{\mathcal{G},\sigma}(g(x)) = \int_{-\sigma}^{\sigma}m_{\mathcal{G}}(x-y)g(y)dy 
    \label{OperatorKDefinition}
  \end{equation}
  and $m_{\mathcal{G}}$ is one of the densities given in \eqref{mfunctions}.
\end{proposition}

\begin{quote}
  \texttt{Throughout the text, the operator $K_{\mathcal{G},\sigma}$ will always depend on $\mathcal{G}$ and $\sigma$. We make note of this now and omit these subscripts in future instances, referring to it simply as $K$.}
\end{quote}

\begin{proof}
  Note that we have already shown $\mathcal{F}(\check{g}) = \overline{\mathcal{F}(g)}$. The same argument shows that $\mathcal{F}^{-1}(\check{g}) = \overline{\mathcal{F}^{-1}(g)}$. Applying the Plancherel Theorem, we have 
  \begin{align*}
	  \frac{\int_{-\infty}^{\infty} \phi(x)W(x)}{\phi(0)} \ &= \ \frac{\int_{-2\sigma}^{2\sigma} \hat{\phi}(\xi)\hat{W}(\xi) d \xi}{\mathcal{F}^{-1}(g \ast \check{g})(0)} \\
    \ &= \ \frac{\int_{-2\sigma}^{2\sigma} \hat{W}(\xi)((g \ast \check{g})(\xi))}{(\mathcal{F}^{-1}(g) \cdot \mathcal{F}^{-1}(\check{g}))(0)} \\
    \ &= \ \frac{\int_{-2\sigma}^{2\sigma}\hat{W}(\xi) \int_{-\sigma}^{\sigma} \overline{g(y - \xi)}g(y) dy \ d\xi}{\left \vert \innerproduct{g,1} \right \vert^{2}} \\
    \ &= \ \frac{\int_{-2\sigma}^{2\sigma} \hat{W}(\xi) \int_{-\sigma}^{\sigma} g(y - \xi)g(y) dy \ d\xi}{\left \vert \innerproduct{g,1} \right \vert^{2}} \tag{$g$ is real valued}\\
    \ &= \ \frac{ \int_{-\sigma}^{\sigma} (\hat{W} \ast g)(\xi) g(\xi) d\xi}{\left \vert \innerproduct{g,1} \right \vert^{2}} \\
    \ &= \frac{\innerproduct{(I+K)g,g}}{\left \vert \innerproduct{g,1} \right \vert^{2}}, 
  \end{align*}
  where the final equality holds by \eqref{OperatorKDefinition} and \eqref{eq:1}
\end{proof}

The following shows that $K$ is compact: 

\begin{theorem}
  The integral operator
  \begin{equation}
    (Tf)(x) = \int_{M}^{}R(x,y)f(y) d\mu(y)
    \label{CompactEqn}
  \end{equation}
  on $L^2(M,d\mu)$ is compact if $R(\cdot,\cdot) \in L^{2}(M \times M,d\mu \otimes d\mu)$.
  \label{IfThisThenCompact}
\end{theorem}
\begin{proof}
  See \cite{RS}, Section 6.6.
\end{proof}

\begin{corollary}
  The operator $I+K$ is a positive definite, i.e., the homogenous equation from Definition \ref{DefFredAlt} has only the trivial solution.  
  \label{HomogEqnOnlyTrivSoln}
\end{corollary}

\newpage

\section{An Optimality Criterion and an Equality for the Infimum}\label{OptCritSect}
In this section, we introduce a necessary and sufficient condition for the optimal $g$ (for each group $\mathcal{G}$ and each $\sigma$) and relate it to an equality for the infimum solely in terms of that $g_{\mathcal{G}}$. 

\begin{quote}
  \texttt{Throughout the text, the optimal $g_{\mathcal{G},\sigma}$ will always depend on $\mathcal{G}$ and $\sigma$. We make note of this now and omit these subscripts in future instances, referring to it simply as $g$.}  
\end{quote}

\begin{lemma}
	 The optimal $g$ satisfies 
  \begin{equation}
    \innerproduct{1,g} \not = 0,
    \label{NotIntegrateToZero}
  \end{equation}
  where $\innerproduct{\cdot,\cdot}$ denotes the standard $L^{2}(-\sigma,\sigma)$ inner product. 
  \label{NonzeroIntegralLemma}
\end{lemma}
  
  \begin{proof}
	  In order for $\phi$ to be admissible, we require $\phi(0) > 0$. Referring to the results of Section \ref{GallagherSection}, $\phi(x) = (h(x))^2$, for $x \in \R$, where $h = \mathcal{F}^{-1}(g)$ for $x \in \R$. Hence, $h(0) \not = 0$, but $h(0) = \innerproduct{1,g}$.  
\end{proof}

\begin{lemma}
  For each group $\mathcal{G}$ and all $\sigma > 0$, the optimal $g$ satisfies 
  \begin{equation}
    (I + K)(g) \ = \ 1,
    \label{CriterionForOptimal}
  \end{equation}
  where $1$ is the function that is identically 1 on $[-\sigma,\sigma]$. In fact, we have 
  \begin{equation}
    \inf_{f \in L^{2}[-\sigma,\sigma]} R(f) \ = \ \frac{1}{\innerproduct{1,g}},
    \label{InfComputationByg}
  \end{equation}
  where $R(f)$ is the functional in \eqref{OperatorForm} we seek to minimize and $g$ is the unique function in $L^{2}[-\sigma,\sigma]$ that satisfies \eqref{CriterionForOptimal}. 
  \label{OptimalCriterionAndComputationLemma}
\end{lemma}
\begin{proof} \cite{ILS}
 The Fredholm alternative tells us that such a $g$ exists and is unique. We show here that such a $g$ is optimal. \\ 

  Let $g$ satisfy \eqref{CriterionForOptimal}, and set $\innerproduct{1,g} = A \not = 0$. In fact, $A > 0$, since 
 	\[ 
		\frac{1}{A} \ = \ \frac{\innerproduct{(I + K)g,g}}{A^2} 
	\]
	and $I + K$ is a positive-definite operator. \\

  Let $t$ be another function in $L^{2}(-\sigma,\sigma)$ corresponding to an optimal $\phi$ (so that among other requirements, it satisfies $\innerproduct{1,t} \not = 0$). The functional $R$ is invariant under scaling, so we assume that $\innerproduct{1,t} = A$. Say $t = g + f$, so that $\innerproduct{1,f} = 0$. Then, we have 
  \begin{align}
    R(t) \ &= \ \frac{\innerproduct{(I + K)(g + f), g + f}}{A^{2}} \notag \\
    \ &= \ \frac{1}{A} + \frac{\innerproduct{f,(I + K)(f)}}{A^{2}} + 2\frac{\innerproduct{(I + K)(g),f}}{A^{2}} \notag \\
    \ &= \ \frac{1}{A} + \frac{\innerproduct{f, (I + K)(f)}}{A^{2}} + 2\frac{\innerproduct{1,f}}{A^{2}} \\
    \ &= \ \frac{1}{A} + \frac{\innerproduct{f,(I + K)(f)}}{A^2} \\ 
    \ &\ge \ \frac{1}{A}, 
    \label{InfEquation}
  \end{align}
  where the second equality holds because $I + K$ is self-adjoint and the inequality holds by positive-definiteness. Note that equality holds precisely when $f$ is identically zero.  
\end{proof}

\newpage

\section{Optimal Test Functions for the Orthogonal Group}\label{orthogsec}
\begin{proposition}
  Let $\sigma > 0$. Then the optimal test function for the weight
  function corresponding to the orthogonal group is 
\[\phi(x) = \left(\frac{\sin(2\pi \sigma x)}{(1+\sigma)\pi x}\right)^2. \] 
\end{proposition}

\begin{proof}
  Here we seek a solution to the integral equation 
  \begin{equation}
    f_0(x) + \frac{1}{2}\int_{-\sigma}^{\sigma}f_0(y)dy = 1_{[-\sigma,\sigma]}
    \label{AnotherVerisonOfIntEqn}
  \end{equation}
  As the function 
  \[ g(x) = 
    \begin{cases}
      \frac{1}{1 + \sigma} &\left \vert x \right \vert \le \sigma \\
      0 &\left \vert x \right \vert > \sigma
    \end{cases}
  \] 
  indeed solves the equation \eqref{AnotherVerisonOfIntEqn}, we may use it to compute the optimal $\phi$. We compute its Fourier inverse:  
\begin{align*}
  \mathcal{F}^{-1}(g) &=
  \frac{1}{1+\sigma}\int_{-\sigma}^{\sigma}e^{2\pi i x\xi}d\xi \\
  &=\frac{1}{1 + \sigma}\left(\frac{\sin(2\pi\sigma x)}{\pi x} \right).  
\end{align*}
We complete the proof by applying Corollary \ref{splitcor}. 
\end{proof}

\newpage 

\section{Lipschitz Continuity and Smoothness Almost Everywhere for $g$} \label{smoothnesssection}
First, we show that for an optimal $\phi$ such that $\hat{\phi} = g \ast \check{g}$, $g$ must be Lipschitz continuous. Then we show that such a function is differentiable almost everywhere, using a theorem of Rademacher. \\

We begin by proving that $g$ is bounded. 
\begin{lemma}
  The optimal $g$ as defined in Proposition \ref{CriterionForOptimal}, is bounded. 
  \label{BoundedBeforeLip}
\end{lemma}
\begin{proof}
  We will show that 
  \begin{equation}
  h(x):=\int_{-\sigma}^{\sigma} m(x - y)g(y) \ dy
    \label{FirstPartTriangle}
  \end{equation}
  is bounded. To show boundedness of $g$, we apply the triangle inequality to 
  \begin{equation}
    g(x) + h(x) = 1,
    \label{SimpleOptgCriterion}
  \end{equation}
  the defining equation for $g$.  \\
  
  We know that $g \in L^2(-\sigma,\sigma)$. By the Cauchy-Schwarz inequality, we have 
  \begin{equation}
    \int_{-\sigma}^{\sigma} m(x - y) g(y) \ dy \le \norm{g}_{L^2} \norm{m(x - y)}_{L^2}
    \label{CauchySchwarzBounded}
  \end{equation}
  We know $\norm{g}_{L^2} < \infty$. Let $m$ be one of the functions in \eqref{mfunctions} and let $x \in [-\sigma,\sigma]$. Then
  \begin{align*}
    \norm{m(x-y)}_{L^2} &= \left( \int_{-\sigma}^{\sigma} m(x - y)^2dy \right)^{1/2} \\
    &\le \left( (1)^2 (2\sigma) \right)^{1/2}, \\
  \end{align*}
  which is a bound independent of $x$. 

  Applying the triangle inequality to  \eqref{SimpleOptgCriterion} shows $g$ is bounded as well.
  \end{proof}

\begin{lemma}
  \label{LipschitzLemma}
The optimal $g$, as defined above is Lipschitz continuous.  
\end{lemma}
\begin{proof}
  Using the optimality criterion (\ref{eq:0}), we see that for $x_1,x_2 \in [-\sigma,\sigma]$, 
  \begin{equation}
  \begin{aligned}
  \left \vert g(x_1) - g(x_2) \right \vert &= \left \vert \int_{-\sigma}^{\sigma}(m(x_1 - y) - m(x_2 - y))g(y) \ dy \ \right \vert \\
    &\le \int_{-\sigma}^{\sigma} |m(x_1 - y) - m(x_2 - y)| |g(y)| \ dy \\
    &\le \max_{y \in [-\sigma,\sigma]}|g(y)| \int_{-\sigma}^{\sigma} |m(x_1 - y) - m(x_2 - y)| \ dy. 
    \label{LipschitzArg}
  \end{aligned}
  \end{equation}

  Now, we analyze (\ref{LipschitzArg}). Note that for all choices of $m$ in \eqref{mfunctions}, the integrand is bounded by 1/2. Now, examine the region of integration. Without loss of generality, assume $x_1 \ge x_2$.  \\
  
  Note that our integrand vanishes everywhere except from $\max\left\{ -\sigma,x_2 - 1 \right\}$ to $\min\{x_1 - 1$, $x_2 + 1, \sigma \}$ and again from $\max\left\{ x_2 + 1, x_1 - 1, - \sigma \right\}$ to $\min\left\{ x_1 + 1, \sigma \right\}$. The size of this region does not scale with $\sigma$, since if $-\sigma \ge x_2 - 1$ and $\sigma \le x_2 + 1$, then $\sigma \le 1$. In fact, this region has measure at most $\min\left\{ 2(x_1 - x_2), 4 \right\}$. As a result, we may revise the inequality in (\ref{LipschitzArg}): 
  \begin{equation}
	  \begin{aligned}
	\left \vert g(x_1) - g(x_2) \right \vert  & \le \max_{y \in [-\sigma,\sigma]}|g(y)| \int_{-\sigma}^{\sigma} |m(x_1 - y) - m(x_2 - y)| \ dy \\
    &\le \max_{y \in [-\sigma,\sigma]}|g(y)| (2|x_1 - x_2|) \left( \frac{1}{2} \right) \\
    &\le \max_{y \in [-\sigma,\sigma]} |g(y)| |x_1 - x_2|. 
    \label{CompletedLipschitzArg}
	  \end{aligned}
  \end{equation}
\end{proof}

\begin{remark}
	When $\sigma \le .5$, each choice of $m$ takes a uniform value on $[-\sigma,\sigma]$. From \eqref{LipschitzArg}, we can deduce that when $\sigma \le .5$, all optimal $g$ are constant functions. 
	\label{ReallySmallSigmaImpliesConstant}
\end{remark}

We now use a Theorem of Rademacher to show that our function $g$ is differentiable almost everywhere. 
\begin{theorem}[Rademacher]
  Let $\Omega \subset \R^n$ be open. If $f: \Omega \longrightarrow \R$ is Lipschitz continuous, then $f$ is differentiable almost everywhere in $\Omega.$ 
  \label{RademacherTheorem}
\end{theorem}
\begin{proof}
  See \cite{Fed}, Theorem 3.1.6. 
\end{proof}
\begin{corollary}
  The optimal $g(x)$ is differentiable almost everywhere. 
  \label{gisdiffa.e.}
\end{corollary}
Finally, we show that each such $g$ is in fact smooth almost everywhere.
\begin{lemma}
  The optimal $g(x)$, as defined above, is smooth almost everywhere.  
  \label{TwiceDiff}
\end{lemma}
\begin{proof}
  We proceed by induction. Our base case, that $g$ is once-differentiable, is established by Corollary \ref{gisdiffa.e.}. Assume that $g$ is $k$-times differentiable almost everywhere. \\

  Note that for any choice of $\mathcal{G},\sigma$, we can write the optimality criterion as 
\begin{equation}
	g(x) = 1 - \left( \alpha_{\mathcal{G}} \int_{-\sigma}^{\sigma} g(y) dy + \beta_{\mathcal{G}} \int_{\max \{x-1, - \sigma \}}^{\min\{x+1,\sigma \}} g(y) dy  
   \right) \label{GrossSimplification}
\end{equation}
where the $\alpha_{\mathcal{G}},\beta_{\mathcal{G}}$ are constants depending on $\mathcal{G}$. We also know that $g$ is continuous. For almost all $x \in [-\sigma,\sigma]$, the limits of integration are smooth functions of $x$. Therefore, the fundamental Theorem of calculus and our hypothesis that $g$ is $k$-times differentiable show that $g$ is in fact $k+1$ times differentiable.   
\end{proof}

\newpage

\section{Explicit Test Functions for Small Support}\label{smallsupportsection}

In this section, we find all optimal test functions for $\sigma \le 1$. First, we show that for $\sigma < 1$, the function we seek is the unique fixed point of a contraction mapping from $C([-\sigma,\sigma])$ to itself. For $\sigma \le .5$, one can find the optimal $g$ by the method of repeated iterations. For $.5 < \sigma < 1$, we find our solution through analyzing the integral equation in the $[-1,1]$ case. 
\subsection{Very Small Support}

Remark \ref{ReallySmallSigmaImpliesConstant} tells us that for $\sigma \le .5$, the optimal $g$ are constant. These constants are not hard to solve for. 
\begin{theorem}
  For $\sigma \le .5$, the optimal test functions for the Orthogonal, SO(even), and SO(odd) groups are given by 
  \[ \phi(x) = \left( \frac{\sin(2 \pi \sigma x)}{(1 + \sigma)\pi x} \right)^{2}. \] 
  The optimal test functions for the Symplectic group are 
  \[ \phi(x) = \left( \frac{\sin(2 \pi \sigma x)}{(1 - \sigma)\pi x} \right)^{2}. \] 
  \label{}
\end{theorem}
\begin{proof}

   The orthogonal case has been proven for all $\sigma$ in Section \ref{orthogsec}. For $\sigma \le .5$, the kernels for $\textrm{SO(even)}$ and $\textrm{SO(odd)}$ agree with the orthogonal kernel, proving the first part of the Theorem. \\
   
   We know a constant function satisfies the optimality criterion and a quick check shows $1/(1-\sigma)$ is the constant we seek. \\

We then square the Fourier inverse to find that the optimal $\phi$ for the Symplectic group, in this range of support, is given by 
\begin{equation}
  \phi(x) = \left( \frac{\sin(2\pi \sigma x)}{(1 - \sigma)\pi x} \right)^2. 
  \label{OptimalPhiSymplecticSmallSupport}
\end{equation}
\end{proof}
\subsection{The $[-1,1]$ Case}
We begin with a series of Theorems concering unique solutions to certain integral equations. We start by analyzing (\ref{eq:0}). As $g, m$ are even, in the case $\sigma = 1$, we may simplify (\ref{eq:0}) in each of the cases $\textrm{SO(even)}, \ \textrm{SO(odd)}$, and $\textrm{Sp}$. For $\textrm{SO(even)}$, (\ref{eq:0}) becomes
\begin{equation}
  \frac{1}{2}\int_{0}^{1} g(y) dy + \frac{1}{2} \int_{0}^{1 - x}g(y) dy + g(x) = 1.
  \label{SOEvenSimplification-1,1}
\end{equation}
The $\textrm{SO(odd)}$ equation becomes 
\begin{equation}
  \frac{3}{2} \int_{0}^{1}g(y) dy - \frac{1}{2} \int_{0}^{1 - x}g(y) dy + g(x) = 1.
  \label{SOOddSimplification-1,1}
\end{equation}
The Symplectic equation becomes 
\begin{equation}
  -\frac{1}{2} \int_{0}^{1}g(y) dy - \frac{1}{2} \int_{0}^{1-x}g(y) dy + g(x) = 1.
  \label{SpSimplification-1,1}
\end{equation}
All equations hold for $0 \le x \le 1$, and $g$ is then the even extension of the function defined for these values of $x$. \\

We present the general form of the three equations above as 
\begin{equation}
  \alpha_{\mathcal{G}} \int_{0}^{1} g(y) dy + \beta_{\mathcal{G}} \int_{0}^{1-x}g(y) dy + g(x) = 1, 
  \label{GeneralFormEquations}
\end{equation}
where values for $\alpha_{\mathcal{G}},\beta_{\mathcal{G}}$ are given in (\ref{SOEvenSimplification-1,1}) - (\ref{SpSimplification-1,1}). \\

In \cite{ILS}, Appendix A, based on a private communication with J. Vanderkam, the optimal test functions for the $[-1,1]$ case are found explicitly. Here, describe a (potentially different) methodology that fits into the framework of our main result. We show how one can deduce the function is first-order trigonometric. The argument hinges differentiation under the integral sign. 
\begin{theorem}[Liebniz]
  Let $f(x,y)$ be a function such that $f_{x}(x,y)$ exists and is continuous. Then 
  \begin{equation}
    \frac{d}{dx} \int_{a(x)}^{b(x)} f(x,y) dy = \int_{a(x)}^{b(x)} \partial_{x} f(x,y) dy + f(b(x),y)b'(x) - f(a(x), y)a'(x). 
    \label{LiebnizRuleEqn}
  \end{equation}
  \label{LiebnizRuleTheorem}
\end{theorem}
  We now derive the unique solutions to (\ref{SOEvenSimplification-1,1}) - (\ref{SpSimplification-1,1}). 
  \begin{theorem}
    The function 
    \begin{equation}
      g(x) = \frac{\cos\left( \frac{|x|}{2} - \frac{(\pi + 1)}{4} \right)}{\sqrt{2}\sin\left( \frac{1}{4} \right) + \sin \left( \frac{\pi + 1}{4} \right)}
      \label{SOEven-1,1Solution}
    \end{equation}
    is the unique solution to (\ref{SOEvenSimplification-1,1}), 
    \begin{equation}
      g(x) =  \frac{\cos\left( \frac{|x|}{2} + \frac{\pi - 1}{4} \right)}{3\sin\left( \frac{\pi + 1}{4} \right) - 2\sin \left( \frac{\pi - 1}{4} \right)}
      \label{SOOdd-1,1Solution}
    \end{equation}
    is the unique solution to (\ref{SOOddSimplification-1,1}), and
    \begin{equation}
      g(x) = \frac{\cos\left( \frac{|x|}{2} + \frac{\pi - 1}{4} \right)}{2\sin\left( \frac{\pi - 1}{4} \right) - \cos\left( \frac{\pi - 1}{4} \right)}
      \label{Sp-1,1Solution}
    \end{equation}
    is the unique solution to (\ref{SpSimplification-1,1}).
    \label{TheThree-1,1Solutions}
  \end{theorem}

Although our differential equations hold only almost everywhere, we are still able to establish the following. 
\begin{lemma}
	For each group, the optimal $g$ satisfies 
  \begin{equation}
	  g'(x) \ = \ \beta_{\mathcal{G}} g(1 - x). 
    \label{HelpfulLemmaEqn}
  \end{equation}
  and 
  \begin{equation}
	  g''(x) + \frac{1}{4} g(x) \ = \ 0. 
	  \label{SecondOptGEquation}
  \end{equation}
  placing the optimal $g$ in a one-parameter family, depending on the symmetry group.  
  \label{UniqueOneParameterFamily}
\end{lemma}
\begin{proof}
For $x \in (0,1)$, the optimality criterion can be written as 
\[ 
	g(x) + \beta_{\mathcal{G}}\int_{x-1}^{1} g(y) dy \ = \ 1
\]
which we differentiate under the integral sign to obtain 
\[ 
	g'(x) - \beta_{\mathcal{G}} g(x-1) \ = \ 0 
\]
which becomes \eqref{HelpfulLemmaEqn}. \\

However, we note that when $x \in (0,1)$, $1-x \in (0,1)$. Differentiating \eqref{HelpfulLemmaEqn}, we see 
 \[ 
	 g''(x) \ = \ - \beta_{\mathcal{G}} g'(1-x) \ = \ - \beta_{\mathcal{G}}^2 g(x)
 \]
 which is \eqref{SecondOptGEquation} in all cases, since $\beta_{\mathcal{G}} = \pm 1/2$. \\

  
  We are left with a standard ODE that is both Lipschitz continuous and measurable in the input, $f$. So, there is a unique absolutely continuous solution in the extended sense for our function on $A$. However, absolute continuity is no restriction, since Lemma \ref{LipschitzLemma} shows us that the optimal $g$ is in fact Lipschitz continuous. For more detail, we refer the reader to \cite{Wal}, Chapter 3, Section 10, Supplement II. 
  
  Equation \eqref{SecondOptGEquation} is a standard linear differential equation that has a two-parameter family of solutions given by  
  \begin{equation}
    c_1 \cos \left( \frac{x}{2} \right) + c_2 \sin \left( \frac{x}{2}  \right). 
    \label{TwoParameterFamily0}
  \end{equation}
  We now apply the symmetry from (\ref{HelpfulLemmaEqn}) to narrow this family down to a one-parameter family. The differential equation (\ref{HelpfulLemmaEqn}) and trigonometric angle addition formulae yield the relation
  \begin{align*}
    \frac{1}{2}\left( -c_1\sin\left( \frac{x}{2} \right) + c_2 \cos \left( \frac{x}{2} \right) \right) &=  \beta_{\mathcal{G}} \left( c_1 \cos\left( \frac{1}{2}\right) + c_{2}\sin\left( \frac{1}{2} \right)  \right) \cos\left( \frac{x}{2}\right) \\
    &+ \beta_{\mathcal{G}} \left( c_1 \sin\left( \frac{1}{2} \right) - c_{2} \cos\left( \frac{1}{2} \right) \right)\sin\left( \frac{x}{2} \right). 
    \label{NoNumber,ButTheLinAlgEqn}
  \end{align*}
  In order for the expression above to vanish, we need the coefficients on $\cos(x/2)$ and $\sin(x/2)$ to both be zero. This translates into the requirement that the vector
  $\begin{pmatrix}
    c_1 \\
    c_2
  \end{pmatrix}$ be in the nullspace of the matrix 
  \begin{equation}
      \begin{pmatrix}
	2 \beta_{\mathcal{G}} \cos(1/2) & 2\beta_{\mathcal{G}}\sin(1/2) - 1 \\
    2\beta_{\mathcal{G}}\sin(1/2) + 1 & -2\beta_{\mathcal{G}}\cos(1/2)
  \end{pmatrix}. 
  \label{NullspaceRequirement0}
\end{equation}
Note the matrix in (\ref{NullspaceRequirement0}) has determinant 
\begin{equation}
  -4\beta_{\mathcal{G}}^2(\sin^2(1/2) + \cos^2(1/2)) + 1 = 0
  \label{ZeroDeterminant}
\end{equation}
because $\beta_{\mathcal{G}} = \pm 1/2$. So, it is of rank one. In fact, (\ref{SOEven-1,1Solution}) - (\ref{Sp-1,1Solution}) are non-trivial solutions to \eqref{HelpfulLemmaEqn} and \eqref{SecondOptGEquation}. Thus, the solutions to those differential equation are among the scalar multiples of a single nonzero solution. 
\end{proof}

We are now ready to prove Theorem \ref{TheThree-1,1Solutions}. 

\begin{remark}
	The case before this was quite nice, since the optimal functions are constant. This case is also nice, in fact nicer than the case in which $.5 < \sigma < 1$. That is because whenever $0 < x < \sigma = 1$, we also have $0 < 1-x < \sigma = 1$, which is not always true for $.5 < \sigma < 1$. The nature of the kernels $m$ (from \eqref{mfunctions}) tells us that the values of $g$ at $x+1, x-1$, or $1-x$, if we get to use symmetry, affect the value of $g$  or $g'$ at $x$. Thus, whether or not $1-x \in [-\sigma,\sigma]$ requires us to break the problem into cases, depending on when $1 - x \in [-\sigma,\sigma]$ or $1 - x \not \in [-\sigma,\sigma]$.   
\end{remark}

  \begin{proof}[(Proof of Theorem \ref{TheThree-1,1Solutions})]
    Note that once we establish that the functions (\ref{SOEven-1,1Solution}) - (\ref{Sp-1,1Solution}) satisfy their respective equations we are done, as uniqueness follows from Corollary \ref{HomogEqnOnlyTrivSoln} and the Fredholm alternative. \\

    We first solve for the functions for $0 \le x \le 1$, which allows us to incorporate the simplified forms (\ref{SOEvenSimplification-1,1}) - (\ref{SpSimplification-1,1}). The functions (\ref{SOEven-1,1Solution}) - (\ref{Sp-1,1Solution}) are the even extensions of the functions we will find.  \\

	From Lemma \ref{UniqueOneParameterFamily}, we know that the one-parameter family we seek falls within the two-parameter family
	\[ 
		c_1 \cos \left( \frac{x}{2} \right) + c_2 \sin \left( \frac{x}{2} \right) 
	\]
	Without loss of generality we can write $g(x) = \cos(a_{\mathcal{G}}x + b_{\mathcal{G}})$. By Lemma \ref{UniqueOneParameterFamily}, our optimal test function is a scalar multiple of that $g$. We now compute $a_{\mathcal{G}}$ and $b_{\mathcal{G}}$.  
    \begin{align*}
	    \beta_{\mathcal{G}}\cos(-a_{\mathcal{G}}x + b_{\mathcal{G}} + a_{\mathcal{G}}) &= -a \sin(a_{\mathcal{G}}x + b_{\mathcal{G}}) \\
	    &= a\sin(-a_{\mathcal{G}}x - b_{\mathcal{G}}) \\
	    &= a\cos(-a_{\mathcal{G}}x - b_{\mathcal{G}} - \pi/2), 
    \end{align*}
    which implies that $a_{\mathcal{G}} = \beta_{\mathcal{G}}$ and that $b_{\mathcal{G}}$ satisfies
    \begin{equation}
	    b_{\mathcal{G}} + \beta_{\mathcal{G}} = -b_{\mathcal{G}} - \pi/2, 
      \label{ShiftEquation}
    \end{equation}
    and so $b_{\mathcal{G}} = -\pi/4 - \beta_{\mathcal{G}}/2$. \\

    Define $f_{\mathcal{G}}$ by 
    \begin{equation}
      f_{\mathcal{G}}(x) := 
\begin{cases}
	\cos\left( \beta_{\mathcal{G}}|x| -\left( \frac{\pi + 2\beta_{\mathcal{G}}}{4} \right) \right) \ & |x| \le 1 \\
	0 & |x| > 1.
\end{cases}
     \label{ReferenceForf_G}
    \end{equation}
    As $f_{\mathcal{G}}$ satisfies \eqref{HelpfulLemmaEqn}, plugging $f_{\mathcal{G}}$ into (\ref{GeneralFormEquations}) yields a constant, $c_{\mathcal{G}}$. In each instance, $c_{\mathcal{G}}$, which is in fact nonzero. The scaling factors in (\ref{SOEven-1,1Solution}) - (\ref{Sp-1,1Solution}) are precisely $1/c_{\mathcal{G}}$. 
  \end{proof}
\subsection{Extension To Medium-Small Support Using Integral Equation Methods}

Note that for $.5 < \sigma < 1$, (\ref{eq:0}) simplifies to 
\begin{equation}
  \frac{1}{2}\int_{0}^{\sigma} g(y) dy + \frac{1}{2} \int_{0}^{\min\{\sigma,1 - x\}}g(y) dy + g(x) = 1
  \label{SOEvenSimplificationMediumSmall}
\end{equation}
in the SO(Even) case, 
\begin{equation}
  \frac{3}{2}\int_{0}^{\sigma} g(y) dy - \frac{1}{2} \int_{0}^{\min\{\sigma,1 - x\}}g(y) dy + g(x) = 1
  \label{SOOddSimplificationMediumSmall}
\end{equation}
in the SO(Odd) case, and 
\begin{equation}
  -\frac{1}{2}\int_{0}^{\sigma} g(y) dy - \frac{1}{2} \int_{0}^{\min\{\sigma,1 - x\}}g(y) dy + g(x) = 1
  \label{SymplecticSimplificationMediumSmall}
\end{equation}
in the symplectic case. Again, we present a general form: 
\begin{equation}
  \alpha_{\mathcal{G}}\int_{0}^{\sigma}  g(y) dy + \beta_{\mathcal{G}}\int_{0}^{\min\{\sigma,1 - x\}}g(y) dy + g(x) = 1. 
  \label{GeneralFormMediumSmallSupport}
\end{equation}

\begin{theorem}
  Let $.5 < \sigma < 1$. Then the function
  \begin{equation}
    g_{\textrm{SO(Even)}}(x) = \frac{1}{\gamma_{\textrm{SO(Even)}}}
    \begin{cases}
      \cos\left( \frac{1 - \sigma}{2} - \left( \frac{\pi + 1}{4} \right) \right) &|x| \le 1 - \sigma \\
     \cos\left( \frac{|x|}{2} - \frac{(\pi + 1)}{4} \right) &1 - \sigma \le |x| \le \sigma
    \end{cases}
    \label{SOEvenMediumSmallSolution}
  \end{equation}
  is the unique solution to (\ref{SOEvenSimplificationMediumSmall}), 
  \begin{equation}
    g_{\textrm{SO(Odd)}}(x) = \frac{1}{\gamma_{\textrm{SO(Odd)}}} = 
    \begin{cases}
      \cos\left( \frac{1 - \sigma}{2} + \frac{\pi - 1}{4} \right) &|x| \le 1 - \sigma \\
      \cos\left( \frac{|x|}{2} + \frac{\pi - 1}{4} \right) &1 - \sigma \le |x| \le \sigma
    \end{cases}
    \label{SOOddMediumSmallSolution}
  \end{equation}
  is the unique solution to (\ref{SOOddSimplificationMediumSmall}), and 
  \begin{equation}
    g_{\textrm{Sp}}(x) = \frac{1}{\gamma_{\textrm{Sp}}} 
    \begin{cases}
       \cos\left( \frac{1 - \sigma}{2} + \frac{\pi - 1}{4} \right) &|x| \le 1 - \sigma \\
\cos\left( \frac{|x|}{2} + \frac{\pi - 1}{4} \right) &1 - \sigma \le |x| \le \sigma
    \end{cases}
    \label{SymplecticMediumSmallSolution}
  \end{equation}
  is the unique solution to (\ref{SymplecticSimplificationMediumSmall}). In each instance, $\gamma_{\mathcal{G}}$ is a constant that is precisely computed below (see (\ref{SOEvenFinalConstant}) - (\ref{SymplecticFinalConstant})). 
  \label{MediumSmallSupportTheorem}
\end{theorem}
\begin{proof}
  As in the proof of Theorem \ref{TheThree-1,1Solutions}, we will find an explicit form for $x \ge 0$. Then, the even extension will satisfy the corresponding equation amongst (\ref{SOEvenSimplificationMediumSmall}) - (\ref{SymplecticSimplificationMediumSmall}). \\

  Note that $\sigma \le 1-x$ if and only if $x \le 1 - \sigma$. So, each of \eqref{SOEvenSimplificationMediumSmall} - (\ref{SymplecticSimplificationMediumSmall}) has two simplifications, depending on $|x|$. The first, for $x \le 1 - \sigma$, is   
  \begin{equation}
    (\alpha_{\mathcal{G}} + \beta_{\mathcal{G}}) \int_{0}^{\sigma} g(y) dy + g(x) = 1, 
    \label{SmallXMediumSmallSupport}
  \end{equation}
  which immediately implies that our function is constant for $|x| \le 1 - \sigma$. We call this constant $C_{\mathcal{G}}$. For $x > 1 - \sigma$, $1 - x < \sigma$, so (\ref{GeneralFormMediumSmallSupport}) becomes 
  \begin{equation}
    \alpha_{\mathcal{G}}\int_{0}^{\sigma} g(y) dy + \beta_{\mathcal{G}}\int_{0}^{1 - x}g(y) dy + g(x) = 1. 
    \label{LargerXMediumSmallSupport}
  \end{equation}
  However, we know from the proof of Theorem \ref{TheThree-1,1Solutions} that $f_{\mathcal{G}}(x) = \cos\left( \beta_{\mathcal{G}} x - \left( \frac{\pi + 2\beta_{\mathcal{G}}}{4} \right)\right)$ satisfies 
  \begin{equation}
    \beta_{\mathcal{G}} \int_{0}^{1-x} f(y) dy + f(x) = \textrm{constant}. 
    \label{StartingIntuition}
  \end{equation}
  We now find the correct scaling in two steps. First, we find $C_{\mathcal{G}}$ so that the function is continuous. Then, we scale that continuous function so that $(I + K)(g) = 1$. To make the function continuous, we must have
    \begin{equation}
      C_{\mathcal{G}} = f_{\mathcal{G}}(1 - \sigma) = \cos\left( \beta_{\mathcal{G}}(1 - \sigma) - \left( \frac{\pi + 2\beta_{\mathcal{G}}}{4} \right) \right).
    \label{FirstScalingOfConstant}
  \end{equation}
  We have a piecewise function given by 
  \begin{equation}
    \widetilde{g}_{\mathcal{G}}(x) = 
    \begin{cases}
      C_{\mathcal{G}} \quad &|x| \le 1 - \sigma \\
      f_{\mathcal{G}}(|x|) \quad &1 - \sigma < |x| \le \sigma \\
      0 &\left \vert x \right \vert > \sigma, 
    \end{cases}
    \label{UnscaledPiecewiseFunctionMedSmall}
  \end{equation}
  where $(I + K)(\widetilde{g}_{\mathcal{G}}) = \gamma_{\mathcal{G}}$ for some $\gamma_{\mathcal{G}} \in \R$. In each instance, this $\gamma_{\mathcal{G}}$ is nonzero. We compute it by calculating $(I + K)(\widetilde{g}_{\mathcal{G}})(0)$, given by  
\begin{equation}
	(\alpha_{\mathcal{G}} + \beta_{\mathcal{G}}) \int_{0}^{\sigma} \widetilde{g}_{\mathcal{G}}(y) dy + \widetilde{g}_{\mathcal{G}}(0) 
  \label{SecondScalingEqn}.
\end{equation}
For SO(Even), (\ref{SecondScalingEqn}) evaluates to 
\begin{equation}
  \beta_{\textrm{SO(Even)}} = (1 - \sigma) \cos\left( \frac{1 - \sigma}{2} - \left( \frac{\pi + 1}{4} \right) \right) + 2\left( \sin \left( \frac{\sigma}{2} - \left( \frac{\pi + 1}{4}  \right)\right) + \sin\left( \frac{\sigma}{2} + \frac{\pi - 1}{4} \right) \right). 
  \label{SOEvenFinalConstant}
\end{equation}
For SO(Odd), (\ref{SecondScalingEqn}) evaluates to 
\begin{equation}
  \beta_{\textrm{SO(Odd)}} = 3(1 - \sigma)\cos\left( \frac{\sigma - 1}{2} - \left( \frac{\pi - 1}{4} \right) \right) + 2\sin\left( \frac{\sigma}{2} + \frac{\pi - 1}{4} \right) + 2\sin\left( \frac{\sigma}{2} - \left( \frac{\pi + 1}{4} \right) \right).
  \label{SOOddFinalConstant}
\end{equation}
For the symplectic group, (\ref{SecondScalingEqn}) evaluates to 
\begin{equation}
  \beta_{\textrm{Sp}} = (\sigma - 1)\cos\left( \frac{\sigma - 1}{2} - \left( \frac{\pi - 1}{4} \right) \right) 2\sin\left( \frac{\sigma}{2} + \frac{\pi - 1}{4} \right) + 2\sin\left( \frac{\sigma}{2} - \left( \frac{\pi + 1}{4} \right) \right).
   \label{SymplecticFinalConstant}
\end{equation}
\end{proof}

\newpage

\section{Extension to $\textrm{support}(\hat{\phi}) \subseteq [-3,3]$}\label{-33section}
For each $\mathcal{G}$ and for $1 < \sigma < 1.5$, we find $g$ such that $g \ast \check{g} = \hat{\phi}$. Using the fact that $g$ must be even, we will explicitly solve for $g(x)$, where $x \ge 0$, and take the even extension of that function as our solution. \\


The main result of this section is the following. 
\begin{theorem}\label{thm:mainresult} Let $\phi$ be an even, nonnegative Schwartz test function such that $\textrm{supp}(\hat{\phi}) \subset [-2\sigma, 2\sigma]$. Then for $1 < \sigma < 1.5$ (or $2 < 2\sigma < 3$) the test function which minimizes \eqref{Minimize} is given by  $\widehat{\phi} = g \ast \check{g}$. Here $\ast$ represents convolution, $\check{g}(x) = \overline{g(-x)}$, and $g$ is given by
\begin{equation}\label{SOEvenExplicitOptimalFunction}
    g_{{\rm SO(even)},\sigma}(x) \ = \  \lambda_{{\rm SO(even)},\sigma}
    \begin{cases}
      c_{1,\mathcal{G},\sigma} \cos\left( \frac{|x|}{\sqrt{2}} \right) \ &|x| \le \sigma - 1 \\
      \cos\left( \frac{|x|}{2} - \frac{(\pi + 1)}{4} \right) \ &\sigma - 1 \le |x| \le 2 - \sigma\\
      \frac{c_{1,\mathcal{G},\sigma}}{\sqrt{2}} \sin\left( \frac{|x|-1}{\sqrt{2}} \right)  + c_{3,\mathcal{G},\sigma} \ &2 - \sigma < |x| < \sigma \\
      0 &|x| \ge \sigma,
  \end{cases}
      \end{equation}
      and
    \begin{equation}\label{OrthogOptimalFunction}
g_{{\rm O},\sigma}(x) \ = \
\begin{cases}
 \frac{1}{1 + \sigma} \ &|x| < \sigma \\
 0 \ &|x| \ge \sigma
\end{cases}
    \end{equation}
for $\mathcal{G} \ = \  {\rm O}$, and
  \begin{equation}\label{SOOdd/SpExplicitOptimalFunction}
    g_{\mathcal{G},\sigma}(x) \ = \  \lambda_{\mathcal{G},\sigma}
    \begin{cases}
      c_{1,\mathcal{G},\sigma} \cos\left( \frac{|x|}{\sqrt{2}} \right) \ &|x| \le \sigma - 1 \\
      \cos\left( \frac{|x|}{2} + \frac{(\pi - 1)}{4} \right) \ &\sigma - 1 \le |x| \le 2 - \sigma\\
      \frac{-c_{1,\mathcal{G},\sigma}}{\sqrt{2}} \sin\left( \frac{|x|-1}{\sqrt{2}} \right) + c_{3,\mathcal{G},\sigma} \ &2 - \sigma < |x| < \sigma \\	
      0 & |x| \ge \sigma
  \end{cases}
  \end{equation}
 for $\mathcal{G} =  {\rm SO(odd)}$ or ${\rm Sp}$. Here, the $c_{i,\mathcal{G},\sigma}$ and $\lambda_{\mathcal{G},\sigma}$ are easily explicitly computed, and are given later in \eqref{ActualCoefficientsSOEven}, \eqref{ActualCoeffSOOdd/Sp}, \eqref{SO(even)Scaling}, \eqref{SpScaling} and \eqref{SOOddScaling}.
\end{theorem}


Moreover, the optimal function $g_{\mathcal{G},\sigma}$, along with its coefficients $c_{i,\mathcal{G},\sigma}$ and its scaling factor $\lambda_{\mathcal{G},\sigma}$, all depend on $\sigma$ and $\mathcal{G}$. As this will be clear from equations \eqref{ActualCoefficientsSOEven} to \eqref{SOOddScaling}, to simplify the notation we omit the subscripts $\mathcal{G}$ and $\sigma$ when there is no danger of confusion.

To help illustrate the main Theorem, we include plots of the optimal $g$ for the groups SO(even), SO(odd), and Sp in Figure \ref{fig:optplots}, and the plots for the corresponding optimal $\phi$ in Figure \ref{fig:optphiplots}; we do not include the optimal plots for the orthogonal case, as the resulting $g$ is constant (and equal to $(1 + \sigma)^{-1}$). 

\begin{figure}[h]
\begin{center}
\scalebox{.5545}{\includegraphics{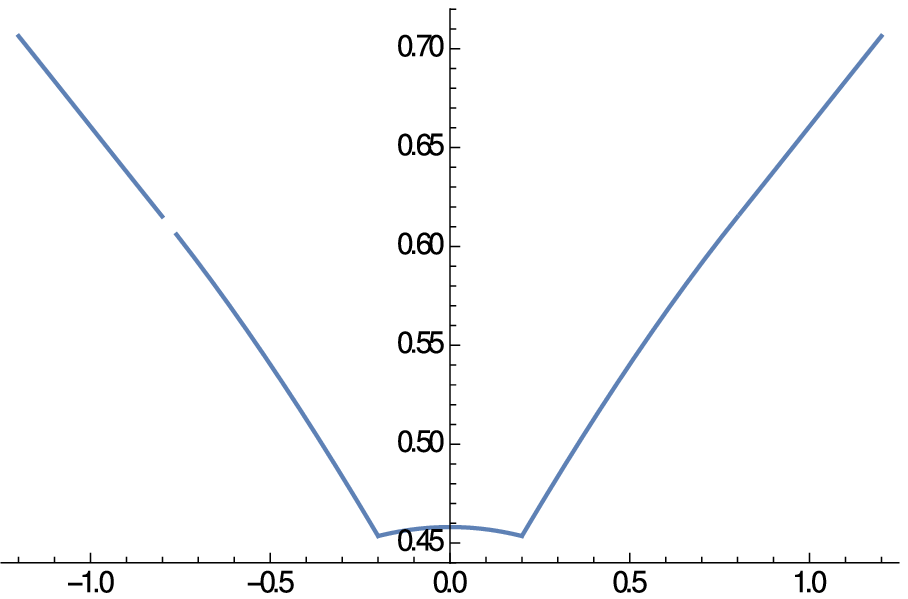}} \ \ \scalebox{.5545}{\includegraphics{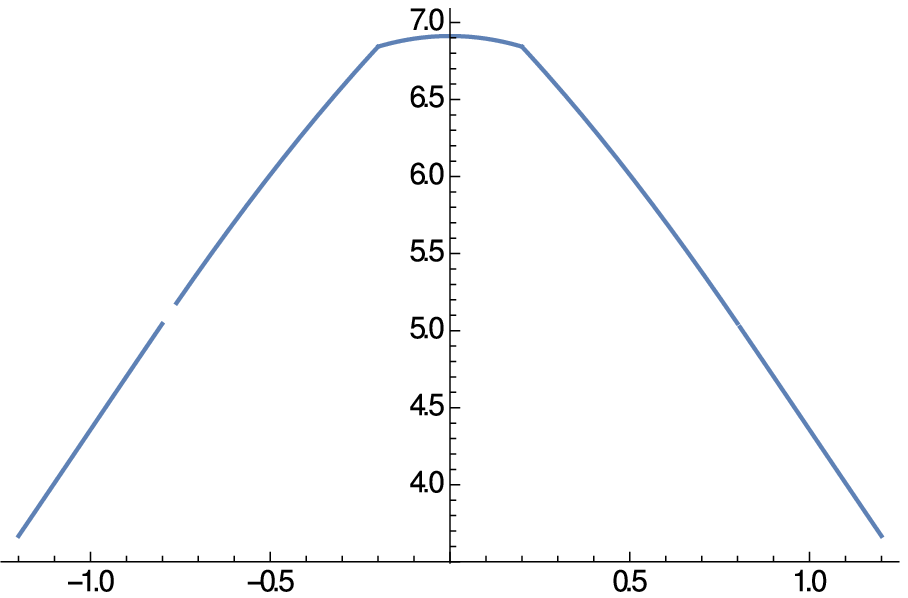}} \ \ \scalebox{.5545}{\includegraphics{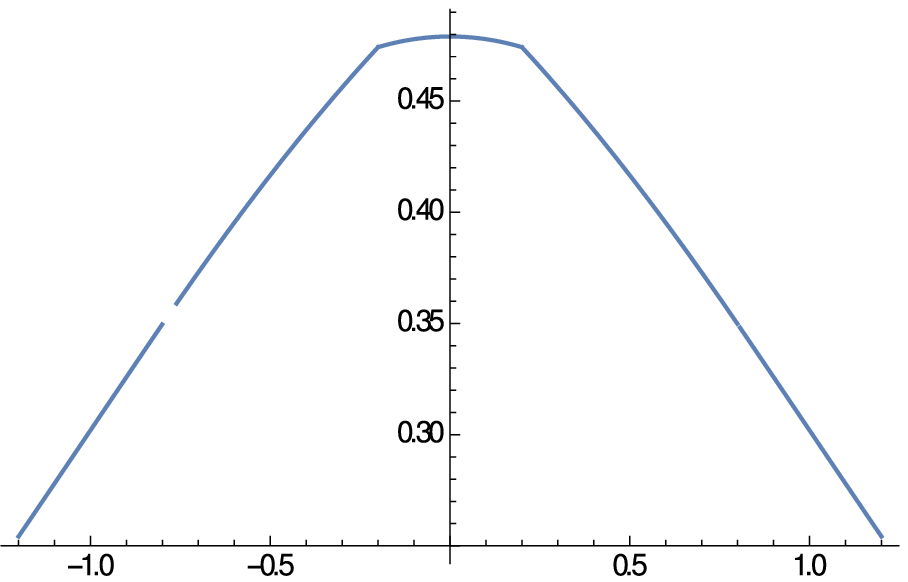}}
\caption{\label{fig:optplots} Plots of the optimal $g$ with $\sigma = 1.2$. Left: Optimal SO(even) function. Middle: Optimal Sp function. Right: Optimal SO(odd) function.}
\end{center}
\end{figure}

\begin{figure}[h]
\begin{center}
\scalebox{.5545}{\includegraphics{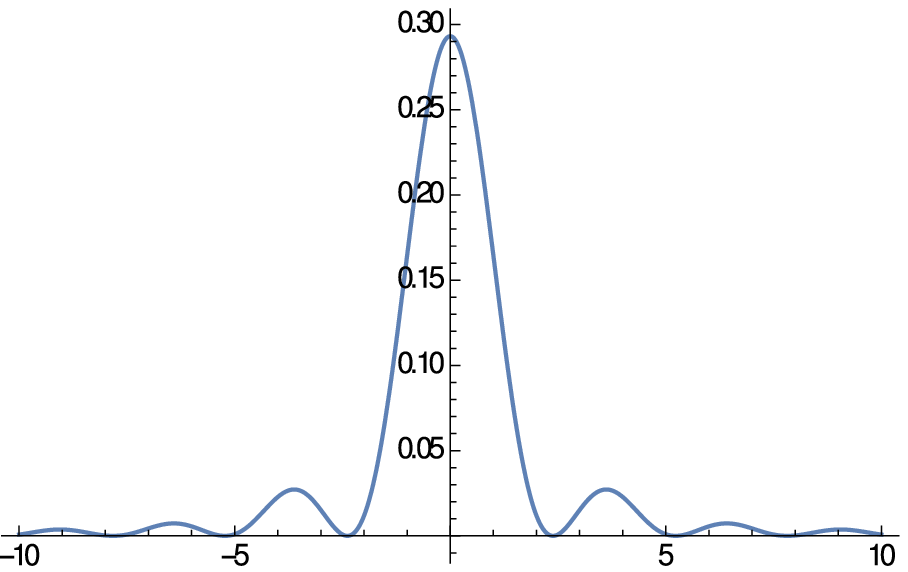}} \ \ \scalebox{.5545}{\includegraphics{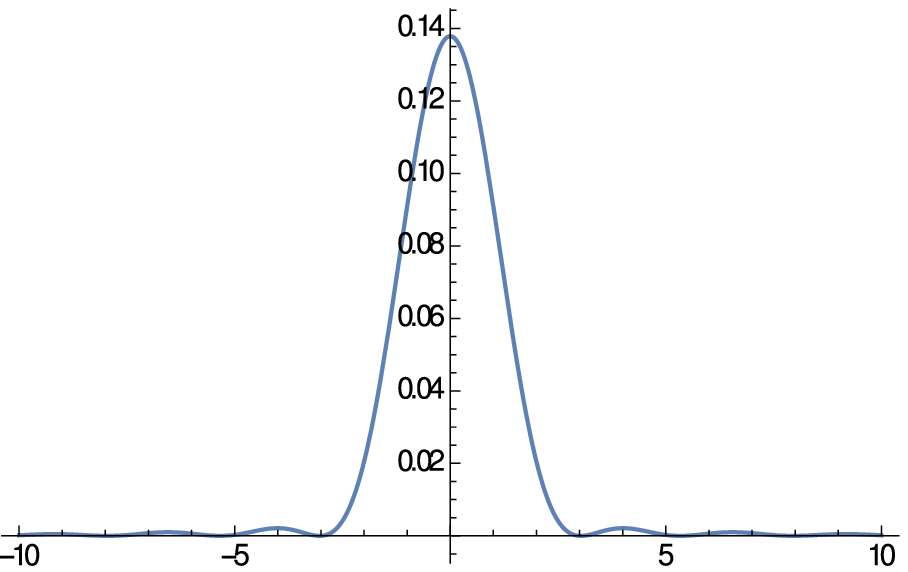}} \ \ \scalebox{.5545}{\includegraphics{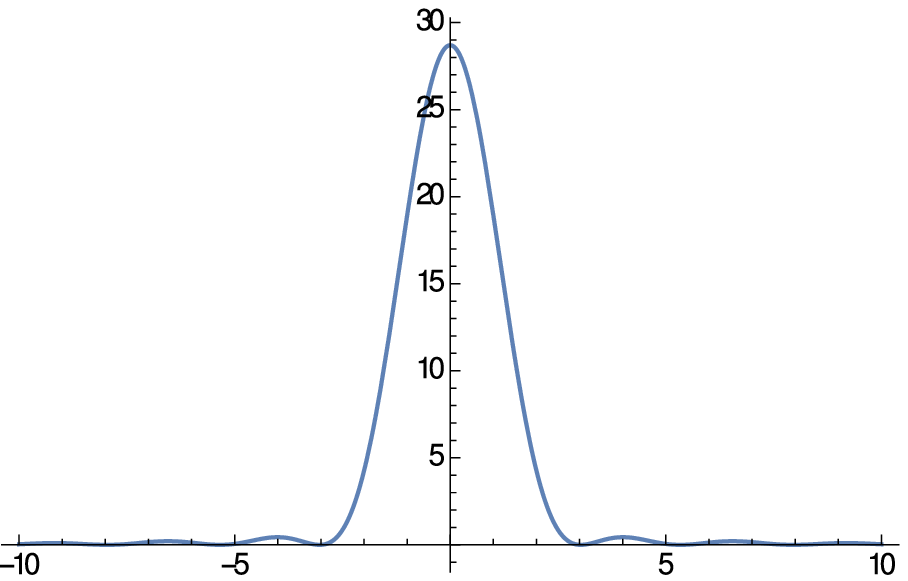}}
\caption{\label{fig:optphiplots} Plots of the optimal $\phi$ with $\sigma = 1.2$ . Left: Optimal SO(even) function. Middle: Optimal Sp function. Right: Optimal SO(odd) function.  }
\end{center}
\end{figure}

The most immediate application of these results is the upper bound on average rank described in \eqref{AvgRankBound}. However, at present it is not verified that these bounds apply to any family of $L$-functions. The largest 1-level density support occurs in families of cuspidal newforms \cite{ILS} and Dirichlet $L$-functions \cite{FiM} (though see also \cite{AM} for Maass forms), where we can take $2\sigma < 2$. It is possible to obtain better bounds on vanishing by using the 2 or higher level densities, though as remarked above in practice the reduced support means these results are not better than the 1-level for extra vanishing at the central point but do improve as we ask for more and more vanishing (see \cite{HM,FrM}). Yet, it is conjectured that support can be extended in some cases. For example, Hypothesis $S$ implies that the one-level density conjecture holds for orthogonal families with $\textrm{support}(\hat{\phi}) \subset \left( -\frac{22}{9}, \frac{22}{9} \right)$. This computation shows the precise benefit of such efforts with respect to bounding average rank.  
\begin{corollary}
  \label{cor:averankmainresult} Let $\mathcal{F}$ be a family of $L$-functions such that, in the limit as the conductors tend to infinity, the 1-level density is known to agree with the scaling limit of unitary, symplectic or orthogonal matrices. Then for every $\varepsilon > 0$ in the limit the average rank is bounded above by $\varepsilon + $
\begin{equation}
  \begin{cases}
      \frac{4 \sqrt{2} \sin \left(\frac{1}{4} (3-2 \sigma)\right)+2 (\sigma-1) \sin \left(\frac{1}{4} (-2 \sigma+\pi +3)\right)+\sin \left(\frac{1}{4} (2 \sigma+\pi -3)\right) \left(\sqrt{2} (\sigma+1) \tan \left(\frac{\sigma-1}{\sqrt{2}}\right)+2\right)}{8 \sqrt{2} \sin \left(\frac{1}{4} (3-2 \sigma)\right)+8 (\sigma-1) \sin \left(\frac{1}{4} (-2 \sigma+\pi +3)\right)+4 \sqrt{2} \sigma \sin \left(\frac{1}{4} (2 \sigma+\pi -3)\right) \tan \left(\frac{\sigma-1}{\sqrt{2}}\right)}  &\mathcal{G} \ = \ {\rm SO(even)} \\
      \frac{-2 (\sigma-1) \sin \left(\frac{1}{4} (2 \sigma+\pi -3)\right)-4 \sqrt{2} \sin \left(\frac{1}{4} (3-2 \sigma)\right)+\sin \left(\frac{1}{4} (-2 \sigma+\pi +3)\right) \left(\sqrt{2} (\sigma-3) \tan \left(\frac{\sigma-1}{\sqrt{2}}\right)+2\right)}{8 (\sigma-1) \sin \left(\frac{1}{4} (2 \sigma+\pi -3)\right)+8 \sqrt{2} \sin \left(\frac{1}{4} (3-2 \sigma)\right)-4 \sqrt{2} (\sigma-2) \sin \left(\frac{1}{4} (-2 \sigma+\pi +3)\right) \tan \left(\frac{\sigma-1}{\sqrt{2}}\right)}  &\mathcal{G} \ = \ {\rm Sp} \\
      \frac{6 (\sigma-1) \sin \left(\frac{1}{4} (2 \sigma+\pi -3)\right)+4 \sqrt{2} \sin \left(\frac{1}{4} (3-2 \sigma)\right)+\sin \left(\frac{1}{4} (-2 \sigma+\pi +3)\right) \left(\sqrt{2} (5-3 \sigma) \tan \left(\frac{\sigma-1}{\sqrt{2}}\right)+2\right)}{8 (\sigma-1) \sin \left(\frac{1}{4} (2 \sigma+\pi -3)\right)+8 \sqrt{2} \sin \left(\frac{1}{4} (3-2 \sigma)\right)-4 \sqrt{2} (\sigma-2) \sin \left(\frac{1}{4} (-2 \sigma+\pi +3)\right) \tan \left(\frac{\sigma-1}{\sqrt{2}}\right)}  &\mathcal{G} \ = \ {\rm SO(odd)} \\
    \frac{1}{2\sigma} + \frac{1}{2} &\mathcal{G} \ = \ {\rm O}.
  \end{cases}
  \label{AvgRankBoundsForExtendedSupport}
\end{equation}
for $1 < \sigma < 1.5$. 
\end{corollary}

\begin{remark}
  We only list $g$ and not the optimal test functions or their Fourier transforms above, as we do not need either function for the computation of the infimum. By Proposition \ref{OptimalCriterionAndComputationLemma}, the infimum is given by
\begin{equation}\label{infcomputationeqn}
  {\rm inf}(\mathcal{G},\sigma)\ =\ \left( \int_{-\sigma}^{\sigma} g(x) dx \right)^{-1}.  
\end{equation}
which is finite because of the requirement that $\phi(0) > 0$.  
\end{remark}

A natural choice of a test function one side of the Fourier pair
\begin{equation}
\phi(x) = \left( \frac{\sin(2 \sigma \pi x)}{2 \sigma \pi x} \right)^{2}, \quad \quad \widehat{\phi}(y) = \frac{1}{2\sigma}\left( 1 - \frac{|y|}{2\sigma} \right) \ \quad \textrm{if} \ |y| < 2 \sigma;
\label{NaiveFourierPair}
\end{equation}
this is the function used for the initial computation of average rank bounds in \cite{ILS} and are optimal for $\sigma = 1$. For the groups ${\rm SO(even)}, {\rm Sp}, {\rm SO(odd)}$, and for $1 < \sigma < 1.5$ the functions we find provide a modest improvement for the upper bounds on average rank over the pair \eqref{NaiveFourierPair}. We illustrate the improvement in Figure \ref{fig:comparisons}, which is much easier to process than \eqref{AvgRankBoundsForExtendedSupport}. 

\begin{figure}[h]
\begin{center}
\scalebox{.5545}{\includegraphics{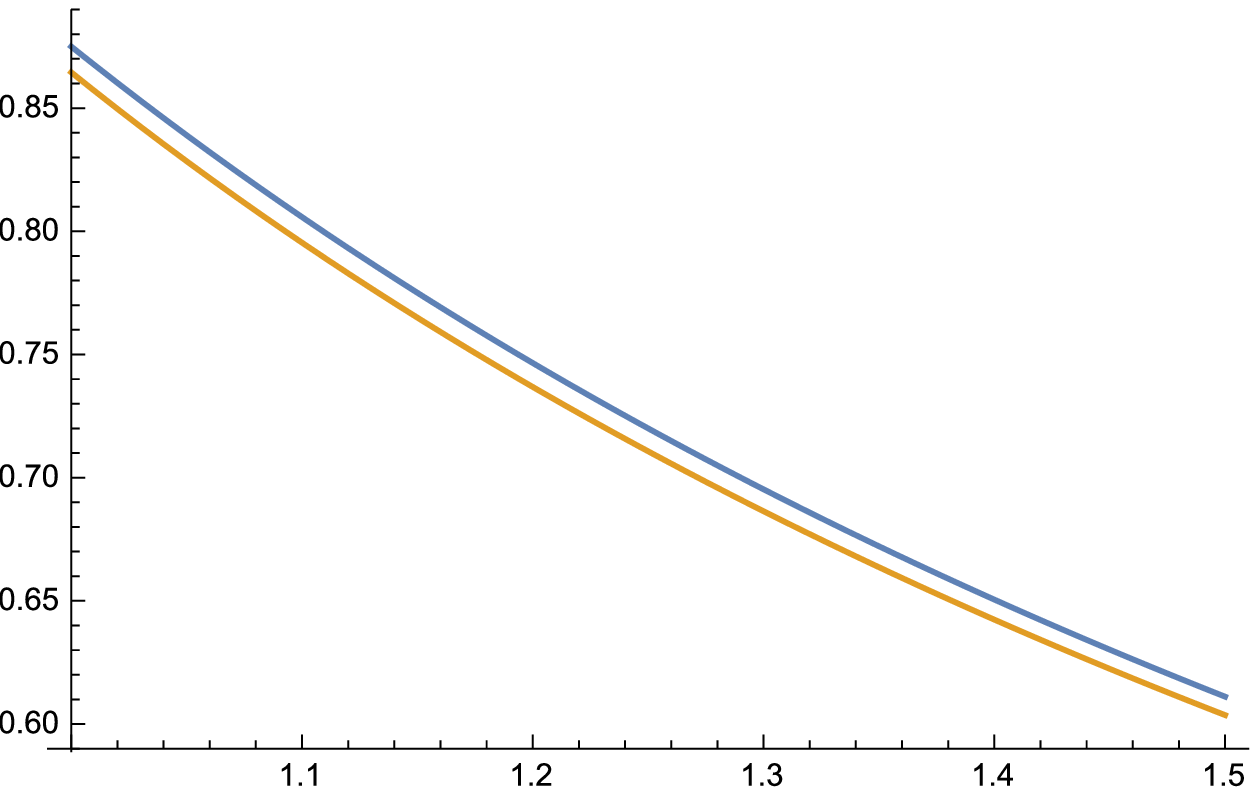}} \ \ \scalebox{.5545}{\includegraphics{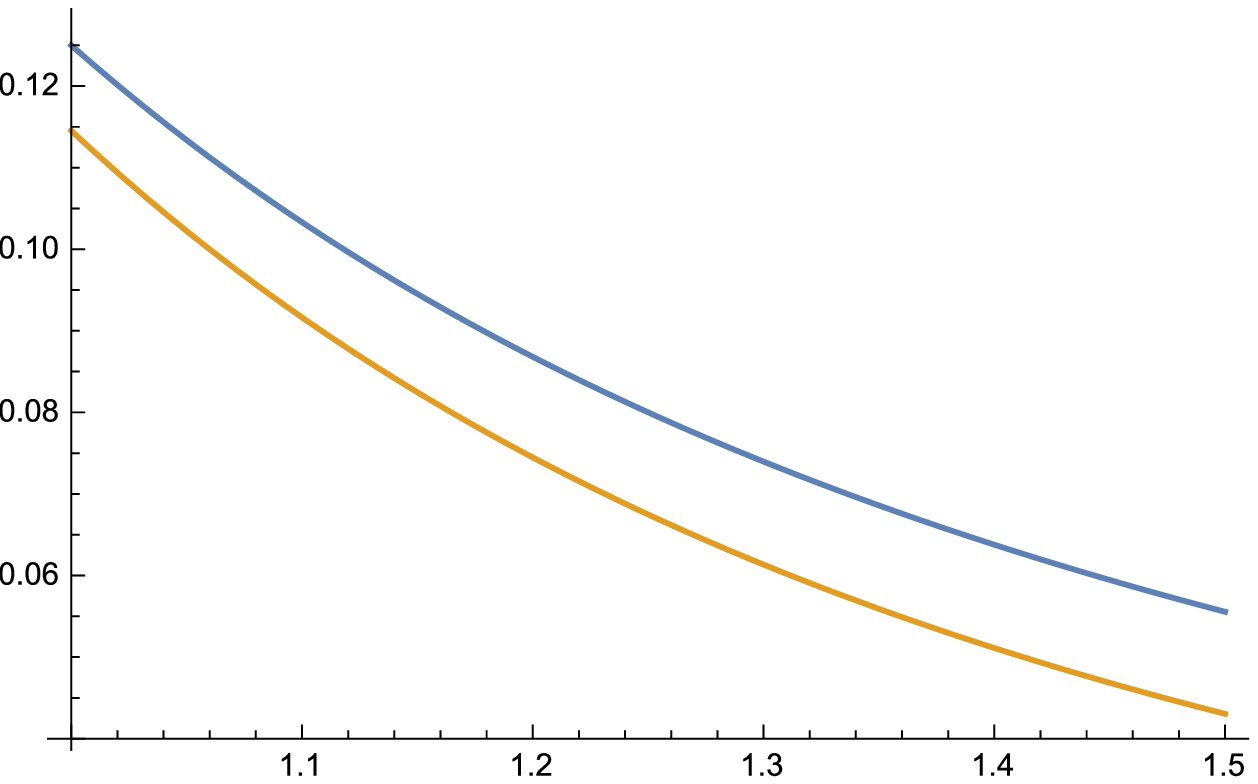}} \ \ \scalebox{.5545}{\includegraphics{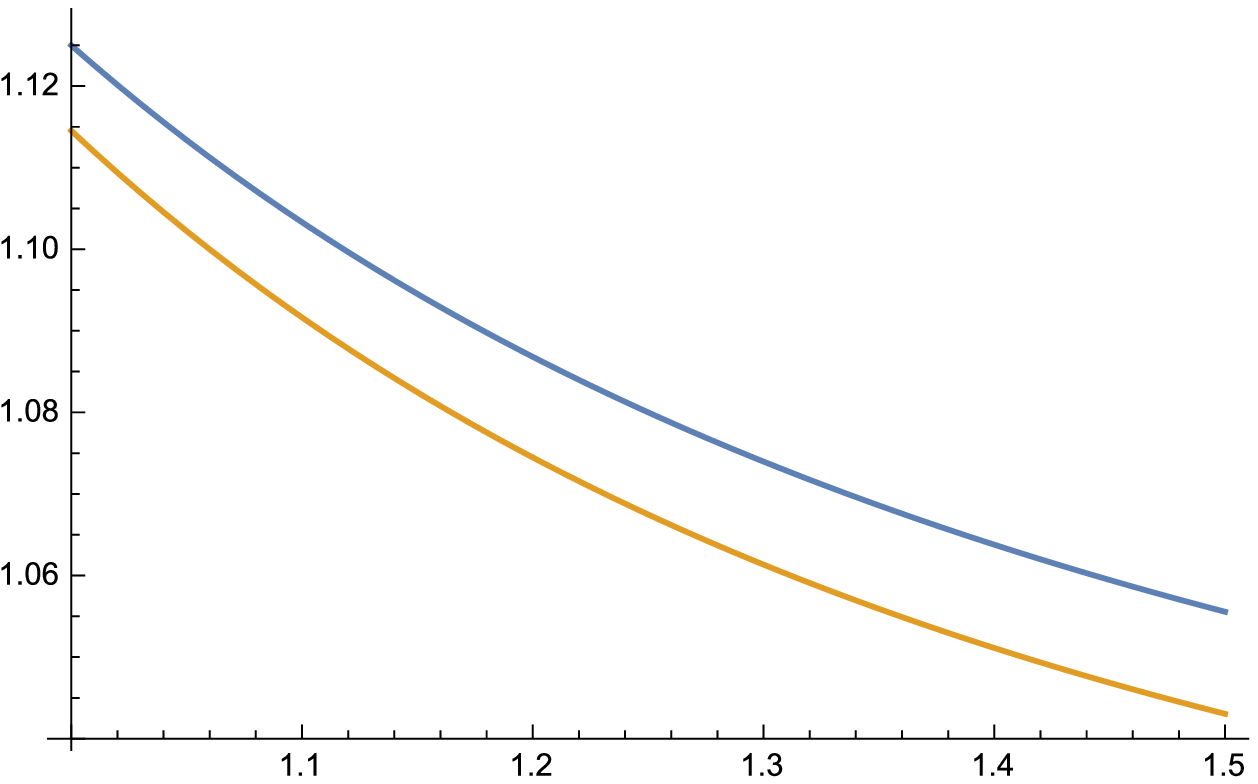}}
\caption{\label{fig:comparisons} Comparison of upper bounds. The larger bound is from using the sub-optimal naive guess \eqref{NaiveFourierPair}, the lower is from using our results from \eqref{AvgRankBoundsForExtendedSupport}. Left: $\mathcal{G}$ = SO(even). Middle: $\mathcal{G} = {\rm Sp}$. Right: $\mathcal{G}$ = SO(odd).}
\end{center}
\end{figure}

Finally, we will show a plot of the optimal $\phi$ for these groups, compared to the na\"{i}ve choice \eqref{NaiveFourierPair}. We do compute the optimal $\phi$ for each group, as the formulae are long, but can be isolated. We include them in Appendix \ref{approxoptappendix}. The following plot shows all four optimal $\phi$ for $\sigma = 1.2$, which corresponds to $\textrm{supp}(\hat{\phi}) \subset [-2.4,2.4]$

\begin{figure}[h]
\begin{center}
\scalebox{1.6}{\includegraphics{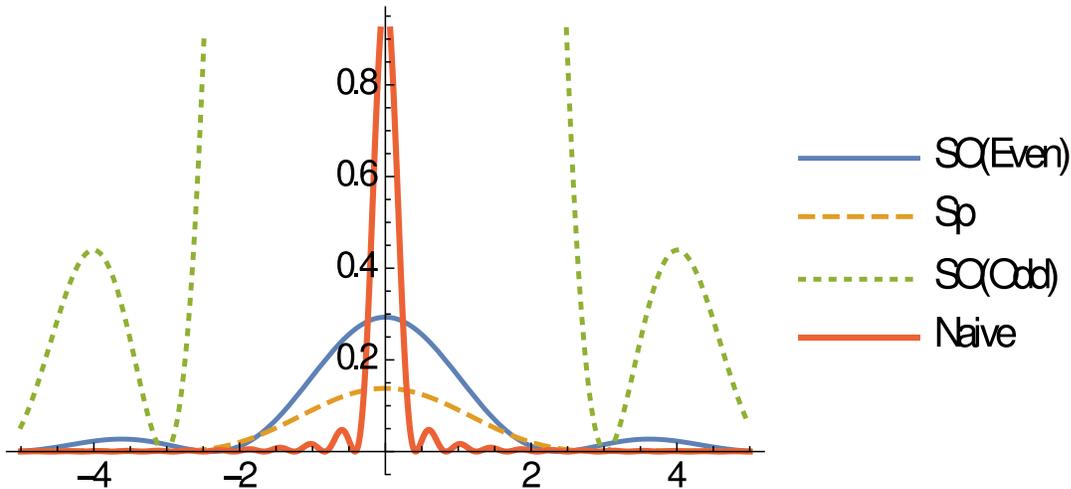}} 
\caption{\label{fig:comparisonfourphiplots} Comparison of the three optimal $\phi$ for SO(even), Sp and SO(odd), as well as the natural  choice from \eqref{NaiveFourierPair}, when $\sigma = 1.2 $.}
\end{center}
\end{figure}

The broad strategy of the proof of Theorem \ref{thm:mainresult} is to use an operator equation from \cite{ILS} to show (non-constructively) that for all $\sigma \in \R^{+}$, there exists a unique optimal test function with ${\rm supp}(\widehat{\phi}) \subseteq [-2\sigma,2\sigma]$ that minimizes the functional
\begin{equation}\label{OptimalityCriterion}
\frac{\int_{-\infty}^{\infty}\phi(x)W_{\mathcal{G}}(x) dx}{\phi(0)}.
\end{equation}
We find a collection of necessary conditions that leave us with precisely one choice for $\phi$.

More explicitly, our argument proceeds as follows.

\begin{enumerate}
  \item \label{firstgoal} We cite the results of Section \ref{OptCritSect}, which show that the optimality criterion \eqref{CriterionForOptimal} holds for all $\sigma \in \R^{+}$, where $\textrm{supp}(\hat{\phi}) \subset [-2\sigma,2\sigma]$.
  \item \label{thirdgoal} Our kernels give us a system of location-specific integral equations. Using the smoothness result of Proposition \ref{TwiceDiff}, we convert those to a system of location-specific delay differential equations, which hold almost everywhere. 
  \item \label{fourthgoal} We solve this sytem to find an $n$-parameter family in which our solution lives. To find this solution, we incorporate symmetries of $g$ -- namely that $g$ must be even. 
  \item \label{fifthgoal} Incorporating more necessary conditions on $g$, we reduce the family to a single candidate function -- by our existence result, the sole remaining candidate is our $g$, from which we may obtain the infimum and our optimal test function $\phi$.
\end{enumerate}

From the list above, we accomplish goal \ref{thirdgoal} in Subsection \ref{thirdgoalsec}, goal \ref{fourthgoal} in Subsection \ref{fourthgoalsec}, and goal \ref{fifthgoal} in Subsection \ref{fifthgoalsec}. Note that we have already found the optimal test functions for $\mathcal{G} = O$, for all levels of support, in Section \ref{orthogsec}.  

\subsection{A System of Integral Equations}

%
%
There are three intervals of importance in our study of this function. These are
\begin{equation}
	\begin{aligned}
	  I_{1} &\ := \  [0, \sigma - 1] \\
  J_{0} &\ := \  [\sigma - 1,2 - \sigma] \\
  I_{0} &\ := \  [2 - \sigma, \sigma].
	\end{aligned}
	\label{IntervalDefs}
\end{equation}
Our function will be defined piecewise, on each interval. \\

As $g$ is even, it suffices to find $g$ on $[0,\sigma]$, which means finding $g$ on all of the intervals above. Examining the kernels in \eqref{mfunctions} and the requirement \eqref{CriterionForOptimal}, we see that for $x \in I_{1}$, the optimal $g$ satisfies
\begin{equation}
  g(x) + \beta_{\mathcal{G}} \int_{0}^{x + 1} g(y) dy + \beta_{\mathcal{G}}\int_{0}^{1-x} g(y) dy + \alpha_{\mathcal{G}} \int_{0}^{\sigma} g(y) dy \ = \  1,
  \label{IntEqnOnI1}
\end{equation}
and for $x \in I_{0}$ or $J_{0}$, we have
\begin{equation}
  g(x) + \beta_{\mathcal{G}} \int_{0}^{1-x} g(y) dy + \alpha_{\mathcal{G}}\int_{0}^{\sigma} g(y) dy  \ = \  1.
  \label{IntEqnOnI23}
\end{equation}
In equations \eqref{IntEqnOnI1} and \eqref{IntEqnOnI23}, we note that $\alpha_{\mathcal{G}} = 0$ for $\mathcal{G} \not = {\rm SO(odd)}$ and $1$ for $\mathcal{G} = {\rm SO(odd)}$ and $\beta_{\mathcal{G}} = 1/2$ for $\mathcal{G} = {\rm SO(even)}$ and $-1/2$ for $\mathcal{G} = {\rm Sp}$ or $\mathcal{G} = {\rm SO(odd)}$.

\subsection{Conversion to Location-Specific System of Delay Differential Equations}\label{thirdgoalsec}
Lemma \ref{TwiceDiff} justifies differentiation of \eqref{IntEqnOnI1} and \eqref{IntEqnOnI23} under the integral signs, which gives the following system of location-specific delay differential equations:
\begin{align}
  g'(x) + \beta_{\mathcal{G}}g(x + 1) - \beta_{\mathcal{G}}g(x-1) \ = \  0
  \label{DelayDiff1}\\
  g'(x + 1) - \beta_{\mathcal{G}}g(x) \ = \  0
  \label{DelayDiff2} \\
  g'(x) \mp \beta_{\mathcal{G}} g(x \mp 1) \ = \ 0 \label{DelayDiff3}
\end{align}
where \eqref{DelayDiff1} holds for $x \in I_{1}$, \eqref{DelayDiff2} holds for $x+1 \in I_{1}$ or $J_{0}$, and $\eqref{DelayDiff3}$ holds for $x \in \pm I_0$.

\subsection{Solving The System}\label{fourthgoalsec}


\begin{lemma} The optimal $g$ satisfies
  \begin{equation}\label{OptimalgOnTwoIntervalsLemma} 
    g(x) \ = \
    \begin{cases}
     c_{1} \cos\left( \frac{x}{\sqrt{2}} \right) + c_{2} \sin \left( \frac{x}{\sqrt{2}} \right)  &\textrm{if} \ x \in I_{1} \\
     c_{1}\beta_{\mathcal{G}}\sqrt{2} \sin\left( \frac{x - 1}{\sqrt{2}} \right) - c_{2}\beta_{\mathcal{G}} \sqrt{2} \cos\left( \frac{x-1}{\sqrt{2}} \right) + c_{3} &\textrm{if} \ x \in I_{0}
    \end{cases}
   \end{equation}
  for some $c_{i} \in \R$.
\end{lemma}

Before proving this lemma, it is important to note the following symmetry among our intervals. We first set some  notation. If $a$ is a number and $I$ is an interval, 
\begin{equation}
  a-I \ := \ \{x: x = a - y, y \in I\}. 
  \label{DefIntervalSubtraction}
\end{equation}  
Note that for the intervals defined in \eqref{IntervalDefs}, we have
\begin{equation}
  1 - I_0 \ = \ (I_{1} \cup -I_{1}) 
  \label{SymmetryI1I3}
\end{equation}
and
\begin{equation}
  1 - J_{0} \ = \  J_{0},
  \label{SymmetryI2I2}
\end{equation}
though we will not use this fact until later, in \eqref{DiffEqVersion2}. 

\begin{proof}
Let $x \in I_0$. Differentiating \eqref{DelayDiff3} yields
  \begin{equation}
    g'(x) - \beta_{\mathcal{G}}g'(x-1) \ = \  0.
    \label{DiffFirstEq}
  \end{equation}
Because of the symmetry \eqref{SymmetryI1I3}, we may use equation \eqref{DelayDiff1} on the $x-1$ term. This gives us the following equation: 
  \begin{equation}\label{ApplyingSecondEquation}
	  g''(x) + \beta_{\mathcal{G}}^{2}(g(x) - \beta_{\mathcal{G}}^2 g(x - 2)) \ = \ 0.
  \end{equation}

  Differentiate, and apply \eqref{DelayDiff3} to the $g(x-2)$ term to get 
  \begin{equation}
  g^{(3)}(x) + \beta^2_{\mathcal{G}} g'(x) - \beta^{-2}_{\mathcal{G}} g'(x-2) \ = \ 0 
  \label{sillylabel1}
  \end{equation}
  and then 
	\begin{equation}
		g^{(3)}(x) + \beta^{2}_{\mathcal{G}} g'(x) + \beta^{3} g'(x-1) \\
		\label{AlmostDoneLabel!}
	\end{equation}
	which, after applying \eqref{DelayDiff3} to the $x-1$ term, becomes
	\begin{equation}
		g^{(3)}(x) + \beta^{2}_{\mathcal{G}} g'(x) \ = \ 0 
		\label{FinalOutsideDiffEq}
	\end{equation}
	for $x \in I_0$. Since $\beta_{\mathcal{G}} = \pm \frac{1}{2}$, we have our result. 
\end{proof}

\ \\
The associated delay differential equation on $J_{0}$ is
\begin{equation}
  g'(x) - \beta_{\mathcal{G}}g(1-x) \ = \  0.
  \label{DiffEqVersion2}
\end{equation}
Due to symmetry \eqref{SymmetryI2I2}, when $x \in J_{0}, 1-x \in J_{0}$ as well. Recall from Lemma \ref{UniqueOneParameterFamily} that the solution in this context falls in the one-parameter family 
  \begin{equation}
    c_{1}\cos\left( \beta_{\mathcal{G}}x - \left( \frac{\pi + 2\beta_{\mathcal{G}}}{4} \right) \right)
    \label{OneParamFamEqnMiddlePart}
  \end{equation}
\subsection{Finding Coefficients}\label{fifthgoalsec}
Substituting values for $\alpha_{i,\mathcal{G}}$ for $i=1,2$, we find
  \begin{equation}\label{SOEvenExplicitOptimalFunction2}
    g_{{\rm SO(even)}}(x) \ = \  \lambda_{{\rm SO(even)}}
    \begin{cases}
      c_{1,\mathcal{G}} \cos\left( \frac{|x|}{\sqrt{2}} \right) + c_{2,\mathcal{G}} \sin\left( \frac{|x|}{\sqrt{2}} \right) \ &|x| \ \le\ \sigma - 1 \\
      \cos\left( \frac{|x|}{2} - \frac{(\pi + 1)}{4} \right) \ &\sigma - 1\ \le\ |x|\ \le\ 2 - \sigma\\
      \frac{c_{1,\mathcal{G}}}{\sqrt{2}} \sin\left( \frac{x-1}{\sqrt{2}} \right) - \frac{c_{2,\mathcal{G}}}{\sqrt{2}} \cos\left( \frac{x-1}{\sqrt{2}} \right) + c_{3} \ &2 - \sigma\ <\ |x|\ \le\ \sigma	
    \end{cases}
      \end{equation}
  for $\mathcal{G} = {\rm SO(even)}$ and
  \begin{equation}\label{SOOdd/SpExplicitOptimalFunction2}
    g_{\mathcal{G}}(x) \ = \  \lambda_{\mathcal{G}}
    \begin{cases}
      c_{1,\mathcal{G}} \cos\left( \frac{|x|}{\sqrt{2}} \right) + c_{2,\mathcal{G}} \sin\left( \frac{|x|}{\sqrt{2}} \right)\ &|x|\ \le\ \sigma - 1 \\
      \cos\left( \frac{|x|}{2} + \frac{(\pi - 1)}{4} \right) \ &\sigma - 1\ \le\ |x|\ \le\ 2 - \sigma\\
      \frac{-c_{1,\mathcal{G}}}{\sqrt{2}} \sin\left( \frac{x-1}{\sqrt{2}} \right) + \frac{c_{2,\mathcal{G}}}{\sqrt{2}} \cos\left( \frac{x-1}{\sqrt{2}} \right) + c_{3}\ &2 - \sigma\ <\ |x|\ \le\ \sigma	
    \end{cases}
      \end{equation}
  for $\mathcal{G} = {\rm SO(odd)}$ or Sp.

  \begin{lemma}\label{OptimalCoeffExist}
    There exist unique, computable coefficients $c_{i, \mathcal{G}}, \lambda_{\mathcal{G}}$ (for $i = 1,2,3$) so that the functions \eqref{SOEvenExplicitOptimalFunction2} and \eqref{SOOdd/SpExplicitOptimalFunction2} satisfy $(I + K)(g) = 1$ and are thus optimal.
  \end{lemma}

  \begin{proof}
    We use \eqref{CriterionForOptimal} and Lemma \ref{TwiceDiff} to find more necessary conditions on such a $g$. In particular, we impose the three restrictions:
    \begin{equation}
	    \begin{aligned}
	    \lim_{x \to (\sigma - 1)^{-}}g(x) &\ = \  \lim_{x \to (\sigma - 1)^{+}}g(x)  \\
    (I + K)(g)(0) &\ = \  (I + K)(g)(.5) \\
(I + K)(g)(0) &\ = \  (I + K)(g)(\sigma).
	    \end{aligned}
	  \label{Req3:LeftAndRightScalarsAgree}
    \end{equation}
  The first gives continuity, the second and third ensure that $(I + K)(g)$ is constant; however, they do not ensure that constant is 1. That is accomplished by scaling by $\lambda_{\mathcal{G}}$. This gives us the matrix equations
 \begin{equation}
  \begin{pmatrix}
    \cos\left( \frac{\sigma - 1}{\sqrt{2}} \right) & \sin\left( \frac{\sigma - 1}{\sqrt{2}} \right) & 0 \\
    \cos\left( \frac{\sigma - 1}{\sqrt{2}} \right) & 0 & 0 \\
 \frac{1}{\sqrt{2}}\sin\left( \frac{\sigma - 1}{\sqrt{2}} \right) + \cos\left( \frac{\sigma - 1}{\sqrt{2}} \right) & \sqrt{2} - \frac{1}{\sqrt{2}}\cos\left( \frac{\sigma -1}{\sqrt{2}} \right) & -1
  \end{pmatrix}
  \begin{pmatrix}
    c_1 \\
    c_2 \\
    c_3
  \end{pmatrix}
  \ = \
  \begin{pmatrix}
    g_{2}(\sigma - 1) \\
    g_{2}(\sigma - 1) \\
    g_{2}(\sigma - 1) - g_{2}(2 - \sigma)
  \end{pmatrix}
  \label{SO(even)MatrixEqn}
\end{equation}
for $\mathcal{G} = {\rm SO(even)}$ and
\begin{equation}
  \begin{pmatrix}
    \cos\left( \frac{\sigma - 1}{\sqrt{2}} \right) & \sin\left( \frac{\sigma - 1}{\sqrt{2}} \right) & 0 \\
    \cos\left( \frac{\sigma - 1}{\sqrt{2}} \right) & 0 & 0 \\
 \frac{-1}{\sqrt{2}}\sin\left( \frac{\sigma - 1}{\sqrt{2}} \right) + \cos\left( \frac{\sigma - 1}{\sqrt{2}} \right) & -\sqrt{2} + \frac{1}{\sqrt{2}}\cos\left( \frac{\sigma-1}{\sqrt{2}} \right) & -1
  \end{pmatrix}
  \begin{pmatrix}
    c_1 \\
    c_2 \\
    c_3
  \end{pmatrix}
  \ = \
\begin{pmatrix}
    g_{2}(\sigma - 1) \\
    g_{2}(\sigma - 1) \\
    g_{2}(\sigma - 1) - g_{2}(2 - \sigma)
  \end{pmatrix}
  \label{SO(odd)AndSpMatrixEqn}
\end{equation}
for $\mathcal{G} = {\rm SO(odd)}$ or ${\rm Sp}$.  Here, $g_{2}$ is $g$ restricted to $J_{0}$. \\

Expanding these matrices along the their third columns, we see that
\begin{equation}
  \left| A_{{\rm SO(even)}} \right| \ = \  \left| A_{\textrm{SO(odd)/Sp}} \right| \ = \  \cos\left( \frac{\sigma - 1}{\sqrt{2}} \right) \sin \left( \frac{\sigma - 1}{\sqrt{2}} \right), 
  \label{The3x3Determinants}
\end{equation}
which are both nonzero for $1 < \sigma < 1.5$. Solving the matrix equations, we obtain
\begin{eqnarray}
  c_{1,{\rm SO(even)}} &\ = \ &  \cos \left(\frac{\sigma-1}{2}+\frac{1}{4} (-1-\pi )\right) \sec \left(\frac{\sigma-1}{\sqrt{2}}\right) \notag  \\
  c_{2,{\rm SO(even)}} &\ = \ & 0 \notag\\
  c_{3,{\rm SO(even)}} &\ = \ & \sin \left(\frac{1}{4} (2 \sigma+3 \pi -3)\right)+\frac{\sin \left(\frac{1}{4} (-2 \sigma+3 \pi +3)\right) \tan \left(\frac{\sigma-1}{\sqrt{2}}\right)}{\sqrt{2}}, \notag \\
  \label{ActualCoefficientsSOEven}
\end{eqnarray}
and
\begin{eqnarray}
  c_{1,\mathcal{G}} &\ = \ & \cos \left(\frac{1-\sigma}{2}+\frac{1-\pi }{4}\right) \sec \left(\frac{\sigma-1}{\sqrt{2}}\right) \nonumber\\
  c_{2,\mathcal{G}} &\ = \ & 0 \notag \nonumber\\
  c_{3,\mathcal{G}} &\ = \ & \sin \left(\frac{1}{4} (-2 \sigma+3 \pi +3)\right)-\frac{\sin \left(\frac{1}{4} (2 \sigma+3 \pi -3)\right) \tan \left(\frac{\sigma-1}{\sqrt{2}}\right)}{\sqrt{2}}
  \label{ActualCoeffSOOdd/Sp}
\end{eqnarray}
for $\mathcal{G} = {\rm SO(odd)}$ or ${\rm Sp}$.

We currently have $(I + K)(\widetilde{g}) = c$ for some constant $c$. Here, $\widetilde{g}$ is the unscaled optimal function. As some of our $c_{i,\mathcal{G}}$ are nonzero and the operator $(I + K)$ is positive definite, this constant is nonzero and it can therefore be scaled to be one. We find the correct scaling factor by computing $((I + K)(\widetilde{g})(0))^{-1}$, setting that equal to $\lambda_{\mathcal{G}}$ in \eqref{SOEvenExplicitOptimalFunction2} and \eqref{SOOdd/SpExplicitOptimalFunction2}. From these computations, we find
  \begin{eqnarray}
    \lambda_{\mathcal{G}, \sigma} & \ = \  &\widetilde{g}_{\textrm{SO(Even)}}(0) + \frac{1}{2}\int_{-1}^{1} \widetilde{g}_{\textrm{SO(Even)}}(y) dy \nonumber\\
    &\ = \ &   2 \sqrt{2} \sin \left(\frac{1}{4} (3-2 \sigma)\right)+(\sigma-1) \sin \left(\frac{1}{4} (-2 \sigma+\pi +3)\right) \nonumber\\
    & & \ \ \ \ +\ \frac{1}{2} \sin \left(\frac{1}{4} (2 \sigma+\pi -3)\right) \left(\sqrt{2} (s+1) \tan \left(\frac{s-1}{\sqrt{2}}\right)+2\right) \label{SO(even)Scaling}
  \end{eqnarray}
  for $\mathcal{G} =   {\rm SO(even)}$, and
  \begin{eqnarray}
    \lambda_{\mathcal{G},\sigma} &\ = \ & \widetilde{g}_{Sp}(0) - \frac{1}{2} \int_{-1}^{1} \widetilde{g}_{Sp}(y) dy \nonumber\\
    &\ = \ &  -2 \sqrt{2} \sin \left(\frac{1}{4} (3-2 \sigma)\right)+(\sigma-1) \cos \left(\frac{1}{4} (2 \sigma+3 \pi -3)\right) \nonumber\\
    & & \ \ \ \ +\ \frac{1}{2} \sin \left(\frac{1}{4} (-2 \sigma+\pi +3)\right) \left(\sqrt{2} (\sigma-3) \tan \left(\frac{\sigma-1}{\sqrt{2}}\right)+2\right)
     \label{SpScaling}
  \end{eqnarray}
  for $\mathcal{G} = {\rm Sp}$, and
\begin{eqnarray}
  \lambda_{\mathcal{G},\sigma} &\ = \ & \lambda_{Sp} + \int_{-\sigma}^{\sigma} \widetilde{g}_{\textrm{SO(Odd)}}(y) dy  \nonumber\\
  &\ = \ & \lambda_{Sp} + 4 (\sigma-1) \sin \left(\frac{1}{4} (2 \sigma+\pi -3)\right)+4 \sqrt{2} \sin \left(\frac{1}{4} (3-2 \sigma)\right) \nonumber\\
  & & \ \ \ \ -2 \sqrt{2} (\sigma-2) \sin \left(\frac{1}{4} (-2 \sigma+\pi +3)\right) \tan \left(\frac{\sigma-1}{\sqrt{2}}\right)
 \label{SOOddScaling}
\end{eqnarray}
for $\mathcal{G} = {\rm SO(odd)}$, completing the proof.
\end{proof}

\newpage

\section{Reduction to Finite-Dimensional Optimization All $\sigma$} \label{findimreductionsection}
Initially, we only know $g \in L^{2}(-\sigma,\sigma)$. So, our optimization problem occurs at first over an infinite-dimensional space. In this section, we reduce the problem to a finite dimensional optimization problem over $\R^{n}$ for some $n$, which we find explicitly, as a function of $\sigma$, in Corollary \ref{DimensionOfOptimizationCor}.  \\

The setup in this section works for all $\sigma \ge .5$. However, for smaller values of $\sigma$, we have already found the optimal functions explicitly. 

We accomplish this in the following way. 
\begin{itemize}
  \item As in the previous sections, differentiate the integral equation \eqref{eq:0} to arrive at two systems of location--specific delay differential equations
  \item We show, using induction and mono-invariants, that each system always resolves to a non-trivial ODE on two intervals. So, on those interval, the optimal $g$ falls into a finite-dimensional family of solutions. 
  \item We show, using the system, that the solution on those two intervals completely determines a finite-dimensonal family in which our optimal $g$ lives. 
\end{itemize}
The main results of this section are Theorems \ref{ExplicitOutsideDiffEqTheorem} and \ref{GOnWholeIntervalThm}. Theorem  \ref{ExplicitOutsideDiffEqTheorem} provides an explicit linear ODE that describes $g$ towards the outermost end of the interval $[-\sigma,\sigma]$. Theorem \ref{GOnWholeIntervalThm} then describes the optimal $g$ on all of $[-\sigma,\sigma]$ in terms of the finite-dimensional family of solutions to the ODE from Theorem \ref{ExplicitOutsideDiffEqTheorem}.  

In the following subsection, we present the general system of delay differential equations, and show a method for reduction to an ODE on the outermost intervals. 
\subsection{Presentation and Reduction of the General System}
Before presenting the location-specific delay-differential equations, we will first describe the intervals into which we subdivide $[-\sigma,\sigma]$. In the examples $.5 < \sigma < 1$ and $1 < \sigma < 1.5$, we solved two systems to find our family for the optimal $g$. We continue with this approach here. There are two cases. \\

If $k \le \sigma < k + 1/2$ for some positive integer $k$, then our first system of intervals is 
  \begin{align*}
    I_{0}:\ &= \ [2k - \sigma, \sigma] \\
    I_{1}:\ &= \ I_{1} - 1 \\
    &\vdots \\
    I_{2k}:\ &= \ I_{1} - 2k, 
  \end{align*}
and our second system is 
\begin{align*}
  J_{0} \ &:= \ [\sigma - 1, 2k - \sigma] \\
  J_{1} \ &:= \ J_{1} - 1 \\
  &\vdots \\
  J_{2k-1} \ &:= \ J_{1} - (2k - 1). 
\end{align*}

If $k + 1/2 \le \sigma < k + 1$ for some positive integer $k$, then our second system of intervals is
\begin{align*}
  J_{0}' \ &:= \ [2k + 1 - \sigma, \sigma] \\
  J_{1}' \ &:= J_{1} - 1 \\
  &\vdots \\
  J_{2k + 1}' \ &:= \ J_{1} - (2k + 1), 
\end{align*}
and the first is
\begin{align*}
  I_{0}' \ &:= \ [\sigma -1, 2k + 1 - \sigma] \\
  I_{1}' \ &:= \ I_{1} - 1 \\
  &\vdots \\
  I_{2k}' \ &:= \ I_{1} - 2k. 
\end{align*}

Note that when $\sigma$ is an integer or a half-integer, one system of intervals is just a collection of finitely many points, each contained in the other system. \\

For the remainder of this section, we will refer to the collection of $I_{i}$ (or $I_{i}')$ as the ``first system'' and the collection of $J_{i}$ (or $J_{i}')$ as the ``second system''. 
\begin{lemma}
  Let $\sigma > 1$ and suppose $k \le \sigma < k + 1/2$ for some positive integer $k$. Then, the optimal $g$ satisfies 
 \begin{align}
   g'(x) - \beta_{\mathcal{G}}g(x - 1) \ &= \ 0 \tag{I.1} \label{first1} \\
   g'(x-1) + \beta_{\mathcal{G}}g(x) - \beta_{\mathcal{G}}g(x-2) \ &= \ 0 \tag{I.2} \label{first2}\\
    g'(x-2) + \beta_{\mathcal{G}}g(x) - \beta_{\mathcal{G}}g(x-3) \ &= \ 0 \tag{I.3} \\
    &\vdots \notag \\
    g'(x - (2k -1)) + \beta_{\mathcal{G}} g(x - (2k-2)) - \beta_{\mathcal{G}} g(x - 2k) \ &= \ 0 \tag{I.$2k$} \label{first2k}\\
    \label{first2k1} g'(x - 2k) + \beta_{\mathcal{G}}g(x - (2k-1)) \ &= \ 0 \tag{I.$2k + 1$} 
  \end{align}
for $x \in I_{0}$, and satisfies
  \begin{align*}
    g'(x) - \beta_{\mathcal{G}}g(x - 1) \ &= \ 0 \tag{J.1a} \label{second1a} \\
    g'(x-1) + \beta_{\mathcal{G}}g(x) - \beta_{\mathcal{G}}g(x-2) \ &= \ 0 \tag{J.2a} \label{second2a} \\
    g'(x-2) + \beta_{\mathcal{G}}g(x) - \beta_{\mathcal{G}}g(x-3) \ &= \ 0 \tag{J.3a} \\
    &\vdots \notag \\
    g'(x - (2k -2))  + \beta_{\mathcal{G}} g(x - (2k -3)) - \beta_{\mathcal{G}} g(x - (2k-1)) \ &= \ 0 \tag{J.$2k-1$a} \label{second2k1a}\\
    g'(x - (2k -1)) + \beta_{\mathcal{G}}g(x - (2k -2)) \ &= \ 0 \tag{J.$2k$a}\label{second2ka}
  \end{align*}
  for $x \in J_{0}$. If $k + 1/2 \le \sigma < k + 1$ for some positive integer $k$, then the optimal $g$ satisfies \eqref{first1} - \eqref{first2k1} for $x \in I_{0}'$ and   \begin{align*}
    g'(x) - \beta_{\mathcal{G}}g(x - 1) \ &= \ 0 \tag{J.1b} \label{second1b}\\
    g'(x-1) + \beta_{\mathcal{G}}g(x) - \beta_{\mathcal{G}}g(x-2) \ &= \ 0 \tag{J.2b} \label{second2b}\\
    g'(x-2) + \beta_{\mathcal{G}}g(x) - \beta_{\mathcal{G}}g(x-3) \ &= \ 0 \tag{J.3b} \\
    &\vdots \notag \\
    g'(x - 2k)  + \beta_{\mathcal{G}} g(x - (2k -1)) - \beta_{\mathcal{G}} g(x - (2k+1)) \ &= \ 0 \tag{J.$2k+1$b} \label{second2k1b}\\
    g'(x - (2k + 1)) + \beta_{\mathcal{G}}g(x - 2k) \ &= \ 0 \tag{J.$2k+2$b} \label{second2k2b}
  \end{align*}
  for $x \in J_{0}'$. 
  \label{GeneralSystemLemma}
\end{lemma}
\begin{proof}
Note that for $|x| < \sigma - 1$ the optimal $g$ satisfies 
\begin{equation}
  g(x) + \beta_{\mathcal{G}}\int_{x-1}^{x + 1} g(y) \ dy + \alpha_{\mathcal{G}} \int_{-\sigma}^{\sigma} g(y) \ dy \ = \ 1, 
  \label{OptimalConditionMiddleSigma}
\end{equation}
and for $x \ge \sigma - 1$ we have 
\begin{equation}
  g(x) + \beta_{\mathcal{G}} \int_{x-1}^{\sigma} g(y) \ dy + \alpha_{\mathcal{G}}\int_{-\sigma}^{\sigma} g(y) \ dy \ = \ 1, 
  \label{OptimalConditionsRightSigma}
\end{equation}
and for $x \le -(\sigma - 1)$ this condition becomes 
\begin{equation}
  g(x) + \beta_{\mathcal{G}}\int_{-\sigma}^{x + 1} g(y) \ dy + \alpha_{\mathcal{G}} \int_{-\sigma}^{\sigma} g(y) \ dy  \ = \ 1.
  \label{OptimalConditionLeftSigma}
\end{equation}
Equations \eqref{first2} - \eqref{first2k}, \eqref{second2a} - \eqref{second2k1a}, and \eqref{second2b} - \eqref{second2k1b} come from differentiating \eqref{OptimalConditionMiddleSigma} under the integral sign, while equations \eqref{first1}, \eqref{first2k1}, \eqref{second1a}, \eqref{second2ka}, \eqref{second1b}, and \eqref{second2k2b} come from applying Liebniz's rule to either \eqref{OptimalConditionsRightSigma} or \eqref{OptimalConditionLeftSigma}. 
\end{proof}

\bigskip

\begin{remark}
  In establishing this sytem, we do not yet use that $g$ is even. We will only use this fact at the very end of our proof of Theorem \ref{GOnWholeIntervalThm}, after we have solved for $g$ on $[0,\sigma]$. 
\end{remark} 

\bigskip

In general, this system can be reduced to an ODE for $g$ in $I_{0}$ and $J_{0}$ or $I_{0}'$ and $J_{0}'$, whichever are the two outermost intervals. To do so, we create the following definitions. 

\begin{definition}
  In the systems of delay differential equations, \eqref{first1} -- \eqref{first2k1}, \eqref{second1a} -- \eqref{second2ka}, or \eqref{second1b} -- \eqref{second2k2b}, the \textit{$i^{\textrm{th}}$ equation} is the one labeled (I.$i$), (J.$i$a), or (J.$i$b). The \textit{final equation} is the last one listed in the order above, regardless of the value of $k$. 
\end{definition}

The following definition/algorithm entails manipulating a single equation, the first equation, in systematic ways, based on the other delay differential equations. 

\begin{definition}
  The \textit{current expression} or \textit{current equation} is the first equation after all manipulations until those in the current step have been executed. 
\end{definition}

In this language, the current expression starts out as 
\begin{equation}
  g'(x) - \beta_{\mathcal{G}}g(x - 1) \ = \ 0.
  \label{firsteqfornewlangexample}
\end{equation}
We claim that using the other equations, we can change the current expression from \eqref{firsteqfornewlangexample} to some non-trivial ODE that describes $g$ on $I_{1}, I_{1}', J_{1}$, or $J_{1}'$. \\

All the terms we will deal with are of the form $\beta g^{(m)}(x - r)$ for some $\beta \in \R$. For these terms, we use the following notation. 
\begin{definition}
  Let $T$ be a term of the form $\beta g^{(m)}(x - r)$ for some $\beta \in \R$. We say the \textit{integer degree}, $Z(T)$, of $T$, is $r$, the \textit{differential degree}, $D(T)$, of $T$ is $m$, and the \textit{full degree}, $F(T)$, of $T$ is $r + m$, the sum of the integer and differential degrees.  
\end{definition}

\begin{definition} \label{upathdef}
  Let $S$ be one of the systems of delay differential equations (\eqref{first1} -- \eqref{first2k1}, \eqref{second1a} -- \eqref{second2ka}, or \eqref{second1b} -- \eqref{second2k2b}). We define the \textit{U path} through the system as follows.  
  \begin{itemize}
    \item Differentiate the first equation. We arrive at 
      \begin{equation}
	g^{(2)}(x) - \beta_{\mathcal{G}}g'(x - 1) \ = \ 0. 
	\label{AfterDiffFirstEq}
      \end{equation}
    \item Apply the second equation to substitute $\beta_{\mathcal{G}}^{2}g(x) - \beta_{\mathcal{G}}^{2}g(x - 2)$ for $-\beta_{\mathcal{G}} g'(x-1)$. 
    \item Differentiate the current expression, which is now 
      \begin{equation}
	g^{(2)}(x) + \beta_{\mathcal{G}}^{2} - \beta_{\mathcal{G}}^{2}g(x-2),
	\label{SecondCurrExpress}
      \end{equation}
      and use the third equation to make a substitution for the integer degree two (and differential degree one) term. 
    \item Continue to differentiate the current expression and use the $r^{\textrm{th}}$ equation to make a substitution for the term $T$ such that $Z(T) \ = \ r-1$ and $D(T) \ = \ 1$ in terms of $g(x - (r + 1))$ and $g(x -(r - 1))$. Stop this process when we have used the final equation to substitute $\beta_{\mathcal{G}}^{q} g(x - l)$ for $\beta^{q+1}_{\mathcal{G}} g'(x - (l + 1))$ for some $q \in \Z$. This step in the $U$ path is called ``the turn''. 
    \item Note that all of the equations in our system (and their derivatives) can be used in the ``reverse direction'', that is, we are given that
      \begin{equation}
	g^{(m)}(x - r) \ = \ \frac{1}{\beta_{\mathcal{G}}}g^{(m + 1)}(x -(r-1)) +  g^{(m)}(x - (r-2))) 
	\label{GeneralReverseDirectionFormula}
      \end{equation}
      for any $m \in \Z_{\ge 0}$, whenever $1 < r < (\textrm{max integer appearing in the system})$. When $r = 1$, we have 
      \begin{equation}
	g^{(m)}(x - 1) = \frac{1}{\beta_{\mathcal{G}}} g^{(m + 1)}(x). 
	\label{ReverseDirForOne}
      \end{equation}
      Using such an equation in the ``reverse direction'' lowers integer degree using the general equations \eqref{GeneralReverseDirectionFormula} and \eqref{ReverseDirForOne}. \\

      In the case when $r$ is the largest integer appearing in the system, using the final equation 
      \begin{equation}
	g^{(m)}(x - r) = -\beta_{\mathcal{G}}g^{(m-1)}(x - (r -1))
	\label{ReverseDirAtEnd}
      \end{equation}
      already reduces the integer degree of every term. So, the process of trading integer degrees for differential degrees begins at the turn. \\

      With this observation, we can describe the final step of the $U$ path. After the turn, use the equations in the reverse direction to reduce the integer degree of all terms in the current expression until all have integer degree zero. From here, we must show that the ODE we have created is non-trivial. 
  \end{itemize}
\end{definition}

Finally, we introduce one more piece of terminology. 
\begin{definition} \label{descendentdef}
  Let $T$ be a term in our current expression. If we substitute $S_{1}$ and $S_{2}$ for $T$ via one of our delay differential equations (or a derivative thereof), in the course of the $U$-path, we say that $S_{1}$ and $S_{2}$ are \textit{direct U-descendents} of $T$. We say a descendent of a descendent of $T$ is also a descendent of $T$, but is not a direct descendent unless it is obtained by a single substitution. 

\end{definition}

Observing equations \eqref{first1} -- \eqref{first2k1}, \eqref{second1a} -- \eqref{second2ka}, \eqref{second1b} -- \eqref{second2k2b}, and \eqref{GeneralReverseDirectionFormula}, we note that any direct descendent $S$ of a term $T$ either realizes $F(S) = F(T)$ or $F(S) = F(T) - 2$. 
\begin{definition}
  If $S$ is a direct descendent of $T$ and $F(S) = F(T)$, then we say $S$ is a \textit{high descendent} of $T$. If $S$ is a direct descedent of $T$ and $F(S) = F(T) - 2$, we say $S$ is a \textit{low descendent} of $T$. 
\end{definition}
Now, we can prove the mono-invariance of total degree on the $U$-path. Colloquially, in trading integer degrees for differential degrees or vice-versa, by moving forwards and backwards in the $U$-path we make a reasonable trade. \\

The following two results are technically not necessary for the proof of Theorem \ref{GOnWholeIntervalThm}. We include it because it provides motivation for our method and intuition for why it works. 

\begin{lemma}
  Suppose that $S$ is a $U$-descendent of $T$. Then, $F(S) \le F(T)$. Moreover, $F(S)$ has the same even/odd parity as $F(T)$. 
  \label{MonoInvTotalDegree}
\end{lemma}
\begin{proof}
  First, we examine our path before the turn. Consider a term $T$ of the form $\beta g^{(m)}(x - r)$. Before the turn, we substitute
  \begin{equation}
    -\beta_{\mathcal{G}}g^{(m-1)}(x - (r -1)) + \beta_{\mathcal{G}}gg^{(m-1)}g(x - (r + 1)). 
    \label{BeforeTurnSubstitution}
  \end{equation}
 for $\beta g^{(m)}(x - r)$. Call these terms $S_{1}, S_{2}$ in order of left to right. Observing that equation shows $Z(S_{1}) = Z(T) - 1, \ D(S_{1}) = D(T) - 1$, and $F(S_{1}) = F(T) - 2$. On the other hand, $Z(S_{2}) = Z(T) + 1, \ D(S_{2}) = D(T) - 1$, and $F(S_2) = F(T)$. \\

  In either case, we have $F(S_{i}) \le F(T)$, and the two total degrees are congruent mod 2. \\

  At the turn, there is only one immediate descendent of the term of highest integer degree. However, in this case we have $F(S) = F(T) - 2$, a strict decrease. \\

  After the turn, we want to reduce the integer degree of terms $T$ of the form $g^{(m)}(x - r)$. Examining the righthand side of \eqref{GeneralReverseDirectionFormula}, we see that the two terms realize $Z(S_{1}) = Z(T) - 1,\ D(S_{1}) = D(T) - 1$ and $F(S_{1}) = F(T)$. Also, $Z(S_{2}) = Z(T) - 2, \ D(S_{2}) = D(T)$, which similarly satisfies both the monotonicity and congruence requirements. 
\end{proof}
Now, we show that the $U$-path resolves to a non-trivial linear ODE.
\begin{proposition}
  The $U$-path resolves any of the systems \eqref{first1} -- \eqref{first2k1}, \eqref{second1a} -- \eqref{second2ka}, or \eqref{second1b} -- \eqref{second2k2b} to a non-trivial linear ordinary differential equation that describes $g$ on $I_{0}$ and $J_{0}$ or $I_{0}'$ and $J_{0}'$.
  \label{ReducesToODE}
\end{proposition}
\begin{proof}
  As we begin with the first equation, we will examine the $U$-descendants of $g'(x)$ and the $U$-descendants of $g(x - 1)$. \\

  From Definition \ref{upathdef}, it is clear that there is only one descendent of the $g(x)$ term. Suppose $k$ is the largest integer in our system. Then, there are $k+ 1$ equations. In the forwards direction and at the turn, we differentiate the current expression $k$ times, once before applying each of the last $k$ equations in our system. So, the descendent of $g'(x)$ is $g^{(k + 1)}(x)$.  To show the resulting ODE is non-trivial, it suffices to show this term cannot cancel with any descendants of the $\beta_{\mathcal{G}} g(x-1)$ term. We show the descendents of $\beta_{\mathcal{G}} g(x-1)$ have lower differential degree. \\



  In the forwards direction, $g(x - 1)$ initially has total degree equal to that of $g'(x)$. During the forwards direction, when we differentiate $g(x - r)$ and use the $(r+ 1)^{\textrm{st}}$ equation for substitution, we end up with a term of higher integer degree but equal total degree and one term with lower integer degree and lower total degree. Call the term of equal total degree a high descendent and the one of lower total degree a low descendent. \\

  From Lemma \ref{MonoInvTotalDegree}, we know that no descendants of any low descendent reach degree $k + 1$, since the degree of one of their ancestors dips strictly below that of the current descendent of $g'(x)$. It therefore suffices to trace the descendants of the sequence of high descendants. \\

  This sequence of high descendents progresses as $g(x -1), g(x - 2), \dots , g(x - k)$. When we reach $g(x - k)$, the total degree of the $g(x)$ descendent is $k$, as we have differentiated the current expression $k$ times. We then differentiate the current expression again. Applying the final epxression in our system gives us only a low descendent of $g'(x - k)$, namely $g(x - (k - 1))$. This has total degree two less than $g'(x)$ descendent. By Lemma \ref{MonoInvTotalDegree}, none of its descendents can recover this difference. \\

  After the turn in the $U$-path, we no longer differentiate the current expression. So, we can say that our final expression, the ODE, in simplest terms, has a term $g^{(k+1)}(x)$ with no other terms of equal total degree. The ODE is therefore non-trivial. 
\end{proof}
\subsection{An Explicit ODE On Outermost Intervals for All $\sigma$}
We can explicitly find the differential equation to which our system resolves. In order to do so, we prove two lemmata. 
\begin{lemma}\label{AfterTurnCurrentLemma}
  After the turn in the $U$-path, the current expression is
  \begin{equation}\label{AfterTurnCurrentEqn}
    g^{(k + 1)}(x) + \sum_{m=0}^{k-1} \beta_{\mathcal{G}}^{m+2}g^{(k-1-m)}(x - m),
  \end{equation}
  where $k$ is the largest integer appearing in our system.
\end{lemma}
\begin{proof}
  We proceed by induction, starting with $k = 1$. We begin by resolving the system 
  \begin{align}
    g'(x) - \beta_{\mathcal{G}}g(x-1) \ &= \ 0 \label{firstinbasecase} \\
    g'(x -1) + \beta_{\mathcal{G}}(x) \ &= \ 0 \label{secondinbasecase}. 
  \end{align}
 We differentiate \eqref{firstinbasecase}, so that our current expression is
 \begin{equation}
   g^{(2)}(x) - \beta_{\mathcal{G}}g'(x -1) \ = \ 0.
   \label{Afterdiffbasecase}
 \end{equation}
 Then, we execute the turn using \eqref{secondinbasecase} to arrive at the expression 
 \begin{equation}
   g^{(2)}(x) + \beta_{\mathcal{G}}^{2}g(x) \ = \ 0,
   \label{basecaseholds}
 \end{equation}
 which shows our base case holds. \\

 For the inductive step, note that before the turn, the $U$-path for the system with largest integer $k$ agrees with the $U$-path for the system with largest integer $k + 1$ until a term of integer degree $k + 1$ appears. The difference at this point is that after the $k^{\textrm{th}}$ differentiation, the term $\beta_{\mathcal{G}}^{k}g'(x - k)$ produces a high descendent, $-\beta_{\mathcal{G}}^{k + 1}g(x - (k + 1))$, in addition to its low descendent $\beta_{\mathcal{G}}^{k + 1}g(x - (k - 1))$, as the support of our $g$ is larger. So, assuming our inductive hypothesis, after $k$ differentiations and substitutions, the current expression is 
 \begin{equation}
   g^{(k + 1)}(x) + \left( \sum_{m=0}^{k-1}\beta_{\mathcal{G}}^{m + 2} g^{(k -1 - m)}(x - m) \right) - \beta_{\mathcal{G}}^{k+1}g(x - (k + 1)). 
   \label{CurrentExpressionPreTurnIndStep}
 \end{equation}
 We then differentiate \eqref{CurrentExpressionPreTurnIndStep}, and make the substitution given by 
 \[ g'(x - (k + 1)) + \beta_{\mathcal{G}}g(x - k) \ = \ 0 \] 
 to arrive at the new current expression of 
 \begin{equation}
   g^{(k + 2)}(x) + \sum_{m=0}^{k} \beta_{\mathcal{G}}^{m + 2} g^{(k - m)}(x - m)
   \label{IndStepCompleted}
 \end{equation}
 which completes the inductive step. 
\end{proof}

Each of the terms in \eqref{AfterTurnCurrentEqn} may be resolved as a linear function of $g(x)$ and its derivatives. In the following lemmata, we provide an explicit formula for these integer degree zero terms. 

\begin{lemma}
  For the optimal $g$, we have 
  \begin{equation}
    g^{(r)}(x - m) = \sum_{n + k = m} {n \choose k} \beta_{\mathcal{G}}^{k - n}g^{(r + n-k)}(x)
    \label{ResolvedLinearTermEqn}
  \end{equation}
  \label{ResolvedLinearTermLemma}
wherever $g(x - m)$ is smooth in $[-\sigma,\sigma]$. 
\end{lemma}
\begin{proof}
	When resolving these terms in the backwards direction of the $U$-path, we use \eqref{GeneralReverseDirectionFormula}. From this equation, we know that each term can be expressed as a function of terms of integer degree zero. Throughout, Figure \ref{fig:PascalFigure}, a resolution of $g(x-5)$ will be a helpful reference. 
  \\

  Given a term $T$, we place $T$ on the lattice $\Z \times \Z$. Place $T$ at $(Z(T), D(T))$. So, in this system, the term $g^{(r)}(x - m)$ begins at the point $(m,r)$. Then, using \eqref{GeneralReverseDirectionFormula}, we see that each term ``branches'' into its direct descendants. Unless a term has integer degree one, it will have two direct descendants. The high descendent, $S_{1}$, will have $Z(S_{1})  =  Z(T) - 1,  D(S_{1})  =  D(T) + 1$. The low descendent, $S_{2}$, will have $Z(S_{2})  =  Z(T) - 2,  D(S_{2})  =  D(T)$. If $Z(T) = 1$, there is only a high descendent. \\

  Motivated by the lattice representation, we say moving from a term to a high descendent is a \textit{diagonal step} and moving to a low descendent is a \textit{horizontal step}. \\

  Let $T^*$ be an integer degree zero descendent of $T = g^{(r)}(x - m)$. Then, examining \eqref{GeneralReverseDirectionFormula} shows that while $Z(T^*) = 0$, $r \le D(T^*) \le r + m$, say $D(T^*) = r + d$, i.e. each path from $T$ to $T^*$ takes $d$ diagonal steps. Diagonal steps reduce integer degree by one and horizontal steps reduce integer degree by two, so $d + 2h = (d + h) + h = m$. The total number of steps is $d+h$, and there are ${d + h \choose d} = {d + h \choose h}$. such paths. \\

  \begin{figure}[H]
    \centering
    \includegraphics[scale=0.5]{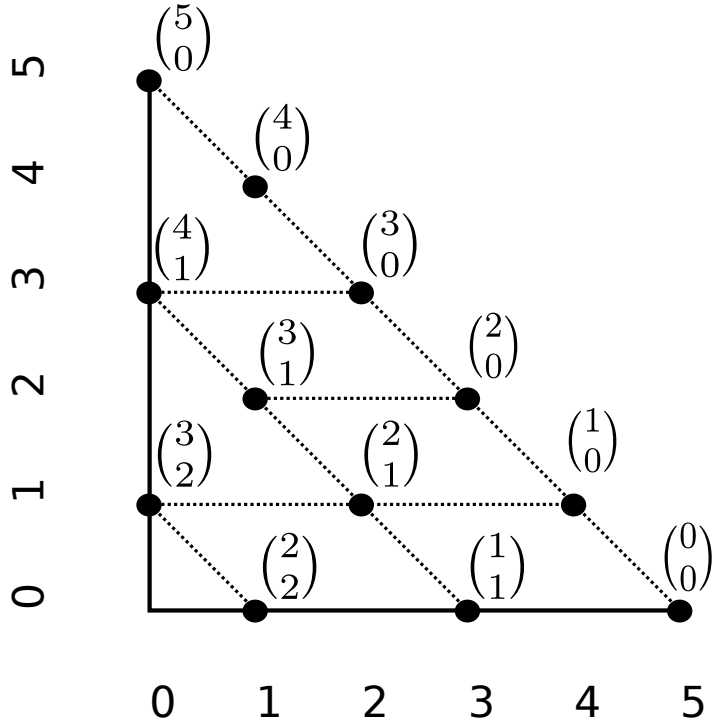}
    \caption{Reduction of the term $g(x-5)$ and the correspondence with Pascal's triangle.}
    \label{fig:PascalFigure}
  \end{figure}
  
  We now translate back into terms in our expression. Examining \eqref{GeneralReverseDirectionFormula}, we see that each diagonal step contributes a factor of $\beta_{\mathcal{G}}^{-1}$. Letting $d + h = n$ and $h = k$ gives \eqref{ResolvedLinearTermEqn}. 
\end{proof}

Lemmata \ref{AfterTurnCurrentLemma} and \ref{ResolvedLinearTermLemma} establish the following.  
\begin{theorem}
  Suppose $k \le \sigma < k + 1/2$ for some positive integer $k$. Then, the optimal $g$ satisfies 
  \begin{align}
    g^{(2k + 1)}(x) + \sum_{m=0}^{2k-1} \sum_{n + r = m} {n \choose r} \beta_{\mathcal{G}}^{m+ r - n + 2}g^{((2k - 1) - m + n - r)}(x) \ &= \ 0, &x \in (\sigma - 2(\sigma - k), \sigma) \label{Case11OutsideDiffEq}\\
    g^{(2k)}(x) + \sum_{m=0}^{2k -2} \sum_{n + r = m}^{} {n \choose r} \beta_{\mathcal{G}}^{m + r - n + 2}g^{( (2k - 2) - m + n - r)}(x) \ &= \ 0, &x \in (\sigma - 1, \sigma - 2(\sigma - k)) \label{Case12OutsideDiffEq}. 
    \end{align}
    If $k + 1/2 \le \sigma < k + 1$ for some positive integer $k$, then the optimal $g$ satisfies
    \begin{align}
      g^{(2k + 1)}(x) + \sum_{m=0}^{2k-1} \sum_{n + r = m} {n \choose r} \beta_{\mathcal{G}}^{m+ r - n + 2}g^{((2k - 1) - m + n - r)}(x) \ &= \ 0, &x \in (\sigma - 1, 2k + 1 - \sigma) \label{Case21OutsideDiffEq} \\
      g^{(2k + 2)}(x) + \sum_{m=0}^{2k} \sum_{n + r = m}^{} {n \choose r} \beta_{\mathcal{G}}^{m + r - n + 2}g^{( (2k) - m + n - r)}(x) \ &= \ 0, &x \in (2k + 1 - \sigma, \sigma) \label{Case22OutsideDiffEq}. 
    \end{align}
  \label{ExplicitOutsideDiffEqTheorem}
\end{theorem}
\begin{proof}
  In the first case, when $k \le \sigma < k + 1/2$, the largest integer appearing in our first system is $2k$ and the largest integer appearing in our second system is $2k - 1$. In the second case, when $k + 1/2 < \sigma < k + 1$, the largest integer appearing in our first system is still $2k$, but the largest integer appearing in our second system is $2k + 1$. With this established, we simply apply lemmata \ref{AfterTurnCurrentLemma} and \ref{ResolvedLinearTermEqn}, and we have the above result almost everywhere in our specified intervals. 
  
\end{proof}

\subsection{Finite Dimensional Families of Solutions for All Intervals, All $\sigma$}

Examining the systems of delay differential equations and the fact that $g$ must be even, it is clear that knowing $g$ on the two outside intervals completely determines $g$ on $[-\sigma,\sigma]$. However, we can express the values of $g$ on the inner intervals as a function of $g$ on the outer intervals. Lemma \ref{ResolvedLinearTermLemma} will allow us to find values of $g$ on diffferent intervals as a function of values of $g$ on the outermost interval. 
\begin{theorem} 
  Let $k \le \sigma < k + 1/2$ for some positive integer, $k$ and $g$ be optimal. Then 
  \begin{equation}
    g\vert_{I_{j}}(x) \ = \ 
    \begin{cases}
       \sum_{n + k = j}^{} {n \choose k}\beta_{\mathcal{G}}^{k-n}g\vert_{I_{0}}^{(n - k)}(|x| + j) \ &1 \le j \le k\\
       g\vert_{I_{2k - j}}(x) &k+1 \le j \le 2k
    \end{cases}
    \label{Case1OptGOnIj}
  \end{equation}
  and  
  \begin{equation}
     g\vert_{J_{j}} \ = \ 
    \begin{cases}
      \sum_{n + k = j}^{} {n \choose k}\beta_{\mathcal{G}}^{k-n}g\vert_{J_{0}}^{(n - k)}(|x| + j). \ &1 \le j \le k-1 \\
      g\vert_{I_{2k-1-j}} \ &k \le j \le 2k-1. 
    \end{cases}
    \label{Case2OptGOnJj}
  \end{equation}
  If $k + 1/2 \le \sigma < k$, then 
  \begin{equation}
    g\vert_{I_{j}'}(x) \ = \ 
    \begin{cases}
       \sum_{n + k = j}^{} {n \choose k}\beta_{\mathcal{G}}^{k-n}g\vert_{I_{0}'}^{(n - k)}(|x| + j) \ &1 \le j \le k \\
       g\vert_{I'_{2k-j}} \ &k+1 \le j \le 2k
    \end{cases}
    \label{Case1OptGOnIjPrime}
  \end{equation}
  and 
  \begin{equation}
    g\vert_{J_{j}'}(x) \ = \ 
    \begin{cases}
      \sum_{n + k = j}^{} {n \choose k}\beta_{\mathcal{G}}^{k-n}g\vert_{J_{0}'}^{(n - k)}(|x| + j). \ &1 \le j \le k \\
      g\vert_{J'_{2k + 1 - j}} \ & k + 1 \le 2k + 1. 
    \end{cases}
    \label{Case2OptGOnJjPrime}
  \end{equation}
  \label{GOnWholeIntervalThm}
\end{theorem}

\begin{proof}
	Again, we are resolving terms using \eqref{GeneralReverseDirectionFormula}. Argue exactly as in the proof of Lemma \ref{ResolvedLinearTermLemma}, to solve for $g$ on the innermost intervals as a function of $g$ on the outermost intervals.  
\end{proof}

\begin{corollary}\label{DimensionOfOptimizationCor}
  Theorem \ref{ExplicitOutsideDiffEqTheorem} reduces the problem of finding the optimal $g$ over $\R^{n}$, where 
  \begin{equation}
    n \ = \
    \begin{cases}
      4k + 1 \ & k < \sigma < k + 1/2 \\
      4k + 3 \ & k + 1/2 < \sigma < k + 1 \\
      2k \ &k \ = \ \sigma \\
      2k + 1 \ &k + 1/2 \ = \ \sigma. 
    \end{cases}
    \label{DimensionOfOptimizationEqn}
  \end{equation}
\end{corollary}

\begin{proof}
  Without initial conditions, the solution to an $m^{\textrm{th}}$ order differential equation is an $m$-parameter family of solutions. When $\sigma$ is neither an integer nor a half-integer, there are two systems to solve. Say these systems have degree $d_{1}, d_{2}$. The total number of free real parameters is at least $d_{1} + d_{2}$. Lemma \ref{GOnWholeIntervalThm} shows us that there are no more than $d_{1} + d_{2}$ free real parameters. The dimensions above are $d_{1} + d_{2}$. The dimension is lower on integers and half integers because in those cases, one of our systems of intervals is trivial. 
\end{proof}

\begin{corollary}\label{NotManyBadPointsCor}
  Within $(-\sigma,\sigma)$, the optimal $g$ has at most $m$ points of non-differentiability, where
  \begin{equation}
    m = 
    \begin{cases}
      4k + 1 \ &k < \sigma < k + 1/2 \\
      4k + 3 \ &k + 1/2 < \sigma < k + 1 \\ 
      2k - 1 \ &k = \sigma \\
      2k + 1 \ &k + 1/2 = \sigma
    \end{cases}
    \label{NumberOfBadPointsEqn}
  \end{equation}
These points are either endpoints of the $I_{i}, J_{i}, I_{i}', J_{i}'$, or zero. Within the interiors of these intervals (except possibly at zero), the optimal $g$ is real-analytic. 
\end{corollary}
\begin{proof}
  Theorems \ref{ExplicitOutsideDiffEqTheorem} and \ref{GOnWholeIntervalThm} establish that, with the exception of an interval containing zero, the optimal $g$ is completely differentiable in the interior of each of the intervals. Because of the absolute value introduced to make $g$ even (as seen in Theorem \ref{GOnWholeIntervalThm}), 0 may also be a point of non-differentiability. \\

  We can count righthand endpoints of intervals. The $m$ above are generated by 
  \[ m =(\# \textrm{intervals in system} - 1) + 1 \] 
  if zero is not a righthand endpoint of an interval in the system. If it is (which occurs precisely when $\sigma$ is an integer), we use the formula 
  \[ m = (\# \textrm{intervals in system} - 1).\] 
  Real analyticity of $g$ follows from the  Cauchy-Kowalevski Theorem (see \cite{Wal}).
\end{proof}

%
%

\newpage

\section{Extension to $\textrm{support}(\hat{\phi}) \subset [-4,4]$} \label{-44section}
When $1.5 < \sigma < 2$, the intervals of importance are
\begin{equation}
	\begin{aligned}
		J_0' \ &= \ [3 - \sigma, \sigma] \\
		J_1' \ &= \ [2 - \sigma, \sigma - 1] 
	\end{aligned}
	\label{JPrime-44Intervals}
\end{equation}
and 
\begin{equation}
	\begin{aligned}
		I_0' \ &= \ [\sigma - 1, 3 - \sigma] \\
		I_1' \ &= \ [0, 2 - \sigma] 
	\end{aligned}
	\label{IPrime-44Intervals}
\end{equation}

Theorem \ref{ExplicitOutsideDiffEqTheorem} tells us that on $J_0'$, the optimal $g$ (for all three cases, $\textrm{SO(Even), SO(Odd), Sp}$) is described by a fourth degree ODE and on $I_0'$, it is described by a third degree ODE. On $J_0'$, this ODE (for each of the three groups) is 
\begin{equation}
	g^{(4)}(x) + \frac{3}{4} g^{(2)}(x) + \frac{1}{16} g(x) \ = \ 0 
	\label{J0PrimeODE}
\end{equation}
and on $I_0'$, the ODE is 
\begin{equation}
	g^{(3)}(x) + \frac{1}{2} g'(x) \ = \ 0. 
	\label{I0PrimeODE}
\end{equation}
\\

Our first task is to reduce the size of this dimension seven problem. Note first that the ODE \eqref{I0PrimeODE} is the same as the ODE that describes the optimal $g$ on $I_0$ in the cases $1 < \sigma < 1.5$. We write the solution to this ODE as 
\[ 
	\frac{c_{1,\mathcal{G},\sigma}}{\sqrt{2}} \sin\left( \frac{x -1}{\sqrt{2}} \right) +  \frac{c_{2,\mathcal{G},\sigma}}{\sqrt{2}} \cos \left( \frac{x -1}{\sqrt{2}} \right)  + c_{3,\mathcal{G},\sigma} 
\]
and we observed that in fact $c_2$ was always zero. Though we did not need to do so when $1 < \sigma < 1.5$, we will show in this case that in fact $c_2$ is necessarily zero. 

\begin{lemma}
	For $\mathcal{G} = \textrm{SO(Even), SO(Odd), or Sp}$ and $1.5 < \sigma < 2$, on $I_0'$, the optimal $g$ is of the form 
	\[ 
		\frac{c_{1,\mathcal{G},\sigma}}{\sqrt{2}} \sin\left( \frac{x -1}{\sqrt{2}} \right) + c_{3,\mathcal{G},\sigma} 
	\]
	for real constants $c_{1,\mathcal{G},\sigma}$ and $c_{3,\mathcal{G},\sigma}$, i.e. $c_{2,\mathcal{G},\sigma} = 0$. 
	\label{c2alwayszero}
\end{lemma}

\begin{center}
		From here on, we omit all subscripts for notational clarity. 
	\end{center}

\begin{proof}
	This proof uses the symmetry that the optimal $g$ must be even. For $x \in I_1$, differentiating 
	\begin{equation}
		(I + K)(g) \ = \ 1 
		\label{IntConditionAgan}
	\end{equation}
	yields 
	\[ 
		g'(x) - \beta g(x-1) +  \beta g(x+1) \ = \ 0
	\]
	which becomes 
	\begin{equation}
		g'(x) -  \beta g(1-x) + \beta g(x+1) \ = \ 0 
		\label{SymmetryExpressionEliminateC2}
	\end{equation}
	because $g$ is even. \\

	In $I_1$ (or $I_1'$), we can compute $g$ based on \eqref{IntConditionAgan}, giving 
	\[ 
		g(x) \vert_{I_0} \ = \ \beta^{-1} \left( \frac{c_1}{2} \cos \left( \frac{x}{\sqrt{2}} \right) - \frac{c_2}{2} \sin \left( \frac{x}{\sqrt{2}} \right) \right)
	\]
	and we can write \eqref{SymmetryExpressionEliminateC2} as 
	\[ 
		\frac{c_1}{\sqrt{2}} \sin \left( \frac{x}{\sqrt{2}} \right) \left( 2 \beta - (2 \beta)^{-1} \right)  - \frac{c_2}{2 \beta \sqrt{2}} \cos \left( \frac{x}{\sqrt{2}} \right) \ = \ 0. 
	\]
	We note that this implies $c_2 = 0$, but $c_1$ is not necessarily zero, since $2\beta - (2 \beta)^{-1}$ is zero for either possible $\beta$, namely $\pm 1/2$. 
\end{proof}

On $J_0'$, our optimal $g$ is described by 
\begin{align*}
	&c_4 \sin \left(\frac{1}{2} \sqrt{\frac{1}{2}
   \left(3+\sqrt{5}\right)} (x-1)\right) +  
   c_5 \cos \left(\frac{1}{2} \sqrt{\frac{1}{2}\left(3+\sqrt{5}\right)} (x-1)\right) + \\ 
   &c_6 \sin \left(\frac{1}{2} \sqrt{\frac{1}{2}
   \left(3-\sqrt{5}\right)} (x-1)\right) + c_7 \cos \left(\frac{1}{2} \sqrt{\frac{1}{2}
   \left(3-\sqrt{5}\right)} (x-1)\right).
\end{align*}

Letting
\begin{align*}
	\alpha_1 \ &= \ \frac{1}{2} \sqrt{\frac{1}{2} \left( 3 + \sqrt{5} \right)} \\
	\alpha_2 \ &= \ \frac{1}{2} \sqrt{\frac{1}{2} \left( 3 - \sqrt{5} \right)}
\end{align*}
we shorten the above to 
\begin{equation}
	g(x) \vert_{J_0'} \ = \ c_4 \sin (\alpha_1(x-1)) + c_5 \cos(\alpha_1(x-1)) + c_6 \sin (\alpha_2(x-1)) + c_7 \cos (\alpha_2 (x-1)). 
	\label{Revised4ParamFamily}
\end{equation}

As above, symmetry lets us establish: 
\begin{lemma}
	For $\mathcal{G} = \textrm{SO(Even), SO(Odd), or Sp}$ and $1.5 < \sigma < 2$, on $J_0'$, the optimal $g$, described by the family \eqref{Revised4ParamFamily}, satisfies 
	\begin{equation}
		\begin{aligned}
			c_5 \ &= \ c_4 \frac{\left( -\frac{\alpha_1^2}{\beta} + \beta - \alpha_1 \sin (\alpha_1) \right)}{\alpha_1 \cos \alpha_1} \\
			c_7 \ &= \ c_6	\frac{\left( -\frac{\alpha_2^2}{\beta} + \beta - \alpha_2 \sin (\alpha_2) \right)}{\alpha_2 \cos \alpha_2} 
		\end{aligned}
		\label{RelationsAmongc5c7}
	\end{equation}
	\label{Relationsc5c7Lemma}
	here $\beta$ is the same constant as before, 1/2 for $\mathcal{G} = \textrm{SO(Even)}, -1/2$ for $\mathcal{G} = \textrm{SO(Odd) or Sp}$. 
\end{lemma}

\begin{proof}
	This proof is similar to the proof of Lemma \ref{c2alwayszero}. We use the symmetry of $g$ and the optimality criterion $(I + K)(g) = 1$ to deduce the result. \\
	
	Differentiating $(I + K)(g) = 1$ shows that on $J_1'$, the optimal $g$ is given by 
	\begin{equation}
		\beta^{-1}\left( \alpha_1 c_4 \cos(\alpha_1 x) - \alpha_1 c_5 \sin(\alpha_1 x) + \alpha_2 c_6 \cos(\alpha_2 x) - \alpha_2 c_7 \sin (\alpha_2 x) \right)
		\label{gonJ_1Prime}.
	\end{equation}
	Again, for $x \in J_1'$, \eqref{SymmetryExpressionEliminateC2} holds, only this time $1 - x \in J_1'$. We compute 
	\begin{align*}
		g'(x) \ &= \ \beta^{-1}(-c_4 \alpha_1^2 \sin (\alpha_1 x) - c_5 \alpha_1^2 \cos (\alpha_1 x) - c_6 \alpha_2^2 \sin (\alpha_2 x) - c_7 \alpha_2^2 \cos (\alpha_2 x)) \\
		- \beta g(1 - x) \ &= \ (-1) (c_4 \alpha_1 \cos (\alpha_1 - \alpha_1 x) -  c_5 \alpha_1 \sin (\alpha_1 - \alpha_1 x) + c_6 \alpha_2 \cos (\alpha_2 - \alpha_2 x) -  c_7 \alpha_2 (\alpha_2 - \alpha_2 x)) \\
		\beta g(x + 1) \ &= \ \beta \left( c_4 \sin(\alpha_1 x) + c_5 \cos(\alpha_1 x) + c_6 \sin (\alpha_2 x) + c_7 \cos (\alpha_2 x) \right).
	\end{align*}

	After applying the angle addition formulas and grouping terms, we arrive at 
	\[ 
		\gamma_1 \sin (\alpha_1 x) + \gamma_2 \cos(\alpha_1 x) + \gamma_3 \sin(\alpha_2 x) + \gamma_4 \cos(\alpha_2 x) = 0.  
	\]

	However, the Wronskian of the equation \eqref{J0PrimeODE} is one, hence all $\gamma_i$ must be zero. The $\gamma_i$ are given by 
	\begin{align*}
		\gamma_1 \ &= \ c_4\left( -\frac{\alpha_1^2}{\beta} + \beta - \alpha_1 \sin (\alpha_1) \right) - c_5(\alpha_1 \cos (\alpha_1)) \\
		\gamma_2 \ &= \ c_5 \left(-\frac{\alpha_1^2}{\beta} + \beta + \alpha_1 \sin(\alpha_1)\right) - c_4 (\alpha_1 \cos (\alpha_1)) \\
		\gamma_3 \ &= \ c_6 \left( - \frac{\alpha_2^2}{\beta} + \beta - \alpha_2 \sin (\alpha_2) \right) - c_7 (\alpha_2 \cos (\alpha_2)) \\
		\gamma_4 \ &= \ c_7 \left( -\frac{\alpha_2^2}{\beta} + \beta + \alpha_2 \sin(\alpha_2) \right) - c_6 (\alpha_2 \cos(\alpha_2))
	\end{align*}
	and it turns out the matrix 
	\[ 
		\begin{pmatrix}
			-\frac{\alpha_1^2}{\beta} + \beta - \alpha_1 \sin (\alpha_1) & -(\alpha_1 \cos (\alpha_1)) & 0 & 0 \\
			-(\alpha_1 \cos (\alpha_1)) & -\frac{\alpha_1^2}{\beta} + \beta + \alpha_1 \sin(\alpha_1) & 0 & 0 \\
			0 & 0 & -\frac{\alpha_2^2}{\beta} + \beta - \alpha_2 \sin (\alpha_2) & -(\alpha_2 \cos (\alpha_2)) \\
			0 & 0 & -(\alpha_2 \cos (\alpha_2)) & -\frac{\alpha_2^2}{\beta} + \beta + \alpha_2 \sin(\alpha_2)
		\end{pmatrix}
	\]
	has rank two, so each block has rank one and $\gamma_1 = 0$ precisely when $\gamma_2 = 0$ and $\gamma_3 = 0$ precisely when $\gamma_4 = 0$. Solving $\gamma_1 = 0$ and $\gamma_3 = 0$ gives the result. 
\end{proof}

We have now reduced the seven-dimensional problem to a four-dimensional one. We have a piecewise description of the optimal $g$ as a function of four free parameters. Namely, 
\begin{align*}
	g\vert_{I_1'} \ &:= \ f_1(x,c_1) \\
	g\vert_{J_1'}	\ &:= \ f_2(x,c_4,c_6) \\
	g\vert_{I_0'} \ &:= \ f_3(x,c_1,c_3) \\
	g\vert_{J_0'} \ &:= \ f_4(x,c_4,c_6) 
\end{align*}

As in the cases $1 < \sigma < 1.5$, we solve for $c_1, c_3, c_4$, and $c_6$ by imposing necessary conditions on $g$ via four linear equations. In particular, for all groups these equations are: 
\begin{equation}
	\begin{aligned}
		f_1(2-\sigma,c_1) \ &= \ f_2(2-\sigma,c_4,c_6) \\
		f_2(\sigma - 1,c_4,c_6) \ &= \ f_3(\sigma - 1,c_1,c_3) \\
		f_3(3 - \sigma,c_1,c_3) \ &= \ f_4(3 - \sigma, c_4,c_6) \\
		(I + K)(g)(0) \ &= \ 1.
	\end{aligned}
	\label{EquationsToReduceProblem}
\end{equation}

The first three incorporate the requirement that the optimal $g$ is continuous. The final equations is the optimality condition. Neither the matrix nor its determinant is practical to write down. Here is a plot of the determinants for the groups SO(Even) and Sp (the Sp equations can be used to find the optimal function for SO(Odd) case).  

\begin{figure}[h]
    \centering
    \begin{subfigure}[b]{0.3\textwidth}
        \includegraphics[width=\textwidth]{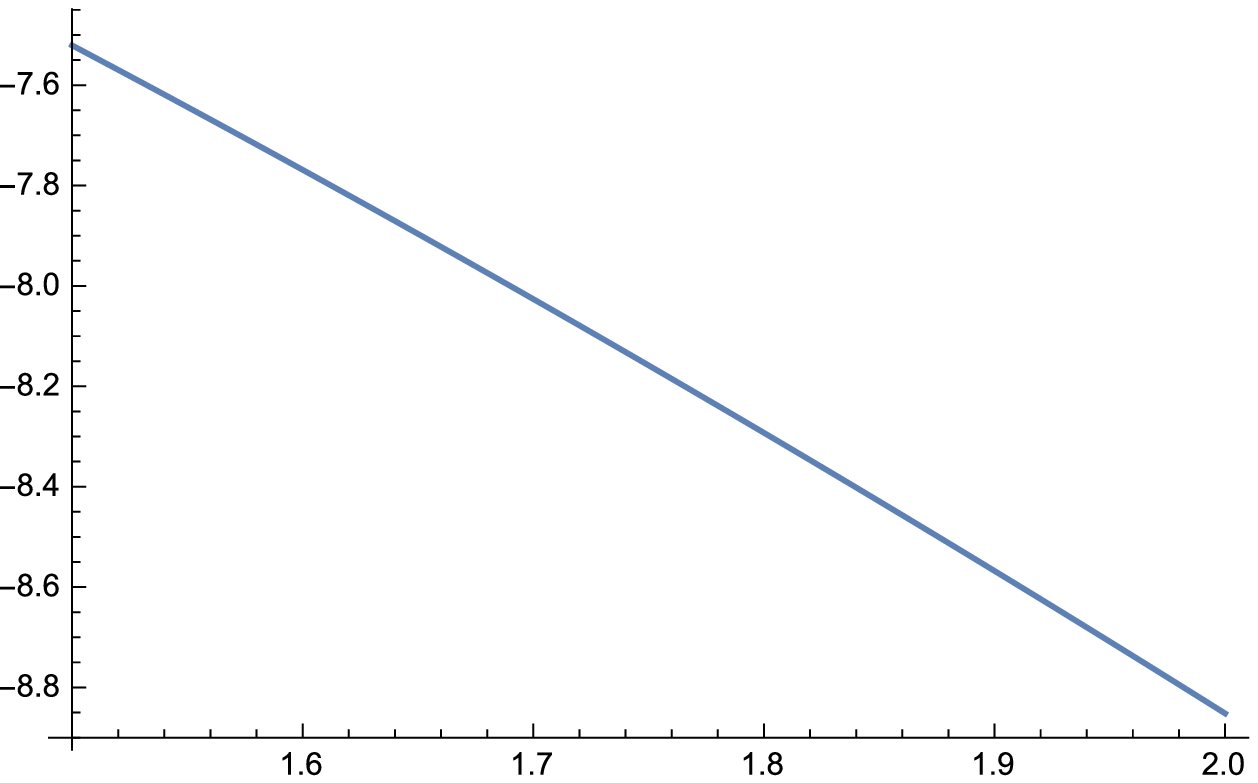}
        \caption{SO(Even)}
        \label{fig:soeven44det}
    \end{subfigure} \qquad \qquad \qquad 
    ~ 
    \begin{subfigure}[b]{0.3\textwidth}
        \includegraphics[width=\textwidth]{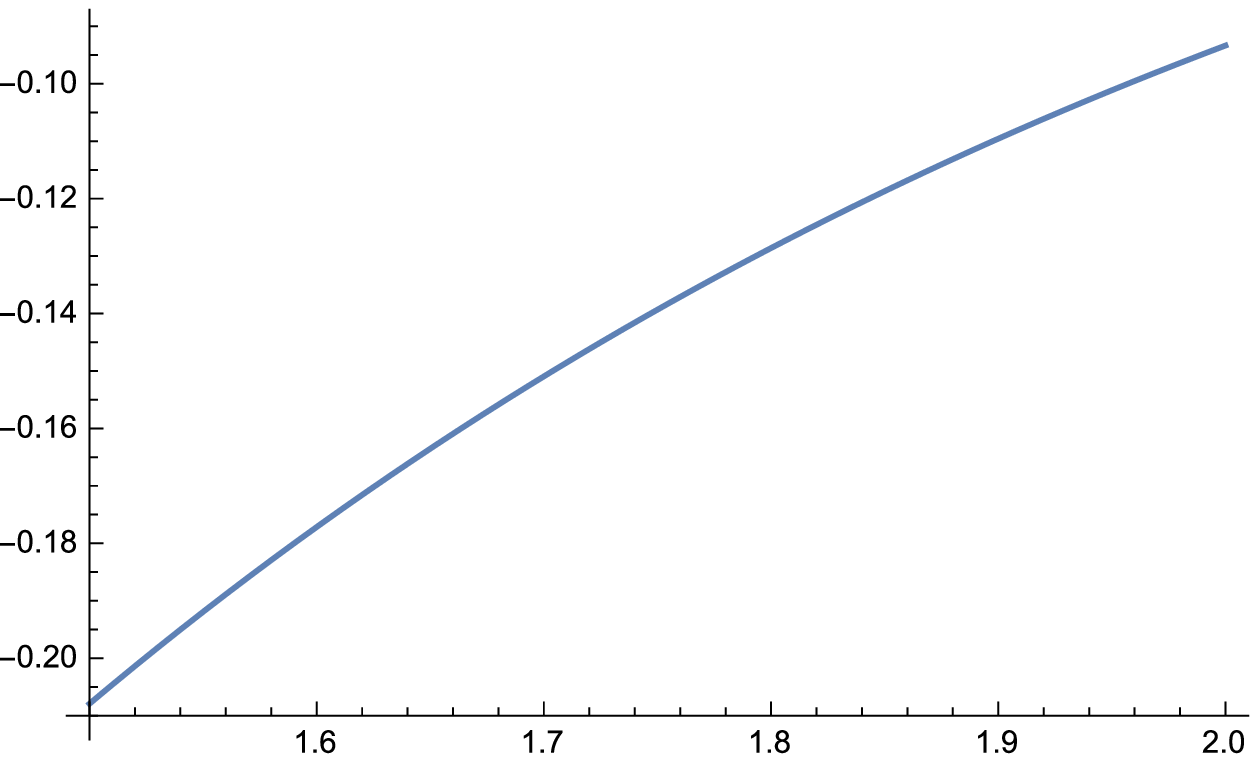}
        \caption{Sp}
        \label{fig:sp44det}
    \end{subfigure}
    ~ 
    \caption{Determinants of the matrices described by \eqref{EquationsToReduceProblem}}\label{fig:44dets}
\end{figure}

Therefore, the equations \eqref{EquationsToReduceProblem} specify unique values of $c_1,c_3,c_4$, and $c_6$, which of depend only on $\sigma$. \\

Again, the coefficients are unwieldy. So is the new infimum/bound on average rank they produce.  Though they will be available on an arXiv version of this paper, we will only reproduce plots of the optimal test functions for $\sigma = 1.7$ and plots of the infimum compared to the na\"{i}ve (though in fact quite close to optimal!) estimate in \cite{ILS}. \\

The optimal $g$ for $\sigma = 1.7$ are 

\begin{figure}[h]
    \centering
    \begin{subfigure}[b]{0.3\textwidth}
        \includegraphics[width=\textwidth]{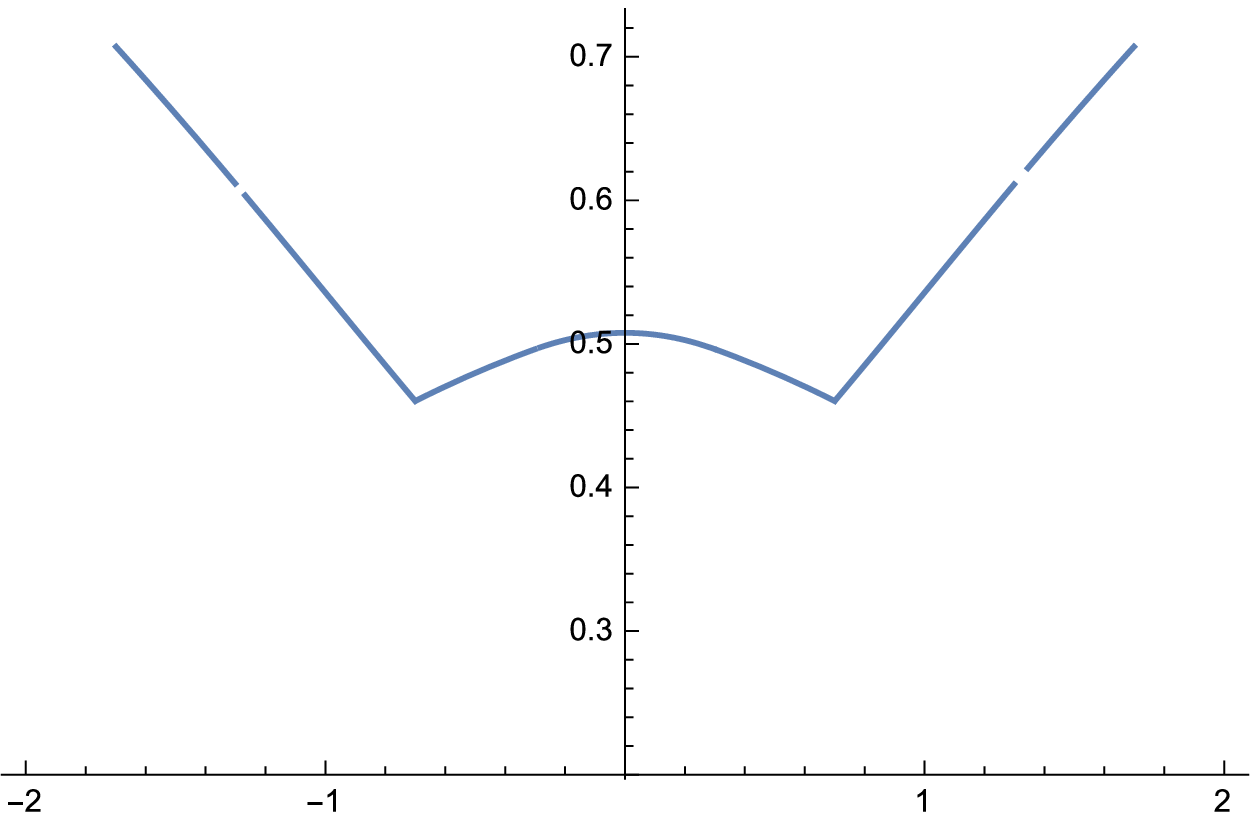}
        \caption{SO(Even)}
        \label{fig:soevenoptg1.7}
    \end{subfigure} \qquad  
    ~ 
    \begin{subfigure}[b]{0.3\textwidth}
        \includegraphics[width=\textwidth]{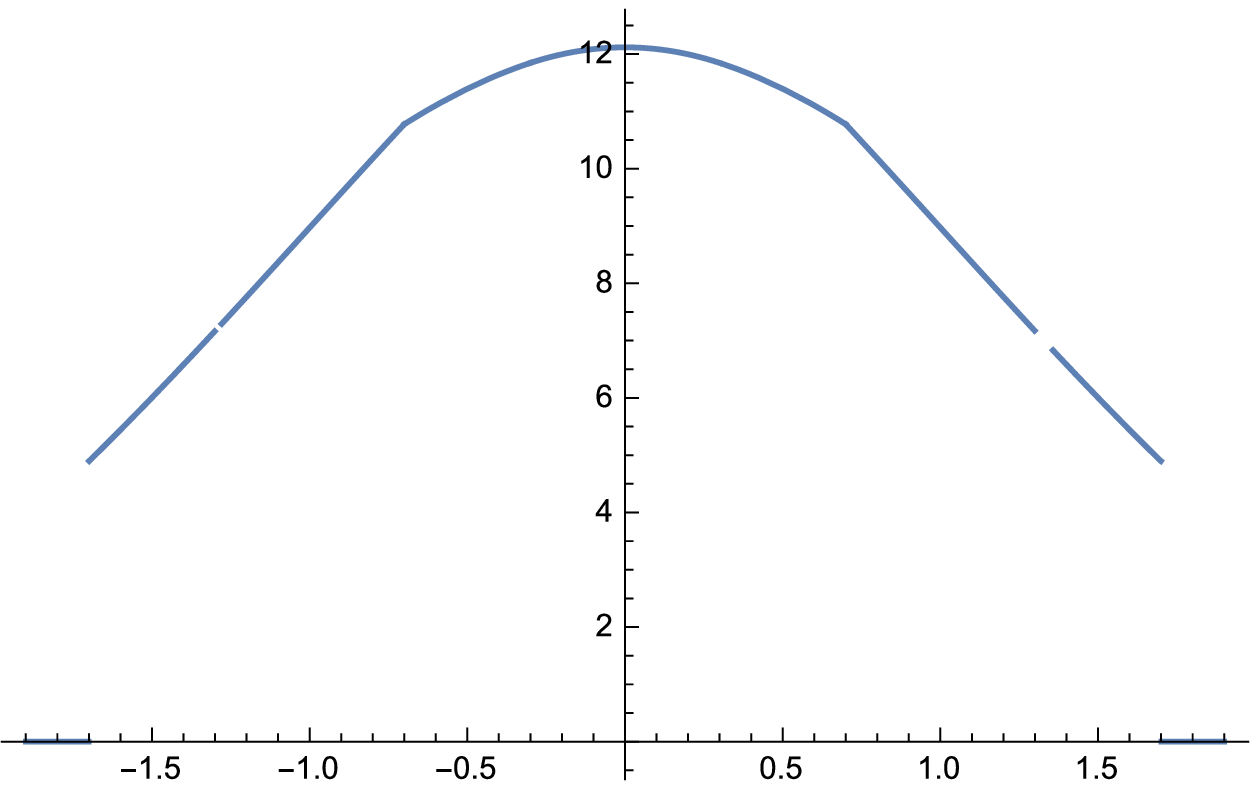}
        \caption{Sp}
        \label{fig:spoptg1.7}
    \end{subfigure} \qquad 
    \begin{subfigure}[b]{0.3\textwidth}
        \includegraphics[width=\textwidth]{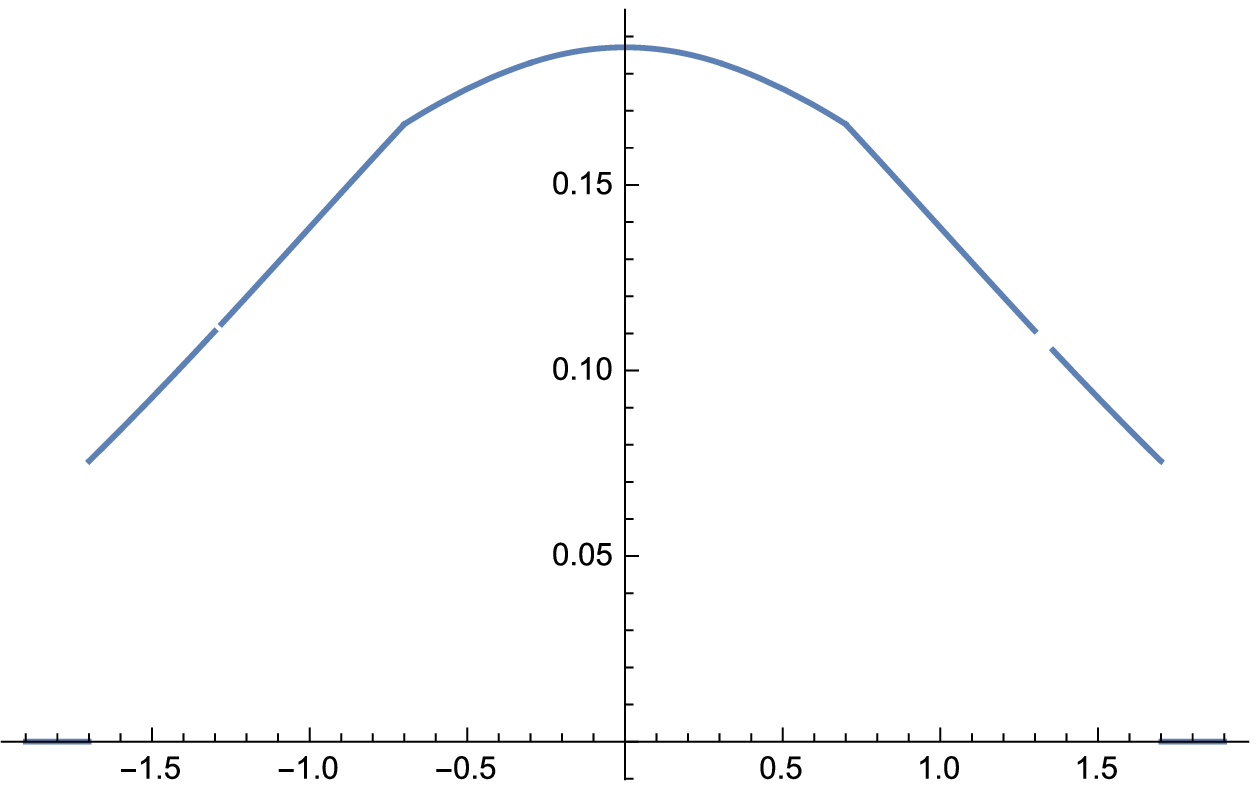}
        \caption{SO(Odd)}
        \label{fig:sooddoptg1.7}
    \end{subfigure}
    ~ 
    \caption{Optimal Test Functions for $\sigma = 1.7$}\label{fig:44optfunctions}
\end{figure}

and the associated infima, compared to the estimates in \cite{ILS}, are 

\begin{figure}[H]
    \centering
    \begin{subfigure}[b]{0.3\textwidth}
        \includegraphics[width=\textwidth]{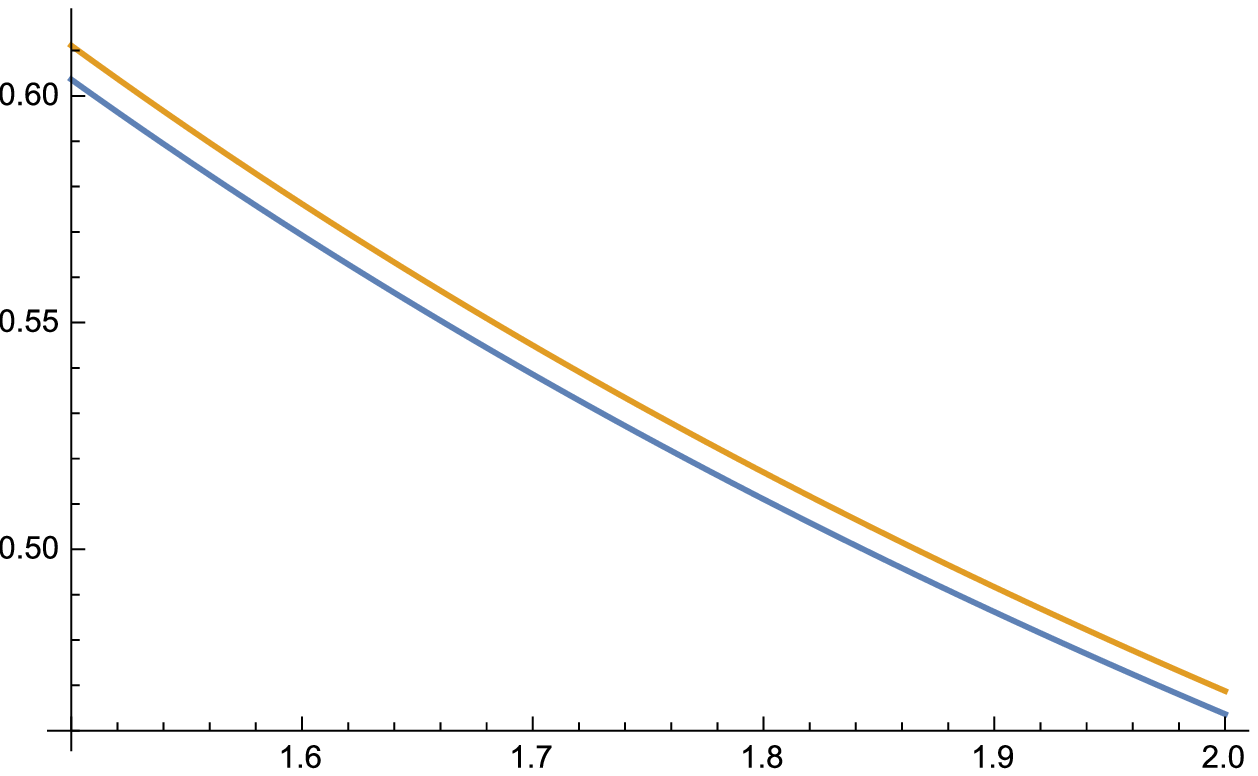}
        \caption{SO(Even)}
	\label{fig:soevencomparison44}
    \end{subfigure} \qquad  
    ~ 
    \begin{subfigure}[b]{0.3\textwidth}
        \includegraphics[width=\textwidth]{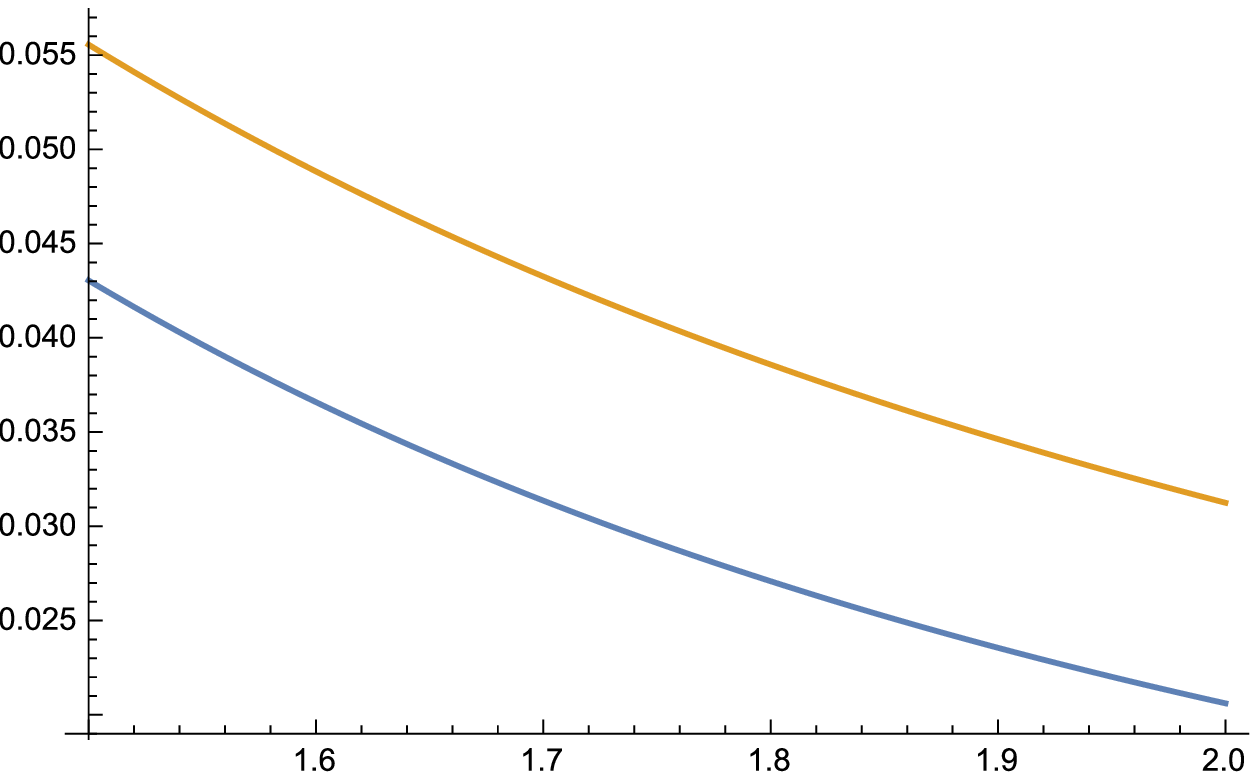}
        \caption{Sp}
        \label{fig:spcomparison-44}
    \end{subfigure} \qquad 
    \begin{subfigure}[b]{0.3\textwidth}
        \includegraphics[width=\textwidth]{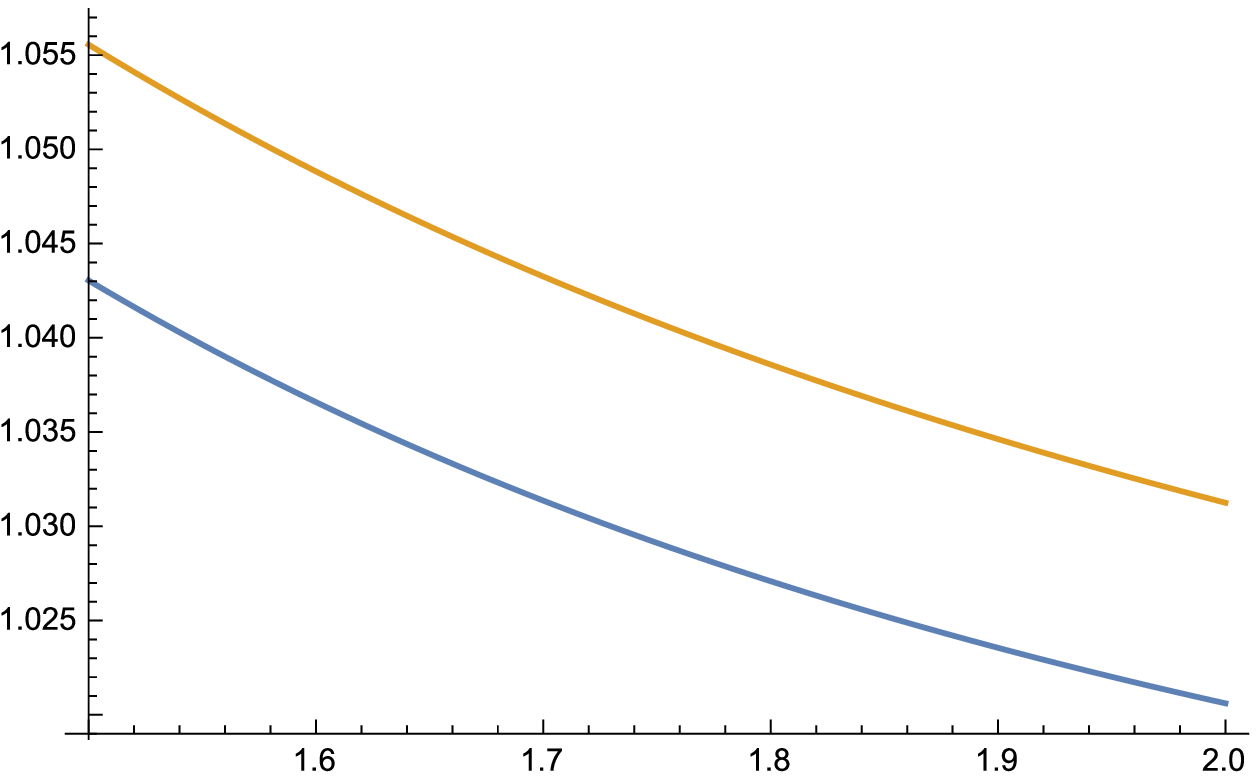}
        \caption{SO(Odd)}
        \label{fig:sooddcomparison-44}
    \end{subfigure}
    ~ 
    \caption{Infima of the functional \ref{OperatorForm}, compared to the na\"{i}ve test functions from \cite{ILS}}. \label{fig:44infima}
\end{figure}

\newpage

\section{Asymptotics For the One-Level Infimum}
For each group, $\mathcal{G}$ and each $\sigma$, we may define 
\begin{equation}
  \mathfrak{I}_{\mathcal{G}}(\sigma) \ := \ \inf_{\phi} \frac{\int_{-\infty}^{\infty} \phi(x) W_{\mathcal{G}}(x) \ dx}{\phi(0)}
  \label{InfFunctionDef}
\end{equation}
where we require $\textrm{support}(\hat{\phi}) \subset [-2\sigma,2\sigma]$. So, $\mathfrak{I}(\sigma): \R \longrightarrow \R$. This section is devoted to proving facts about $\mathfrak{I}_{\mathcal{G}}(\sigma)$. A first result about these asymptotics is proven in \cite{ILS}. They use the na\"{i}ve Fourier pair \eqref{NaiveFourierPair}, given by 
\[ \phi(x) = \left( \frac{\sin(2\sigma \pi x)}{2 \sigma \pi x} \right)^{2} \quad \quad  \widehat{\phi}(y) = \frac{1}{2\sigma}\left( 1 - \frac{|y|}{2\sigma} \right) \ \quad \textrm{if} \ |y| < 2 \sigma
\] 
to explicitly compute a value of 
\[ \frac{\int_{-\infty}^{\infty} \phi(x)W(x) dx}{\phi(0)}. \] 
This then provides an upper bound on $\mathfrak{I}_{\mathcal{G}}(\sigma)$. The result they obtain is: 
\begin{align}
  \mathfrak{I}_{\mathcal{G}}(s) \ &\le \ \frac{1}{s} + \frac{1}{2} \ &\mathcal{G} = O \label{OrthogAsympBound} \\
  \mathfrak{I}_{\mathcal{G}}(s) \ &\le \ 
  \begin{cases}
    \frac{1}{s} + \frac{1}{2} \ & s \le 1 \\
    \frac{2}{s} - \frac{1}{(s)^{2}} \ & s \ge 1
  \end{cases}
  &\mathcal{G} = \textrm{SO(Even)} \label{SOEvenAsympBound} \\
  \mathfrak{I}_{\mathcal{G}}(s) \ &\le \
  \begin{cases}
    \frac{1}{s} + \frac{1}{2} \ & s \le 1 \\
    1 + \frac{1}{(s)^{2}} \ & s \ge 1 
  \end{cases} 
  &\mathcal{G} = \textrm{SO(Odd)} \label{SOOddAsmypBound}\\
  \mathfrak{I}_{\mathcal{G}}(s) \ &\le \ 
  \begin{cases}
    \frac{1}{s} - \frac{1}{2} \ & s \le 1 \\
    \frac{1}{(s)^{2}} \ & s \ge 1 
  \end{cases} 
  &\mathcal{G} = \textrm{Sp}
  \label{SpAsympBound}
\end{align}
where $s:= 2\sigma$. \\

In each case, there is also a lower bound on $\lim_{s \to \infty} \mathfrak{I}_{\mathcal{G}}(\sigma)$. 

\begin{proposition}
	We have 
	\begin{equation}
		\lim_{\sigma \to \infty} \mathfrak{I}_{\mathcal{G}}(\sigma) \ = \ 
		\begin{cases}
			1/2 & \mathcal{G} = O \\
			0 & \mathcal{G} = \textrm{SO(Even), Sp} \\
			1 & \mathcal{G} = \textrm{SO(Odd)}
		\end{cases}
		\label{AsymptoticBoundsEqn}
	\end{equation}
	\label{AsympBoundsProp}
\end{proposition}

\begin{proof}
	For the case $\mathcal{G} = O$, our work in Section \ref{orthogsec} shows that in fact $\mathfrak{I}_{\mathcal{G}}(\sigma) = 1/2 + 1/(2 \sigma)$. When $\mathcal{G} = \textrm{SO(Even)}$ or $\textrm{Sp}$, note that the infimum is bounded below by zero, so the upper bounds in \eqref{SOEvenAsympBound} and \eqref{SpAsympBound} drive the real infimum to zero as $\sigma$ approaches infinity. \\

	To prove the final claim about the $\textrm{SO(Odd)}$ infimum, it suffices to show $\mathfrak{I}_{\textrm{SO(Odd)}} = \mathfrak{I}_{\textrm{Sp}} + 1$. \\

	Suppose $g_{\textrm{Sp},\sigma}$ is optimal for $\textrm{Sp}$ and $\sigma$. Then, $(I + K_{\textrm{Sp},\sigma})(g_{\textrm{Sp},\sigma}) = 1$. Note that $I + K_{\textrm{SO(Odd)},\sigma} = I + K_{\textrm{Sp},\sigma} + 1$ and so 
	\[ 
		(I + K_{SO(Odd),\sigma})(g_{\textrm{Sp},\sigma}) \ = \ 1 + \innerproduct{1,g_{\textrm{Sp},\sigma}} 
	\]
	and hence 
	\begin{equation}
		g_{\textrm{SO(Odd)},\sigma} \ = \ \frac{g_{\textrm{Sp},\sigma}}{1 + \innerproduct{1,g_{\textrm{Sp},\sigma}}}
		\label{RelationBetweenOptSpAndOptSOOdd}
	\end{equation}
	which implies
	\begin{align*}
		\mathfrak{I}_{\textrm{SO(Odd)}}(\sigma) \ &= \ \frac{1}{\innerproduct{1,g_{\textrm{SO(Odd)},\sigma}}} \\
		\ &= \ \frac{1 + \innerproduct{1,g_{\textrm{Sp},\sigma}}}{\innerproduct{1,g_{\textrm{Sp},\sigma}}} \\
		\ &= \ \mathfrak{I}_{\textrm{Sp}}(\sigma) + 1.
	\end{align*}
	Note that the inner product $\innerproduct{1,g_{\mathcal{G},\sigma}}$ is never $-1$, due to the bounds in \eqref{SpAsympBound}, so we have not divided by zero in the above manipulations. 
\end{proof}

Although we have upper bounds and asymptotics for $\mathfrak{I}_{\mathcal{G}}(\sigma)$, it does not give us the continuity or smoothness results we might expect from such a function. To establish those, we begin with a na\"{i}ve result. \\

\begin{lemma}
  For $\mathcal{G} = \textrm{SO(Even)}, \textrm{SO(Odd)}, Sp$, or $O$, $\mathfrak{I}_{\mathcal{G}}(\sigma)$ is non-increasing in $\sigma$. 
  \label{NonIncreasingLemma}
\end{lemma}
\begin{proof}
  Let $\sigma_{2} \ge \sigma_{1}$ and let $\phi_{\mathcal{G}, \sigma_{1}}$ be the optimal function for the pair $(\mathcal{G}, \sigma_{1})$. Note that $\phi_{\mathcal{G},\sigma_{1}}$ is an admissible function for the pair $(\mathcal{G},\sigma_{2})$. As $\mathfrak{I}_{\mathcal{G}}(\sigma)$ is defined as an infimum, we have 
  \begin{equation}
    \mathfrak{I}_{\mathcal{G}}(\sigma_{2}) \ \le \ \frac{\int_{-\infty}^{\infty} \phi_{\mathcal{G}, \sigma_{1}}(x) W_{\mathcal{G}}(x) \ dx}{\phi_{\mathcal{G},\sigma_{1}}(0)} \ = \ \mathfrak{I}_{\mathcal{G}}(\sigma_{1}). 
    \label{NaiveEqnForNonInc}
  \end{equation}
\end{proof}

We may apply Theorem \ref{GOnWholeIntervalThm} for a short proof that $\mathfrak{I}_{\mathcal{G}}(\sigma)$ is in fact strictly decreasing in $\sigma$. 

\begin{theorem}
  For $\mathcal{G} = \textrm{SO(Even)}, \textrm{SO(Odd)}, Sp$, or $O$, $\mathfrak{I}_{\mathcal{G}}(\sigma)$ is strictly decreasing in $\sigma$. 
  \label{StrictDecrease}
\end{theorem}
\begin{proof}
  We show that for any of the aforementioned $\mathcal{G}$ and for $k < \sigma_{1} < \sigma_{2} < k + 1/2$  or $k + 1/2 < \sigma_{1} < \sigma_{2} < k + 1$, that the optimal $g$ corresponding to $\sigma_1$ and the optimal $g$ corresponding to $\sigma_2$ are different. Combined with Lemma \ref{NonIncreasingLemma}, this proves our desired result. \\

  First, we make a few observations. By Lemma \ref{OptimalCriterionAndComputationLemma}, we know that for each pair $(\mathcal{G},\sigma)$, there is a unique $g_{\mathcal{G}, \sigma}$ realizing $\mathfrak{I}_{\mathcal{G}}(\sigma)$ for each $g$ and $\sigma$. So, it suffices to show that if $\sigma_{2} > \sigma_{1}$, that $g_{\mathcal{G},\sigma_{1}}$ is not optimal. \\

  There are two cases, $k < \sigma_{1} < \sigma_{2} < k + 1/2$  and $k + 1/2 < \sigma_{1} < \sigma_{2} < k + 1$. We examine the first case, using a proof by contradiction. The second case is identical. Suppose $g_{\mathcal{G},\sigma_{1}} = g_{\mathcal{G},\sigma_{2}}$ (that $g_{\mathcal{G}, \sigma_{1}}$ is again optimal for support $\sigma_{2}$). Note that $g_{\mathcal{G},\sigma_{1}}$ vanishes on $(\sigma_{1}, \sigma_{2})$. But, by Theorem \ref{GOnWholeIntervalThm}, $g_{\mathcal{G},\sigma_{2}}$ is real analytic on both 
  \[ \mathring{I}_{0, \sigma_{2}} := (\sigma_{2} - 2(\sigma_{2} - k), \sigma_{2}) \] 
  and 
  \[ \mathring{J}_{0,\sigma_{2}} := (\sigma_{2} - 1, \sigma_{2} - 2(\sigma_{2} - k)). \] 

  We know that $(\sigma_{1}, \sigma_{2}) \cap \mathring{I_{0, \sigma_{2}}}$ is nonempty. It follows from real-analyticity and the Identity Theorem that $g_{\mathcal{G}, \sigma_{2}} = g_{\mathcal{G}, \sigma_{1}}$ must vanish identically on $\mathring{I}_{0, \sigma_{2}}$ and also on $\mathring{I}_{0,\sigma_{1}}$. \\

  Note that $g_{\mathcal{G}, \sigma_{1}}$ is also  real analytic on 
  \[ \mathring{I}_{0, \sigma_{1}} := (\sigma_{1} - 2(\sigma_{1} - k), \sigma_{1}) \] 
  and 
  \[ \mathring{J}_{0,\sigma_{1}} := (\sigma_{1} - 1, \sigma_{1} - 2(\sigma_{1} - k)). \] 

  Figure \ref{fig:IntervalOverlap} shows the intervals in discussion and crucially the overlap between them. 
  \begin{figure}[H]
    \centering
    \includegraphics[scale=0.5]{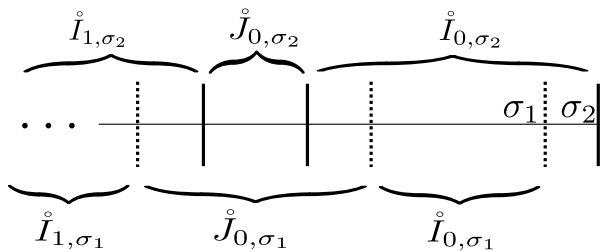}
    \caption{The four outermost intervals and the overlap between them}
    \label{fig:IntervalOverlap}
  \end{figure}

  But, $\sigma_{1} < \sigma_{2}$, so $\sigma_{2} - 2(\sigma_{2} - k) < \sigma_{1} - 2(\sigma_{1} - k)$. Consequently, $\mathring{J}_{0,\sigma_{1}} \cap \mathring{I}_{0,\sigma_{2}}$ is nonempty and open. Therefore, $g_{\mathcal{G}, \sigma_{2}} = g_{\mathcal{G}, \sigma_{1}}$ also vanishes identically on $\mathring{J}_{0, \sigma_{1}}$. It follows from the system of delay differential equations \eqref{first1} -- \eqref{first2k1}, \eqref{second1a} -- \eqref{second2ka} that $g_{\mathcal{G}, \sigma_{1}}$ is identically zero. This clearly contradicts our assumption that $g_{\mathcal{G}, \sigma_{1}}$ was optimal for the pair $(\mathcal{G}, \sigma_{1})$, since the optimal function cannot be identically zero. 
  \end{proof}

The following corollary is especially reassuring in our search for asymptotic behavior in the one-level case.   

\begin{corollary}
  For $\mathcal{G} = \textrm{SO(Even), SO(Odd)}, Sp,$ or $O$, the following facts hold for $\mathfrak{I}_{\mathcal{G}}(\sigma): (0, \infty) \longrightarrow \R$. 
  \begin{enumerate}[(a)]
    \item \label{resulta} The function $\mathfrak{I}_{\mathcal{G}}(\sigma)$ is continuous, excepting at most countably many points. The only discontinuities can be jump discontinuities. 
    \item \label{resultb} The function $\mathfrak{I}_{\mathcal{G}}(s)$ is differentiable almost everywhere. 
    \item \label{resultc} The limit
\begin{equation}
    \lim_{\sigma \to \infty}\mathfrak{I}_{\mathcal{G}}(\sigma)
    \label{LimEqn}
  \end{equation}
  exists and is real-valued for all $\mathcal{G}$. 
\item \label{resultd} If $T_{\mathcal{G}}$ is the range of $\mathfrak{I}_{\mathcal{G}}(\sigma)$, then $\mathfrak{I}_{\mathcal{G}}(\sigma)$ has an inverse on $T_{\mathcal{G}}$. 
  \end{enumerate}
  \end{corollary}
  \begin{proof} In \cite{B}, one may find proofs that \eqref{resulta}, \eqref{resultb}, and \eqref{resultd} hold for decreasing functions from $\R$ to $\R$. Result \eqref{resultc} holds as well, but the limit is allowed to b either $\pm \infty$.  \\

  Results \eqref{resulta} and \eqref{resultb} from above are standard for functions $f: \R \longrightarrow \R$. To show that they hold for our function, consider $\mathfrak{I}(e^{x}): \R \longrightarrow \R$. Since $d_{x}(e^{x}) = e^{x}$ never vanishes, wherever $\mathfrak{I}(e^{x})$ exists, the derivative of $\mathfrak{I}$ exists as well. \\

  For result \eqref{resultc}, note that we require $\phi \ge 0$. So, the limit cannot be $-\infty$. Since $\mathfrak{I}(\sigma)$ is decreasing and exists for all $\sigma$, the limit is not $+\infty$. \\

  Result \eqref{resultd} does not need to be modified from the standard result. 
\end{proof}

\newpage

\section{Future Works}

There are two clear directions for future work: remaining analysis of the one-level and analysis of optimal functions for the $n$-level, for $n > 1$. We discuss each in turn. 
\subsection{The One-Level} Due to restrictions on the number theory side of density computations, it is not as pressing to know all of the optimal test functions for all $\sigma$. However, our new method raises a few questions we feel should be addressed by future work. \\

First, in typical computations involving delay differential equations, one knows the function on an interval the size of the shift. For example, for the standard delay differential equation with a single delay given by 
\begin{equation}
  \frac{d}{dt} x(t) = f(x(t), x(t - \tau))
  \label{StandardDDE}
\end{equation}
we are given an initial condition such as $\phi: [-\tau,0] \longrightarrow \R^{n}$. It can then be shown that the delay differential equation is equivalent to a homogenous initial value problem which can be solved via successive iteration \cite{BC}. \\

In this work we solve a system of location-specific delay-differential equations without any such initial condition or ``history'' of the solution. As this computation resembles solving a system of equations without the use of matrices, we are lead to ask if there is a general criterion for when such systems are solvable. \\

Second, after resolving this system of equations, we are left with a finite-dimensional optimization that leads to the optimal $\phi$. We are able to solve two of these optimization problems ($1 < \sigma < 1.5$, and $1.5 < \sigma < 2$) using matrices. However, it is not obvious whether this method is necessary or sufficient. We seek a general solution to this finite-dimensional optimization problem.  \\

Finally, there is much more to learn about the infimum function, $\mathfrak{I}_{\mathcal{G}}(\sigma)$. While we are able to show that the function is strictly decreasing, we suspect the following statement also holds.  
\begin{conjecture}
  The function $\mathfrak{I}_{\mathcal{G}}(\sigma)$ is continuous in $\sigma$ and is real-analytic except at integers and half-integers.
\end{conjecture}



\subsection{Higher Level Densities} Finding optimal test functions for higher level densities seems like an especially ambitious project. For the $m$-level density, the weight functions are given by  
\begin{align}
  W_{m,\varepsilon}(x) \ &= \ \textbf{det}\left( K_{\varepsilon}(x_{i},x_{j}) \right)_{i,j \le m} \notag \\
  W_{m,O^{+}}(x) \ &= \ \textbf{det}\left( K_{1}(x_{i},x_{j}) \right)_{i,j \le m} \notag \\
  W_{m,O^{-}}(x) \ &= \ \textbf{det}\left( K_{-1}(x_{i},x_{j}) \right)_{i,j \le m} + \sum_{k=1}^{m} \delta(x_{k})\textbf{det}\left( K_{-1}(x_{i},x_{j}) \right)_{i,j \not = k}\notag \\
  W_{m,O}(x) \ &= \ \frac{1}{2} W_{m,O^{+}}(x) + \frac{1}{2}W_{m,O^{-}} \notag \\
  W_{m,U}(x) \ &= \ \textbf{det}\left( K_{0}(x_{i},x_{j}) \right)_{i,j \le m} \notag \\
W_{m,Sp}(x) \ &= \ \textbf{det}\left( K_{-1}(x_{i},x_{j}) \right)_{i,j \le m}
  \label{mLevelDensities}
\end{align}
where $K(y) = \frac{\sin \pi y}{\pi y}, K_{\varepsilon}(x,y) = K(x - y) + \varepsilon K(x + y)$, for $\varepsilon = 0, \pm 1$, $O^{+}$ denotes the group $\textrm{SO(Even)}$ and $O^{-}$ the group $\textrm{SO(Odd)}$ \cite{KS, Mil1}. \\

For $m > 1$, $\widehat{W}$ becomes substantially more complicated. It is not clear whether any of the methods developed in this paper will help discover optimal test functions for higher level density. However, higher level density calculations are quite important (see \cite{Mil1}), so we feel this is an valuable area for future research.   

\newpage

\begin{appendices}
\section{Plotting Approximately Optimal $g$}
\label{approxoptappendix}
  We use the following code to obtain the shape of the optimal $\phi$. This code is written in Mathematica. The code is explained in the comments: 
  \begin{verbatim}
  
(* Here, our fredholm equation is of the form 
\[ \int_{-\sigma}^{\sigma} K(x,y)g(y) dy - \mu g(x) = f(x)\] 
The notation used to name variables will correspond to this *)
(* \
Ultimately, in this approximation, we are solving the matrix equation \
gvect_{1 \times n + 1}BigMat_{n+1\times n+1} = fvect_{1\times n + 1}. \
This gives a discrete set of values for g. We use row vectors for \
conveniece *)

sigma = Set_This _Yourself; (*Half the support of the Fourier \
Transform of phi - the support of the function g so that g \star \
\check{g} = \hat{\phi} *)

CoordFunction[m_] := -sigma + 2 (sigma (m - 1))/n
k[x_, y_] := 
  If[Abs[x - y] <= 1, .5, 
   0]; (*The kernel for SO(Even) *)
(*k[x_,y_]:= If[Abs[x-y]\
\[LessEqual] 1, .5,1]; *)(*The kernel for SO(Odd) *)
(*k[x_,y_]:= \
If[Abs[x-y]\[LessEqual] 1, -.5,0];*) (*The kernel for Symplectic *)

mu = -1; (* A parameter in a general Fredholm type 2 Equation, which \
in the ILS case is -1*)

n = Set_This _Yourself; (* This is the number of partitions of the \
interval [-\sigma,\sigma]. The partition is regular.*)

CoordFunction[m_] := -sigma + 
   2 (sigma (m - 1))/
    n ;(*This function of integers just gives the endpoints of the \
partition - it makes things a little easier to only write it once *) 
fMatrixBase[r_, s_] := 
  If[r == 1 || r == n + 1, 
   sigma/n (k[CoordFunction[s], CoordFunction[r]] ), 
   sigma/n (2 k[CoordFunction[s], CoordFunction[r]] )]; 
(* We use two functions to create our matrix since we need to make a \
diagonal adjustment. This seemed like the neatest way to do it. Note \
that our vectors are row vectors. While we would  *)

fMatrix[r_, s_] := 
  If[r == s, fMatrixBase[r, s] - mu, fMatrixBase[r, s]];
(* We need to make an adjustment along the diagonal because of the \
identity operator in the Fredholm equation *) 
f[x_] := 1; 
(* Recall $f$ is the function on the righthand side of the Fredholm \
equation *)
fvect = Table[f[m], {m, 1, n + 1}];
BigMat = Transpose[Array[fMatrix, {n + 1, n + 1}]];
gvect = fvect.Inverse[BigMat];
xvect = Table[CoordFunction[k], {k, 1, n + 1}]; 
(* Just a vector of x-coordinates to plot against the $g$-values*)
\
ListPlot[Transpose[{xvect, gvect}]]
  \end{verbatim}

\newpage

\section{The case $1 < \sigma < 1.5$}
  Below is the notebook that computes costants and generates plots for the case $1 < \sigma < 1.5$.
  \begin{verbatim}
  (* For SO(Even) *)
(*g[x_]= Cos[x /2- (Pi + 1)/4]; *)
(* For \
Sp(Odd)/Sp *)
g[x_] = Cos[x/2 + (Pi - 1)/4];
(*The matrix for SO(Even) *)
(*M= \
{{Cos[(s-1)/Sqrt[2]],Sin[(s-1)/Sqrt[2]], 0}, {Cos[(s-1)/Sqrt[2]], 0, \
0}, {1/Sqrt[2]Sin[(s-1)/Sqrt[2]] + Cos[(s-1)/Sqrt[2]], Sqrt[2] - \
1/Sqrt[2]Cos[(s-1)/Sqrt[2]],-1}}; *)
(* The Matrix for SO(Odd)/Sp *)

M = {{Cos[(s - 1)/Sqrt[2]], Sin[(s - 1)/Sqrt[2]], 
    0}, {Cos[(s - 1)/Sqrt[2]], 0, 
    0}, {-1/Sqrt[2] Sin[(s - 1)/Sqrt[2]] + 
     Cos[(s - 1)/Sqrt[2]], -Sqrt[2] + 
     1/Sqrt[2] Cos[(s - 1)/Sqrt[2]], -1}};
V = {{g[s - 1]}, {g[s - 1]}, {g[s - 1] - g[2 - s]}};
Simplify[Inverse[M].V]



c11[s_] = Sec[(-1 + s)/Sqrt[2]] Sin[1/4 (3 + 3 \[Pi] - 2 s)];
c31[s_] = 
  Simplify[Sin[1/4 (-3 + 3 \[Pi] + 2 s)] + (
    Sin[1/4 (3 + 3 \[Pi] - 2 s)] Tan[(-1 + s)/Sqrt[2]])/Sqrt[2]];
c12[s_] = Sec[(-1 + s)/Sqrt[2]] Sin[1/4 (3 + \[Pi] - 2 s)];
c32[s_] = 
  Simplify [
   Sin[1/4 (-3 + \[Pi] + 2 s)] - (
    Sin[1/4 (3 + \[Pi] - 2 s)] Tan[(-1 + s)/Sqrt[2]])/Sqrt[2] ];

(* g1 is for SO(Even), g2 is for the other groups*)

lambda1[s_] := 
  c11[s] + Integrate[g1[x, s], {x, 0, 1}, 
    Assumptions -> Element[s, Reals] && 1 < s < 1.5];
(*FullSimplify[lambda1[s]];*)

g1[x_, s_] := 
  Piecewise[{{0, Abs[x] > s}, {c11[s] Cos[Abs[x]/Sqrt[2] ], 
     Abs[x] <= s - 1}, {Cos[1/2 Abs[x ] - (Pi + 1)/4], 
     s - 1 < Abs[x] < 
      2 - s}, {c11[s]/Sqrt[2] Sin[(Abs[x] - 1)/Sqrt[2]] + c31[s], 
     2 - s <= Abs[s] <= s}}];
scaledg1[x_, s_] = (1/lambda1[s]) g1[x, s];
(*FullSimplify[1/Integrate[scaledg1[x,s],{x,-s,s}, Assumptions \
\[Rule] Element[s,Reals] && 1 < s < 1.5]]*)
\
(*Plot[scaledg1[x,1.2],{x,-1.2,1.2}] *)
(* This will generate a plot \
of the actual phi function. *)
\
(*(InverseFourierTransform[scaledg1[x,1.2],x,t])^2 *)
\
(*DiscretePlot[(InverseFourierTransform[scaledg1[x,1.2], \
x,t])^2,{t,lb,ub,stepsize}]*)

$Aborted

phisoeven[
   t_] = ((0.09294262051124703) (t (-0.17677669529663692` + 
          0.35355339059327384` t + 0.35355339059327384` t^2 - 
          0.7071067811865477` t^3) Cos[0.15` + 0.8` t] + 
       0.17677669529663692` t Sin[0.15` + 0.8` t] - 
       0.35355339059327384` t^2 Sin[0.15` + 0.8` t] - 
       0.35355339059327384` t^3 Sin[0.15` + 0.8` t] + 
       0.7071067811865477` t^4 Sin[0.15` + 0.8` t] - 
       0.25` t Sin[0.6353981633974484` + 0.8` t] - 
       0.5` t^2 Sin[0.6353981633974484` + 0.8` t] + 
       0.5` t^3 Sin[0.6353981633974484` + 0.8` t] + 
       1.` t^4 Sin[0.6353981633974484` + 0.8` t] + 
       0.29674900870488613` t^2 Sin[0.2` t] + 
       1.8552402359675454`*^-16 t^4 Sin[0.2` t] + 
       0.4322920546658651` Cos[t] Sin[0.2` t] - 
       2.5937523279951913` t^2 Cos[t] Sin[0.2` t] + 
       3.458336437326921` t^4 Cos[t] Sin[0.2` t] + 
       0.29674900870488613` t Sin[0.2` t] Sin[t] - 
       1.1869960348195445` t^3 Sin[0.2` t] Sin[t] + 
       t Cos[0.2` t] (0.4322920546658654` - 0.9243327675738177` t^2 - 
          0.05974865824208703` t Sin[t] + 
          0.2389946329683481` t^3 Sin[t]))^2)/(0.125` t - 0.75` t^3 + 
     1.` t^5)^2;
Plot[phisoeven[t], {t, -5, 5}, PlotRange -> All]
(* The fourier transform *)
(*FourierTransform[phisoeven[t],t,x]*)



(* This one is for Sp *)
Clear[lambda3, scaledg3]
g3[x_, s_] := 
  Piecewise[{{0, Abs[x] > s}, {c12[s] Cos[Abs[x]/Sqrt[2] ], 
     Abs[x] <= s - 1}, {Cos[Abs[x]/2 + (Pi - 1)/4], 
     s - 1 < Abs[x] < 
      2 - s}, {-c12[s]/Sqrt[2] Sin[(Abs[x] - 1)/Sqrt[2]] + c32[s], 
     2 - s <= Abs[s] <= s}, {0, Abs[x] > s}}];
lambda3[s_] := 
  c12[s] - Integrate[g3[x, s], {x, 0, 1}, 
    Assumptions -> Element[s, Reals] && 1 < s < 1.5];
FullSimplify[lambda3[s]]
scaledg3[x_, s_] = (1/lambda3[s] ) g3[x, s];
lambda3SOOdd[s_] := 
 c12[s] - Integrate[g3[x, s], {x, 0, 1}, 
   Assumptions -> Element[s, Reals] && 1 < s < 1.5] + 
  2 Integrate[g3[x, s], {x, 0, s}, 
    Assumptions -> Element[s, Reals] && 1 < s < 1.5]
Plot[scaledg3[x, 1.2] , {x, -1.2, 1.2}] 
(*InfSp[s_] = FullSimplify[1/Integrate[scaledg3[x,s],{x,-s,s}, \
Assumptions \[Rule] Element[s,Reals] && 1 < s < 1.5]]*)

g3SOOdd[x_, s_] = (1/lambda3SOOdd[s]) g3[x, s];
(*Plot[g3SOOdd[x,1.2],{x,-1.2,1.2}] *)
\
(*FullSimplify[Convolve[g3SOOdd[x,s],g3SOOdd[-x,s],x,y]]*)
\
(*InfSOOdd[s_] = FullSimplify[1/Integrate[g3SOOdd[x,s],{x,-s,s}, \
Assumptions \[Rule] Element[s,Reals] && 1 < s < 1.5]] *)
(*To \
generate a plot of the optimal phi *)

(InverseFourierTransform[g3SOOdd[x, 1.2], x, t])^2
phisp[t_] := (InverseFourierTransform[scaledg3[x, 1.2], x, t])^2
Plot[Re[phisp[t]], {t, -10, 10}, PlotRange -> All]


phisp2[t_] := 
 Re[1/(0.125` t - 0.75` t^3 + 1.` t^5)^2 (0.05526297606879339` + 
     1.8076800794049643`*^-18 I) E^((0.` - 
      1.` I) t) (t ((0.17677669529663687` + 
          0.17677669529663687` I) + (0.35355339059327373` + 
           0.35355339059327373` I) t - (0.35355339059327373` + 
           0.35355339059327373` I) t^2 - (0.7071067811865475` + 
           0.7071067811865475` I) t^3 + 
        E^((0.` + 
            1.` I) t) ((0.17677669529663687` - 
             
             0.17677669529663687` I) + (0.35355339059327373` - 
              0.35355339059327373` I) t - (0.35355339059327373` - 
              0.35355339059327373` I) t^2 - (0.7071067811865475` - 
              0.7071067811865475` I) t^3)) Sin[
       0.15` - 0.3` t] + (0.17677669529663687` - 
        0.17677669529663687` I) t Sin[
       0.15` + 0.3` t] + (0.17677669529663687` + 
        0.17677669529663687` I) E^((0.` + 1.` I) t)
       t Sin[0.15` + 0.3` t] - (0.35355339059327373` - 
        0.35355339059327373` I) t^2 Sin[
       0.15` + 0.3` t] - (0.35355339059327373` + 
        0.35355339059327373` I) E^((0.` + 1.` I) t)
       t^2 Sin[
       0.15` + 0.3` t] - (0.35355339059327373` - 
        0.35355339059327373` I) t^3 Sin[
       0.15` + 0.3` t] - (0.35355339059327373` + 
        0.35355339059327373` I) E^((0.` + 1.` I) t)
       t^3 Sin[
       0.15` + 0.3` t] + (0.7071067811865475` - 
        0.7071067811865475` I) t^4 Sin[
       0.15` + 0.3` t] + (0.7071067811865475` + 
        0.7071067811865475` I) E^((0.` + 1.` I) t)
       t^4 Sin[0.15` + 0.3` t] - 
     0.40241772554482175` E^((0.` + 0.5` I) t) t^2 Sin[0.2` t] + 
     1.609670902179287` E^((0.` + 0.5` I) t) t^4 Sin[0.2` t] + 
     0.2562367939226508` E^((0.` + 0.5` I) t) Cos[t] Sin[0.2` t] - 
     1.5374207635359047` E^((0.` + 0.5` I) t)
       t^2 Cos[t] Sin[0.2` t] + 
     2.0498943513812065` E^((0.` + 0.5` I) t)
       t^4 Cos[t] Sin[0.2` t] - 
     0.4024177255448217` E^((0.` + 0.5` I) t) t Sin[0.2` t] Sin[t] + 
     1.6096709021792868` E^((0.` + 0.5` I) t) t^3 Sin[0.2` t] Sin[t] +
      E^((0.` + 0.5` I) t)
       t Cos[0.2` t] (0.04051221478223531` - 
        0.16204885912894124` t^2 + (0.0810244295644706` t - 
           0.3240977182578824` t^3) Sin[t]))^2]
Plot[phisp2[t], {t, -5, 5}, PlotRange -> All]

(* functions are being re-defined for the sake of labels *)

SOEven[t_] := phisoeven[t];
Sp[t_] := phisp2[t];
SOOdd[t_] := phisoodd[t];
Naive [t_] := naivephi[t];
fns[t_] := {phisoeven[t], phisp2[t], phisoodd[t], naivephi[t]};
len := Length[fns[t]];
Plot[Evaluate[fns[t]], {t, -5, 5}, 
 PlotStyle -> {Normal, Dashed, Dotted, Thick}, 
 PlotLegends -> {"SO(Even)", "Sp", "SO(Odd)", "Naive"}]
  \end{verbatim}

\newpage

\end{appendices}

%

\end{document}